\documentclass[reqno,11pt]{amsart}

\usepackage[
    language		=	english,
    uniquename		=	init,
    backend			=	biber,
    maxcitenames	=	3]
{biblatex}

\usepackage{amsmath}
\usepackage{amsthm}
\usepackage{graphicx}
\usepackage{amsfonts}
\usepackage{amssymb}
\usepackage{mathtools}

\usepackage{bm}
\usepackage{color}
\usepackage{hyperref}
\usepackage{enumitem}
\usepackage[margin=1.2in]{geometry}
\usepackage{subfiles}

\numberwithin{equation}{section}

\newtheorem{theorem}{Theorem}[section]
\newtheorem{lemma}[theorem]{Lemma}

\newtheorem{proposition}[theorem]{Proposition}
\newtheorem{corollary}[theorem]{Corollary}

\theoremstyle{definition}

\newtheorem{remark}[theorem]{Remark}

\def\E{{\mathbb E}}

\def\R{{\mathbb R}}

\def\N{{\mathbb N}}
\def\PP{{\mathbb P}}

\def\P{{\mathcal P}}

\def\B{{\mathcal B}}
\def\V{{\mathcal V}}

\def\G{{\mathcal G}}

\def\F{{\mathcal F}}
\def\FF{{\mathbb F}}

\def\C{{\mathcal C}}
\def\Var{\mathrm{Var}}
\def\Cov{\mathrm{Cov}}

\DeclareMathOperator*{\esssup}{ess\,sup}

\DeclareMathOperator*{\argmin}{argmin}
\DeclareMathOperator{\sgn}{sgn}

\def\cV{\mathcal{V}}

\def\eps{\epsilon}
\def\bx{\bm{x}}
\def\bz{\bm{z}}
\def\bZ{\bm{Z}}
\def\cP{\mathcal{P}}
\def\bd{\mathbf{d}}
\def\wt{\widetilde}
\def\cL{\mathcal{L}}
\def\by{\bm{y}}
\def\ba{\bm{a}}
\newcommand{\bP}{\mathbb{P}}
\newcommand{\cU}{\mathcal{U}}

\newcommand{\cC}{\mathcal{C}}
\newcommand{\cF}{\mathcal{F}}
\newcommand{\cG}{\mathcal{G}}
\newcommand{\cN}{\mathcal{N}}
\newcommand{\bzero}{\bm{0}}
\newcommand{\bX}{\bm{X}}
\newcommand{\bY}{\bm{Y}}
\newcommand{\cA}{\mathcal{A}}
\newcommand{\bbF}{\mathbb{F}}
\newcommand{\bbG}{\mathbb{G}}

\newcommand{\cstar}{c_0}
\newcommand{\bzeta}{\boldsymbol{\zeta}}
\newcommand{\bG}{\bm{G}}

\renewcommand{\epsilon}{\varepsilon}

\DeclarePairedDelimiter\floor{\lfloor}{\rfloor}

\makeatletter
\let\save@mathaccent\mathaccent
\newcommand*\if@single[3]{%
  \setbox0\hbox{${\mathaccent"0362{#1}}^H$}%
  \setbox2\hbox{${\mathaccent"0362{\kern0pt#1}}^H$}%
  \ifdim\ht0=\ht2 #3\else #2\fi
  }
\newcommand*\rel@kern[1]{\kern#1\dimexpr\macc@kerna}
\newcommand*\widebar[1]{\@ifnextchar^{{\wide@bar{#1}{0}}}{\wide@bar{#1}{1}}}
\newcommand*\wide@bar[2]{\if@single{#1}{\wide@bar@{#1}{#2}{1}}{\wide@bar@{#1}{#2}{2}}}
\newcommand*\wide@bar@[3]{%
  \begingroup
  \def\mathaccent##1##2{%
    \let\mathaccent\save@mathaccent
    \if#32 \let\macc@nucleus\first@char \fi
    \setbox\z@\hbox{$\macc@style{\macc@nucleus}_{}$}%
    \setbox\tw@\hbox{$\macc@style{\macc@nucleus}{}_{}$}%
    \dimen@\wd\tw@
    \advance\dimen@-\wd\z@
    \divide\dimen@ 3
    \@tempdima\wd\tw@
    \advance\@tempdima-\scriptspace
    \divide\@tempdima 10
    \advance\dimen@-\@tempdima
    \ifdim\dimen@>\z@ \dimen@0pt\fi
    \rel@kern{0.6}\kern-\dimen@
    \if#31
      \overline{\rel@kern{-0.6}\kern\dimen@\macc@nucleus\rel@kern{0.4}\kern\dimen@}%
      \advance\dimen@0.4\dimexpr\macc@kerna
      \let\final@kern#2%
      \ifdim\dimen@<\z@ \let\final@kern1\fi
      \if\final@kern1 \kern-\dimen@\fi
    \else
      \overline{\rel@kern{-0.6}\kern\dimen@#1}%
    \fi
  }%
  \macc@depth\@ne
  \let\math@bgroup\@empty \let\math@egroup\macc@set@skewchar
  \mathsurround\z@ \frozen@everymath{\mathgroup\macc@group\relax}%
  \macc@set@skewchar\relax
  \let\mathaccentV\macc@nested@a
  \if#31
    \macc@nested@a\relax111{#1}%
  \else
    \def\gobble@till@marker##1\endmarker{}%
    \futurelet\first@char\gobble@till@marker#1\endmarker
    \ifcat\noexpand\first@char A\else
      \def\first@char{}%
    \fi
    \macc@nested@a\relax111{\first@char}%
  \fi
  \endgroup
}
\makeatother

\author[C.\ Fiedler]{Christian Fiedler}
\address{C.\ Fiedler,
IEOR Department, Columbia University,
\newline\hphantom{\quad \ \ C.\ Fiedler}
500 W. 120th Street, New York, NY 10027, USA
}
\email{christian.fiedler@columbia.edu}

\author[J.\ Jackson]{Joe Jackson}
\address{J.\ Jackson,
Department of Mathematics, University of Chicago,
\newline\hphantom{\quad \ \ J. Jackson}
5734 S.~University Avenue, Chicago, IL 60637, USA
}
\email{jsjackson@uchicago.edu}

\author[D.\ Lacker]{Daniel Lacker}
\address{D.\ Lacker,
IEOR Department, Columbia University,
\newline\hphantom{\quad \ \ D.\ Lacker}
500 W. 120th Street,  New York, NY 10027, USA
}
\email{daniel.lacker@columbia.edu}

\author[J.\ Niles-Weed]{Jonathan Niles-Weed}
\address{J.\ Niles-Weed,
Courant Institute for Mathematical Sciences, New York University,
\newline\hphantom{\quad \ \ J. Niles-Weed}
251 Mercer Street, New York, N.Y. 10012, USA
}
\email{jnw@cims.nyu.edu}

\thanks{J. Jackson was supported by the NSF under Grant No. DMS-2302703. J. Niles-Weed was supported by the NSF CAREER award DMS-2339829. D. Lacker and C. Fiedler were partially supported by the NSF CAREER award DMS-2045328. D. Lacker also acknowledges support from an Alfred P.\ Sloan Fellowship.}

\title{The mean-field limit of online stochastic vector balancing}
  \dedicatory{Dedicated to the memory of Robert V.\ Kohn}

\begin{document}

\begin{abstract}
We study an online vector balancing problem, in which $n$ independent Gaussian random vectors $\boldsymbol{\zeta}(1),\dots,\boldsymbol{\zeta}(n) \sim \mathcal{N}(0, I_n)$, each of dimension $n$, arrive one at a time. The goal is to choose signs $\varepsilon(1),\dots,\varepsilon(n) \in \{\pm 1\}$ with $\varepsilon(k)$ depending only on $\boldsymbol{\zeta}(1),\dots,\boldsymbol{\zeta}(k)$, so as to minimize the expected $\ell^{\infty}$ norm of the signed sum $\frac{1}{\sqrt{n}}\sum_{k = 1}^n \varepsilon(k) \boldsymbol{\zeta}(k)$. Prior work showed that the optimal value $V^n$ is $O(1)$, at least for Rademacher $\boldsymbol{\zeta}(k)$'s, by constructing specific algorithms. Our main contribution is to determine the exact limit $V^{\infty} = \lim_{n\to\infty} V^n$ as the value of a nonstandard stochastic control problem of mean-field type: find the narrowest terminal interval into which a Brownian motion can be adaptively steered under a uniform-in-time $L^2$ constraint on the drift. The proof of the lower bound $V^{\infty} \leq \liminf_{n \to \infty} V^n$ uses probabilistic compactness arguments, and is very flexible. In fact, we show that the lower bound is universal, in that it holds as long as the entries of the $\boldsymbol{\zeta}(k)$ vectors are i.i.d.\ with mean zero, variance 1, and finite fourth moment. The proof of the upper bound $\limsup_{n \to \infty} V^n \leq V^{\infty}$ is more delicate, relying on dynamic programming principles and a priori bounds obtained from a coupling procedure involving the Föllmer drift, which makes explicit use of the Gaussian structure. In addition to our main convergence result, we provide some analysis and asymptotics for the limiting mean-field control problem.
\end{abstract}

\maketitle

\tableofcontents

\section{Introduction}

\subsection{Vector balancing}
\emph{Vector balancing} (or \emph{discrepancy minimization}) is a fundamental combinatorial optimization problem, with connections to convex geometry, approximation algorithms, differential privacy, and statistical physics. Given $m$ vectors $\bzeta(1),\dots,\bzeta(m)$ in $\R^n$, the goal is to choose signs $\eps(1),\dots,\eps(m) \in \{\pm 1\}$ in order to minimize
\begin{align*}
    \Big| \sum_{k = 1}^m \eps(k) \bzeta(k) \Big|_{\infty} = \max_{i = 1,\dots,n} \Big| \sum_{k = 1}^m \eps(k) \zeta_i(k) \Big|, 
\end{align*}
where each $\bzeta(k)$ is written in coordinates as $\bzeta(k) = (\zeta_1(k),\dots,\zeta_n(k))$, and $|\cdot|_{\infty}$ denotes the $\ell^{\infty}$ norm of a vector. The majority of prior work on this question focuses on the setting when $n$ is large and $m$ is proportional to $n$, where the asymptotics for the problem are the most delicate. In this paper, we focus for simplicity on the case $n = m$ but, as discussed below, the arguments easily generalize to the case that $m = m(n)$ satisfies $m(n)/n \to T \in (0,\infty)$.

There are several variants of the vector balancing problem, depending on how the vectors $\bzeta(k)$ and the signs $\eps(k)$ are chosen. For example, in the \textit{worst-case} version of the problem, $\bzeta(k)$ is an arbitrary vector lying in some set (e.g., $[-1,1]^n$), while in the \textit{random} version, the vectors $\bzeta(k)$ are drawn from some fixed distribution. Meanwhile, the term \textit{offline} indicates that $\eps(1),\dots,\eps(n)$ can be chosen arbitrarily after viewing all $n$ vectors $\bzeta(1),\dots,\bzeta(n)$, while \textit{online} means that $\eps(k)$ must be chosen after viewing only $\bzeta(1),\dots,\bzeta(k)$.

The classical theory of vector balancing goes back to the seminal results of Spencer~\cite{Spencer1977} (and, independently, Gluskin~\cite{Gluskin1989}). They analyzed the \textit{offline}, \textit{worst-case} version of the problem, in which the vectors $\bzeta(1), \dots, \bzeta(n) \in [-1, 1]^n$ are arbitrary, and the player may see the whole collection before deciding the signs. In this case, a random choice of signs achieves expected $\ell^\infty$ norm $O(\sqrt{n \log n})$, but there exists a choice of signs for which the norm of the signed sum is bounded by $K \sqrt n$ for some universal $K \leq 6$; in Spencer's memorable formulation, ``six standard deviations suffice.'' Later work developed algorithmic approaches to such bounds, notably Bansal~\cite{Bansal2010} and Lovett--Meka~\cite{LovettMeka2015}. An online version of the vector balancing problem was considered as early as 1977 in \cite{Spencer1977}, but it originally attracted little attention because the worst-case version of the problem is trivial (see the discussion of Game~I in \cite{Spencer1977}).

Here, we are interested in a \textit{random} and \textit{online} version of the problem. We fix random vectors $\bzeta(1),\dots,\bzeta(n)$ such that $(\zeta_i(k))_{i,k = 1,\dots,n}$ are i.i.d.\ standard Gaussian random variables. For notational simplicity, we rescale the vectors by $1/\sqrt{n}$, and set
\begin{align*}
    \bZ(k) \coloneqq \frac{1}{\sqrt{n}} \bzeta(k),
\end{align*}
so that $(Z_i(k))_{i,k = 1,\dots,n}$ are i.i.d.\ Gaussians with mean $0$ and variance $1/n$. With this parameterization, the optimization problem of interest is
\begin{align} \label{def.vn.intro}
    V^n \coloneqq \inf_{\eps(1),\dots,\eps(n)} \E \Big| \sum_{k = 1}^n \eps(k) \bZ(k) \Big|_{\infty} = \inf_{\eps(1),\dots,\eps(n)} \frac{1}{\sqrt{n}} \E \Big| \sum_{k = 1}^n \eps(k) \bzeta(k) \Big|_{\infty},
\end{align}
with the infimum taken over all $\{\pm 1\}$-valued random variables $\eps(1),\dots,\eps(n)$ such that $\eps(k)$ is a function of $\bZ(1),\dots,\bZ(k)$ (or equivalently, the discrete stochastic process $\eps(\cdot)$ is adapted with respect to the filtration generated by $\bZ(\cdot)$).

How large is $V^n$? Notice that if $\eps(k) = +1$ for all $k$ (or indeed, if the player uses any deterministic or random sequence of signs independent of the vectors), then the entries of $\sum_{k=1}^n \eps(k) \bZ(k)$ are independent standard Gaussians, and the sum's expected $\ell^\infty$ norm is, to leading order, $\sqrt{2\log n}$.
A striking result of Bansal and Spencer~\cite{BansalSpencer2020} shows that (in the case that $(\zeta_i(k))_{i,k = 1,\dots,n}$ are Rademacher random variables, rather than standard Gaussians) there are much better strategies: there is a simple algorithm for selecting the signs that achieves $\E \left| \sum_{k=1}^n \eps(k) \bZ(k)\right|_\infty \leq K$ for some $K$ independent of $n$. In particular, this means that the sequence $(V^n)_{n \in \N}$ is bounded.

The work of Bansal and Spencer reignited interest in generalizations and variants of the online vector balancing problem.
These include versions in which the vectors $\bzeta(k)$ are adversarial but ``oblivious'' (i.e., are fixed before the signs $\eps(k)$ are chosen), as in \cite{AlweissLiuSawhney,Kulkarni2024}, or ``smoothed'' (i.e., are convolved with small isotropic noise), as in \cite{Haghtalab2024}.
Other randomized variants focus on vectors with non-independent entries (for example, uniform $d$-sparse binary vectors), for which the asymptotics of $V^n$ can be different~\cite{bansal2021online,Altschuler2025}, or on stronger control over the partial sums, for instance, when the $\ell^\infty$ norm must be controlled uniformly across time, as opposed to just the final time~\cite{Bansal2019OnlineVB}.

\subsection{The mean-field limit}

Most existing results on vector balancing focus on designing and analyzing algorithms which achieve the optimal cost (or a nearly optimal cost), up to constant factors and with high probability. In the setting of \eqref{def.vn.intro}, such results imply that $V^n$ is bounded, uniformly in $n$. Here we take a different perspective, and aim to understand the asymptotic behavior of $V^n$ as $n \to \infty$. Our main result identifies the exact limit of $V^n$ as the value of a nonstandard stochastic control problem of mean-field type:
\begin{align} \label{def.vinf.intro}
    V^{\infty} \coloneqq \inf \bigg\{ \Big\| \int_0^1 \alpha(t)\, dt + B(1) \Big\|_\infty : \|\alpha(t)\|_2 \leq \sqrt{\frac{2}{\pi}} \text{\; for all $t \in [0,1]$} \bigg\}.
\end{align}
This optimization problem is set on a fixed probability space hosting a (one-dimensional) Brownian motion $B = (B(t))_{0 \leq t \leq 1}$. The infimum is taken over processes $\alpha = (\alpha(t))_{0 \leq t \leq 1}$ which are progressively measurable with respect to the Brownian filtration and satisfy the constraint $\|\alpha(t)\|_2=\E[\alpha(t)^2]^{1/2} \leq \sqrt{2/\pi}$ for all $t\in[0,1]$. Here and throughout, $\|\cdot\|_\infty$ denotes the $L^\infty$ norm of a random variable, i.e., $\|X\|_\infty = \esssup |X|$.

\begin{theorem} \label{th:intro}
If $(\zeta_i(k))_{i,k = 1,\dots,n}$ are independent standard Gaussians, and $V^n$ is defined as in \eqref{def.vn.intro}, then $\lim_{n\to\infty} V^n = V^\infty$. If instead $\zeta_i(k)$ are i.i.d.\ with mean zero, variance 1, and finite fourth moment, then we still have $\liminf_{n\to\infty} V^n \ge V^\infty$.
\end{theorem}

We conjecture that $\lim_{n\to\infty} V^n = V^\infty$ remains true beyond the Gaussian setting, for a broad class of distributions for $\zeta_i(k)$, but we have not been able to prove the upper bound $\limsup_{n\to\infty} V^n \le V^\infty$ beyond the Gaussian setting.

The same arguments also apply when the number of vectors is allowed to differ from the ambient dimension, provided the two remain proportional. More precisely, for $m \in \N$, define
\begin{equation} \label{def.vnm}
    V^{n,m} \coloneqq  \inf_{\eps(1),\dots,\eps(m)} \E \Big| \sum_{k = 1}^m \eps(k) \bZ(k) \Big|_{\infty},
\end{equation}
with $(Z_i(k))_{i = 1,\dots,n, k = 1,\dots,m}$ being i.i.d.\ Gaussians with variance $1/n$. Our arguments show that
\begin{equation} \label{Tconvergence}
    \lim_{n \to \infty} \frac{m(n)}{n} = T \in (0,\infty) \quad\implies\quad
    \lim_{n \to \infty} V^{n,m(n)} = V_T^{\infty},
\end{equation}
with
\begin{equation} \label{def.vTinf}
    V_T^{\infty} \coloneqq \inf \bigg\{ \Big\| \int_0^T \alpha(t)\, dt + B(T) \Big\|_\infty : \|\alpha(t)\|_2 \leq \sqrt{\frac{2}{\pi}}\text{\; for all $t \in [0,T]$} \bigg\}.
\end{equation}
In addition to the intrinsic interest in uncovering a scaling limit of the online vector balancing problem \eqref{def.vn.intro}, there are also potential algorithmic implications; see Subsection~\ref{subsec.algo}.
We also emphasize that, to the best of our knowledge, the bound $V^n = O(1)$ has so far only been proved by Bansal and Spencer in the case that $(\zeta_i(k))_{i,k = 1,\dots,n}$ are Rademacher; in particular our results give the first proof that $V^n = O(1)$ in the Gaussian case.

We are aware only of two other results on exact scaling limits for discrepancy problems.
The concurrent work of Huang--Sellke--Sun and Huang--Sellke~\cite{HuangSellkeSun2026,HuangSellke2026} derives a similar control-theoretic limit for the \emph{offline} version of the problem. Their results apply to a more general formulation, where the signs are chosen so that the coordinates of the signed sum $\sum_{k=1}^m \eps(k)\bZ(k)$ either all lie in a given subset of $\R$, or have an empirical measure which is close in 2-Wasserstein distance to a given probability measure.
Less closely related is that of Kohn--Serfaty \cite{kohn2006deterministic}, in which the vectors $\bzeta(k)$ are deterministic unit vectors, chosen adversarially, and the controller gets to choose a sign and move a length $h$ along the signed vector. The goal is to exit from a given domain as quickly as possible, and the limiting value as $h\downarrow 0$ is characterized in terms of the mean curvature flow of the domain. This is, of course, a very different sort of objective and scaling limit.

\subsection{Results on the mean-field problem}

As mentioned above, the limiting optimization problem \eqref{def.vinf.intro} is a mean-field control (MFC) problem, i.e., a stochastic control problem in which the objective is a nonlinear function of the law of the controlled process. MFC problems and their competitive counterparts, mean field games (MFG), have been studied extensively over the past two decades. We refer to the two-volume monograph \cite{CD1, CD2} for an overview of the theory, and in particular to Chapter 6 of \cite{CD1} for an introduction to mean-field control theory. To make the connection to MFC more transparent, observe that the optimization criterion is a \emph{nonlinear} functional of the law $\cL(X(1))$ of $X(1)$:
\begin{align} \label{mfc}
    \inf_{\alpha \in \mathbb{A}} \big[\cL(X(1))\big]_\infty, \quad \text{ where } [m]_{\infty} \coloneqq \inf \big\{ r > 0 : m([-r,r]^c) = 0 \big\},
\end{align}
where $\mathbb{A}$ is the set of progressively measurable processes satisfying $\|\alpha(t)\|^2_2 \leq 2/\pi$ for all $t \in [0,1]$,
and $X$ is the state process determined from $\alpha$ by
\begin{align*}
    dX(t) = \alpha(t)\,dt + dB(t), \quad X(0) = 0.
\end{align*}
The problem \eqref{mfc} falls outside the scope of standard MFC theory in two ways. First, the condition $\|\alpha(t)\|^2_2 \leq 2/\pi$ is a constraint on the law of the control, which is not typical. Second, the cost function $[m]_{\infty}$ is very singular; it is not even obvious a priori how to find an admissible control $\alpha$ which achieves a finite cost.

Because $[m]_\infty$ is lower semicontinuous (but not upper semicontinuous) with respect to weak convergence, it is straightforward (at least once we know that $V^{\infty}$ is finite) to prove the existence of an optimizer to \eqref{mfc}. Unfortunately, because of the singularity of $[m]_\infty$, we do not have a complete understanding of the structure of optimizers. It is not even clear to us if the optimizer is unique. However, we do have a duality result which gives an alternative characterization of $V^{\infty}$; see Theorem \ref{thm:duality-regularized}. Moreover, if we soften the terminal condition $[m]_\infty$ in an obvious way by replacing it with  $( \int |x|^p\,m(dx))^{1/p}$,
then we can show that there is a unique optimizer $\alpha$, and we can use convex duality to rigorously obtain necessary and sufficient conditions for optimality; see Theorem \ref{thm:duality-regularized} and Proposition \ref{prop:primal-optimizers-finite-p}. One nice way to express these results is to say that $\alpha$ takes the form
\begin{align*}
    \alpha(t) = - \sqrt{\frac{2}{\pi}} \frac{Y(t)}{\|Y(t)\|_2},
\end{align*}
where $(X,Y,Z)$ is the unique solution of the forward-backward stochastic differential equation
\begin{equation*}
        \left\{
            \begin{alignedat}{2}
                dX(t) &= -\sqrt{\frac{2}{\pi}} \frac{Y(t)}{\|Y(t)\|_{2}}\,dt + dB(t),\qquad&& X(0) = 0,\\
                dY(t) &= Z(t)\, dB(t),\qquad&& Y(1) = \frac{X(1) |X(1)|^{p-2}}{\|X(1)\|_p^{p-1}},
                \qquad \smash{\raisebox{\normalbaselineskip}{$0 \le t \le 1$.}}
            \end{alignedat}
        \right.
\end{equation*}
Understanding the structure of this softened version of the problem is interesting in its own right, but also allows us to infer some properties of optimizers of the limiting problem; for example, we show in Proposition \ref{lem.optproperties} that there exists at least one optimizer for \eqref{mfc} with the property that $\alpha(t) = \hat{\alpha}(t,X_t)$, with $\hat{\alpha}(t,\cdot)$ an odd function which points inwards, in the sense that $\hat{\alpha}(t,x) x \leq 0$. This property turns out to be crucial in the proof of the upper bound in Theorem \ref{th:intro}.

While we cannot solve the problem \eqref{def.vinf.intro} explicitly, we do have several explicit upper and lower bounds for $V^{\infty}$, and more generally for $V_T^{\infty}$. (We also explain briefly in Section \ref{sec:l2-problem} that the analogous and much easier problem with $L^2$ instead of $L^\infty$ objective is explicitly solvable.)
These bounds,  presented in Section \ref{sec:bounds_and_asymptotics}, imply the following:
\begin{itemize}
    \item For $T = 1$, we have
    \begin{align*}
       1.09 \leq \Phi^{-1}\Big(\frac{1 + e^{-1/\pi}}{2} \Big) \leq V^{\infty} \leq \Big(\frac{\pi}{2}\Big)^{3/2} \leq 1.97
    \end{align*}
    where $\Phi : \R \to (0,1)$ is the CDF of the standard Gaussian.
    In particular, for $n$ large enough, we know that $V^n$ is between $1$ and $2$. The lower bound, shown in Proposition \ref{prop.firstlowerbound}, essentially comes from relaxing the constraint from $\sup_{t \in [0,1]}\E[\alpha(t)^2] \le 2/\pi$ to $\int_0^1 \E[\alpha(t)^2]\,dt \le 2/\pi$, which results in an optimization problem which is explicitly solvable via a least-entropy principle and Girsanov's theorem. The upper bound comes from the fact $T \mapsto V_T^{\infty}$ is nondecreasing (shown in Lemma \ref{lem.nondecreasing}), and Proposition \ref{prop.firstupperbound} bounds the limit as $T\to\infty$ by $(\pi/2)^{3/2}$.
    This is shown using an explicit choice of control which happens to correspond to Brownian motion conditioned to remain in an interval for all time (Remark \ref{re:conditionedBM}).
    \item As $T\to 0$, we have $V_T^{\infty} \sim \sqrt{2T \log(1/T) }$,
    in the sense that the ratio converges to 1.
    This is obtained by combining the aforementioned lower bound of Proposition \ref{prop.firstlowerbound} (see in particular Remark \ref{rmk.phiinverse.taylor}) with upper bounds of Propositions \ref{prop.secondupperbound} and \ref{prop.smalltimeasymptotics}. The upper bound comes from explicit choices of controls, corresponding to the F\"ollmer drift associated with any terminal distribution having small enough relative Fisher information with respect to a Gaussian. Among this class of controls, the optimal value can be computed in terms of the leading Dirichlet eigenvalue of the Ornstein--Uhlenbeck operator restricted to a bounded interval.
    \item Qualitatively, we observe in Proposition \ref{pr:scaling-dec+convex} that $V^\infty_T$ is continuous in $T$, and $T^{-1/2}V_T^{\infty}$ is a nonincreasing convex function of $\sqrt{T}$.
\end{itemize}

\subsection{Heuristic derivations of the mean-field limit}
The heuristic link between~\eqref{def.vn.intro} and~\eqref{def.vinf.intro} follows from a ``mean-field'' viewpoint.
Consider the $n$-dimensional state process associated with the discrete control problem $V^n$ in~\eqref{def.vn.intro}:
\begin{equation} \label{def:intro:Y(k)}
	\bY(k) \coloneqq  \sum_{\ell=1}^k \eps(\ell) \bZ(\ell),\qquad k=0,1,\dots,n
\end{equation}
Letting $\cG(k) \coloneqq \sigma(\bZ(\ell) : \ell \leq k)$ denote the canonical filtration, we can write the Doob decomposition
\begin{equation}
	\bY(k) - \bY(k-1) = \bm{b}(k-1) + \bm{\sigma}(k)\,, \label{intro:Doob}
\end{equation}
where $\bm{b}(k-1) \coloneqq \E[\eps(k) \bZ(k)\,|\,\cG(k-1)]$ is a predictable drift and $\bm{\sigma}(k) \coloneqq \eps(k) \bZ(k) - \bm{b}(k-1)$ is the martingale increment.
For any unit vector $\bm{v} \in \R^n$, the fact that $\eps(k)$ takes values in $\{\pm 1\}$ implies
\[\bm{v} \cdot \bm{b}(k-1) = \E\big[\eps(k) \bm{v} \cdot \bZ(k) \, \big|\, \cG(k-1) \big] \leq  \E\big[|\bm{v} \cdot \bZ(k)| \, \big|\, \cG(k-1) \big] = \sqrt{\frac{2}{\pi n}} \quad \text{a.s.}\]
Taking a supremum over $\bm{v}$ shows that the drift satisfies $|\bm{b}(k-1)|_2 \leq \sqrt{2/(\pi n)}$ almost surely.
Conversely, one can show that for any $\cG(k-1)$-measurable random vector $\bm{u}$ satisfying $|\bm{u}|_2 \leq \sqrt{2 / (\pi n)}$, there exists a $\cG(k)$-measurable $\eps(k) \in \{\pm 1\}$ such that $\bm{b}(k-1) = \bm{u}$.
We reach the first important observation in our heuristic derivation: \textit{the constraint that $\eps(k) \in \{\pm 1\}$ is equivalent to an $\ell^2$ bound on the drift}.

Similarly, the martingale increment $\bm{\sigma}(k)$ satisfies
\begin{equation}
	\Var(\bm{v} \cdot \bm{\sigma}(k)\,|\, \cG(k-1)) = \E \big[\eps(k)^2 (\bm{v} \cdot \bZ(k))^2 \,|\, \cG(k-1)\big] - (\bm{v} \cdot \bm{b}(k-1))^2 = \frac 1n - (\bm{v} \cdot \bm{b}(k-1))^2\,.
\end{equation}
Since this quantity is order $n^{-1}$, the projection of the martingale increment in any direction has the same scaling as the increments of a Brownian motion under the time reparameterization $t = k/n$.
We reach the second observation: \textit{The one-dimensional marginals satisfy the scaling assumptions of the martingale central limit theorem.}

Combining these observations, we are led to consider the behavior of a typical coordinate of $\bY(k)$: letting $U \sim \mathrm{Unif}\{1,\dots,n\}$ be a random coordinate, independent of everything else, we can define a new one-dimensional process $\widebar X(t) \coloneqq Y_U(\lfloor nt \rfloor)$ representing the evolution of the ``typical entry'' in our state vector.
The above considerations imply that, for $k = \lfloor nt \rfloor$,
\begin{equation}\label{discretesde.intro}
	\widebar X\left(t + \frac 1n\right) - \widebar X(t) = (n b_U(k)) \cdot \frac 1n + \sigma_U(k+1)\,,
\end{equation}
where $\alpha(t) \coloneqq n b_U(k)$ satisfies
\begin{equation}\label{alphat.intro}
	\E\, \alpha(t)^2 = \frac 1n \sum_{i=1}^n \E (n b_i(k))^2 \leq n \cdot \frac{2}{ \pi n} = \frac 2 \pi
\end{equation}
and where
\begin{equation}\label{martingale.intro}
	\Var\big(\sigma_U(k+1)\,|\,U, \cG(k)\big) = \frac 1n - b_U(k)^2 = \frac 1n \left(1- \frac 1n \alpha(t)^2\right) \approx \frac 1n \,.
\end{equation}
Combining~\eqref{discretesde.intro} with~\eqref{alphat.intro} and~\eqref{martingale.intro} suggests that, in the limit, the typical coordinate $\widebar X$ evolves according to
\begin{equation}
	d \widebar X(t) = \alpha(t)\, dt + d B(t)\,,
\end{equation}
where $B$ is a standard Brownian motion and $\alpha$ satisfies $\E\,\alpha(t)^2\leq 2/\pi$ for every $t$. The $\ell^\infty$ norm of the vector $\bY(n)$ in the objective of $V^n$ is then naturally replaced by the $L^\infty$ norm of the random variable $\widebar X(1)$.
This is the control problem $V^\infty$ appearing in~\eqref{def.vinf.intro}.

\subsubsection{Algorithmic implications}\label{subsec.algo}
This derivation of~\eqref{def.vinf.intro} also suggests a method for algorithmically extracting a strategy for the finite-$n$ problem from a strategy for its limiting counterpart.
Suppose that a control $\alpha$ of~\eqref{def.vinf.intro} is given in feedback form as $\alpha(t) = \hat \alpha(t, X(t))$.
The above derivation suggests coupling the state $\bY(k)$ in the discrete control problem with $n$ independent copies $X_1(t), \dots, X_n(t)$ of $X(t)$ with $t = k/n$, and choosing $\eps(k)$ to satisfy
\begin{equation}\label{eq:intro-alg-goal}
b_i(k-1) = \E\big[\eps(k) Z_i(k)\,\big|\,\cG(k-1)\big] = \frac{1}{n} \hat \alpha(t, X_i(t)).
\end{equation}
Lemma~\ref{lem.coupling.discrete} shows there is such a choice of $\eps(k)$ so long as the resulting vector $\bm{b}(k-1)$ satisfies $|\bm{b}(k-1)|_2 \leq \sqrt{2/(\pi n)}$.
For example, in the simple case where $|\bm{b}(k-1)|_2 = \sqrt{2/(\pi n)}$, we can take $\eps(k) = \sgn(\bm{b}(k-1) \cdot \bZ(k))$.
Since $X_1(t), \dots, X_n(t)$ are independent, the law of large numbers implies that $n|\bm{b}(k-1)|_2^2 = \frac{1}{n} \sum_{i=1}^n \hat \alpha(t, X_i(t))^2 \to \E\, \alpha(t)^2 \leq 2/\pi$.
Therefore, up to a possible negligible modification of $\bm{b}(k-1)$ to ensure it satisfies the norm bound,~\eqref{eq:intro-alg-goal} is achievable.

The result is a simple prescription: given an optimal continuous control in feedback form $\hat \alpha(t, X(t))$ and a current state $\bY(k)$ for the discrete problem, choose $\eps(k)$ by projecting the vector $(\hat \alpha(k/n, Y_i(k)))_{i=1}^n$ to the $\ell^2$ ball in $\R^n$ of radius $\sqrt{2/\pi}$, and use Lemma~\ref{lem.coupling.discrete} to ensure that~\eqref{eq:intro-alg-goal} holds.
Though we expect the resulting discrete-time strategy to achieve a value close to $V^\infty$, we stress that our proof of the upper bound does not take this approach, since making it rigorous requires rather stringent regularity assumptions on the function $\hat \alpha$, which we cannot establish for optimal controls of~\eqref{def.vinf.intro}.
As we describe below, our upper bound proof uses an indirect argument that does not immediately yield an algorithmic guarantee.

\subsubsection{A PDE viewpoint}
Alternatively, one can formally derive the limiting problem \eqref{def.vinf.intro} in a more analytical way, by considering the dynamic value function $\cV^n : \{0,1,\dots,n\} \times \R^n \to \R$ for the $n$-dimensional optimization problem. This function records the value of the same optimization problem from different initial conditions, i.e.,
\begin{align*}
    \cV^n(k,\bx) \coloneqq \inf_{\eps(\cdot)} \E \, \Big| \bx + \sum_{\ell = k+1}^n \eps(\ell) \bZ(\ell)\Big|_{\infty}.
\end{align*}
It can alternatively be characterized by a dynamic programming principle
\begin{align*}
    \cV^n(k, \bx) = \E\Big[ \inf_{\eps \in \{\pm 1\}} \cV^n\big(k+1, \bx + \eps \bZ \big) \Big],
\end{align*}
where $\bZ$ is a centered Gaussian with covariance $n^{-1} I_n$,
together with the terminal condition $\cV^n(n, \bx) = |\bx|_{\infty}$. Now, suppose that there is a smooth function $\cU^n : [0,1] \times \R^n \to \R$ such that
\begin{align*}
    \cU^n(k/n, \bx) = \cV^n(k,\bx), \qquad k =0,\dots,n.
\end{align*}
Then, we can perform a Taylor expansion to conclude that
\begin{align*}
    \cV^n\big(k+1,\bx &+ \eps  \bZ \big) - \cV^n(k,\bx) = \cU^n\big( (k+1)/n, \bx + \eps  \bZ \big) - \cU^n\big( k/n, \bx \big)
    \\
    & \approx \frac{1}{n} \partial_t \cU^n(k/n, \bx) + \eps \sum_{i = 1}^n \partial_{x_i} \cU^n(k/n,\bx) Z_i + \sum_{i,j = 1}^n \partial_{x_ix_j} \cU^n(k/n,\bx) Z_i Z_j,
\end{align*}
and so optimizing over $\eps \in \{\pm 1\}$,
\begin{align*}
    \inf_{\eps \in \{\pm 1\}} &\cV^n(k+1,\bx + \eps \bZ) - \cV^n(k,\bx)
    \\
    &\approx \frac{1}{n} \partial_t \cU^n(k/n, \bx) -  \Big|  \sum_{i = 1}^n \partial_{x_i} \cU^n(k/n,\bx) Z_i \Big| + \sum_{i,j = 1}^n \partial_{x_ix_j} \cU^n(k/n, \bx) Z_i Z_j.
\end{align*}
Now, if the $Z_i\sim \mathcal{N}(0,1/n)$ are independent, then we can take expectations and use the fact that
\begin{align*}
    \sum_{i = 1}^n \partial_{x_i} \cU^n(k/n,\bx) Z_i \sim \cN\Big(0,  \frac{1}{n} \sum_{i = 1}^n |\partial_{x_i} \cU^n(k/n,\bx)|^2 \Big)
\end{align*}
to deduce that
\begin{align*}
    &0 = \E\Big[ \inf_{\eps \in \{\pm 1\}} \cV^n\big(k+1,\bx + \eps  \bZ \big) \Big] - \cV^n(k,\bx)
    \\
    &\quad \approx \frac{1}{n} \partial_t \cU^n(k/n, \bx) - \sqrt{\frac{2}{\pi}} \bigg( \frac{1}{n} \sum_{i = 1}^n |\partial_{x_i} \cU^n(k/n,\bx) |^2 \bigg)^{1/2} + \frac{1}{n} \sum_{i = 1}^n \partial_{x_ix_i} \cU^n(k/n,\bx).
\end{align*}
Multiplying by $n$, we find that $\cU^n$ should approximately solve a Hamilton--Jacobi--Bellman equation:
\begin{align*}
    - \partial_t \cU^n - \sum_{i = 1}^n \partial_{x_ix_i} \cU^n + \sqrt{\frac{2 n}{\pi}} \Big(\sum_{i = 1}^n |\partial_{x_i} \cU^n |^2 \Big)^{1/2} \approx 0, \quad \cU^n(1,\bx) = |\bx|_{\infty}.
\end{align*}
Recognizing the control problem whose value function solves this HJB equation, this suggests that
\begin{align} \label{vn.npartcontrol}
    V^n \approx \inf_{\bm{\alpha}} \E\, \Big| \int_0^1 \bm{\alpha}(t)\, dt + \bm{B}(1) \Big|_{\infty} = \inf_{\bm{\alpha}} \E\Big[ \max_{i = 1,\dots,n} \Big| \int_0^1 \alpha_i(t)\, dt + B_i(1) \Big|\Big],
\end{align}
where $\bm{B}$ is a standard $n$-dimensional Brownian motion, and the infimum is taken over all processes $\bm{\alpha}$ which are progressively measurable with respect to the filtration generated by $\bm{B}$ and satisfy
\begin{align*}
    \frac{1}{n} \sum_{i = 1}^n \alpha_i(t)^2 \leq \frac{2}{\pi}, \quad \text{a.s., for all } 0 \leq t \leq 1.
\end{align*}
Finally, to get to \eqref{def.vinf.intro}, we can note that the stochastic control problem appearing in \eqref{vn.npartcontrol} is of mean-field type, and its formal mean-field limit is exactly \eqref{def.vinf.intro}.

\subsection{Overview of the proof}

We prove the lower and upper bounds
\begin{align} \label{lower-upperbounds.intro}
    V^{\infty} \le \liminf_{n \to \infty} V^n, \quad \text{and} \quad
    \limsup_{n \to \infty} V^n \leq V^{\infty},
\end{align}
through completely separate arguments.\medskip

\noindent
\textit{The proof of the lower bound.} We prove the lower bound under the ``universal'' assumption that the coordinates $\zeta_i(k)$ are i.i.d.\ with mean zero, variance one, and finite fourth moment. As above, we work with the rescaled variables $Z_i(k)=n^{-1/2}\zeta_i(k)$. Under this assumption, the proof proceeds by compactness and weak convergence arguments, drawing inspiration from the classical proof of Donsker's invariance principle and from compactness methods used in the limit theory for (continuous-time) mean-field stochastic control~\cite{lacker2017limit}. We show that if we choose for each $n \in \N$ \textit{any} admissible control $\eps^n$ for the problem defining $V^n$, then we have
\begin{equation}
    \label{eq:thm-lower-bound:to-show}
    V^\infty \leq \liminf_{n\to\infty} \E\,\Big|\sum_{k=1}^n \epsilon^n(k) \bm{Z}^n(k)\Big|_{\infty}.
\end{equation}
To do this, we work with a weak formulation of the limiting control problem, in which the ``weak control" is the joint law of $(\int_0^{\cdot} \alpha(t)\, dt, B(\cdot))$, viewed as a measure on the space $\cC^2$ of continuous paths $[0,1] \to \R^2$. More precisely, denoting by $A$ and $B$ the canonical processes on the space $\cC^2$, we denote by $\cA^\mathrm{w}$ the set of measures $m$ on $\cC^2$ such that, under $m$,
\begin{gather}
	A=(A(t))_{t\in[0,1]} \text{ has absolutely continuous paths and } A(0)=0, \label{eq:thm:limit-points:ac0-paths}\\
	\E\,\dot{A}(t)^2 \leq \frac{2}{\pi},\quad \text{for a.e.\ } t\in[0,1],\label{eq:thm:limit-points:l2-constraint}\\
	B=(B(t))_{t\in[0,1]} \text{ is an $\mathbb{F}^{A,B}$-Brownian motion}.\label{eq:thm:limit-points:bm}
\end{gather}
Here, $\mathbb{F}^{A,B}$ denotes the filtration generated by $(A,B)$. As shown in Lemma \ref{lem.equivstrong}, the possible extra randomness in $A$ which is not adapted to the Brownian filtration cannot improve the value; that is,
\begin{equation*}
	V^\infty = \inf_{m \in \mathcal{A}^\mathrm{w}} \|A(1) + B(1)\|_{L^\infty(m)}.
\end{equation*}
With this weak formulation of the limiting problem in hand, we prove \eqref{eq:thm-lower-bound:to-show} through the following compactness argument:
\begin{enumerate}
	\item Fix a sequence $(\eps^n)_{n\in\N}$ such that $\eps^n$ is an admissible control for $V^n$. For each $n$, the associated state process
	\begin{equation*}
        \bm{Y}^n(k) \coloneqq \sum_{\ell=1}^k \eps^n(\ell) \bm{Z}^n(\ell),\qquad k=0,1,\dots,n,
    \end{equation*}
    is embedded in the time interval $[0,1]$ by time rescaling and linear interpolation. The resulting continuous-time process $\hat{\bm{Y}}^n=(\hat{\bm{Y}}^n(t))_{t\in[0,1]}$ is then decomposed into two additive terms, a ``drift'' term $\bm{A}^n$ and a ``martingale'' term $\bm{M}^n$, based on the Doob decomposition of $\bm{Y}^n$ from \eqref{intro:Doob}.
	\item We analyze $\hat{\bm{Y}}^n$ through the joint empirical measure $\hat{\mu}_n$ of the $n$ coordinate trajectories $(A_1^n, M_1^n),\dots, (A_n^n,M_n^n)$. This is a random element of $\cP(\C^2)$, the space of probability measures on $\C^2$, equipped with weak convergence. We show that the laws $\cL(\hat \mu_n)$ form a tight sequence in $\P(\P(\C^2))$.
	\item We then characterize the limit points of $\cL(\hat{\mu}_n)$ in Theorem~\ref{thm:limit-points}, showing that any limit point is supported on the set $\cA^\mathrm{w}$.
	\item The proof concludes using the lower semicontinuity of the map $\P(\C^2)\to[0,\infty]$ which sends $m=\cL(A,B)$ to $\|A(1)+B(1)\|_{L^\infty(m)}$.
\end{enumerate}\bigskip

\noindent
\textit{The proof of the upper bound.} For the upper bound, we return to the Gaussian setting, where the increments $Z_i(k)$ are independent $\mathcal N(0,1/n)$ random variables. The proof is significantly more challenging because the function $m \mapsto [m]_{\infty}$ is lower semicontinuous, but not upper semicontinuous with respect to weak convergence or any topology in which we can expect to obtain compactness. A natural idea is to embed controls for the limiting problem into the $n$-dimensional problem in an appropriate way. More precisely, we would like a procedure which starts with an admissible control $\alpha$ for $V^\infty$, and gives an admissible control $\eps$ for $V^n$ such that
\begin{align} \label{embedding}
    \E\, \Big| \sum_{k = 1}^n \eps(k) \bZ(k) \Big|_{\infty} \leq \Big\| \int_0^1 \alpha(t)\,dt + B(1) \Big\|_{\infty} + o_n(1).
\end{align}
In Subsection \ref{subsec.coupling}, we design such a procedure, relying on the Gaussian assumption in an essential way. The starting point is to take $n$ independent copies $(\alpha_i,X_i,B_i)_{i=1,\dots,n}$ of the triple $(\alpha,X, B)$, where $\alpha$ is a control for the limiting problem, $X$ is the corresponding state process, and $B$ is the driving Brownian motion. In particular,
\begin{equation*}
    X_i(t) = \int_0^t \alpha_i(s)\, ds + B_i(t), \qquad 0 \leq t \leq 1,
\end{equation*}
for each $i=1,\dots,n$. We then construct random vectors $(\bZ(k))_{k = 1,\dots,n}$ with independent $\mathcal{N}(0,1/n)$ coordinates, together with a suitable control $(\eps(k))_{k = 1,\dots,n}$, so that the state process $(\bY(k))_{k = 1,\dots,n}$ defined as in \eqref{def:intro:Y(k)}
satisfies 
\begin{align} \label{embedding2}
    \E  | \bY(n) - \bX(1) |_{\infty} = o_n(1),\qquad\text{where } \bX(1) \coloneqq (X_1(1),\dots,X_n(1)).
\end{align}
Since $|X_i(1)| \leq \|X(1)\|_{\infty}$ a.s., this implies \eqref{embedding}. 

We emphasize that the obvious ``synchronous" coupling, in which we take $Z_i(k) = B_i\big(\frac{k}{n}\big) - B_i \big(\frac{k-1}{n}\big)$, does not seem to lead to an estimate like \eqref{embedding2}. Instead, we take $Z_i(k) = W_i\big(\frac{k}{n}\big) - W_i \big(\frac{k-1}{n}\big)$, where $(W_i)_{i = 1,\dots,n}$ is an auxiliary $n$-dimensional Brownian motion  which is coupled with $(B_i)_{i = 1,\dots,n}$ through a delicate ``sign-flipping" procedure involving the controls $(\alpha_i)_{i = 1,\dots,n}$. 

While the coupling we construct can be implemented for general admissible controls $\alpha$, unfortunately we can only prove the estimate \eqref{embedding} for a restricted class of sufficiently regular controls: roughly speaking, controls which are mean-zero martingales and bounded in $L^p$ for some $p > 4$. Since we do not expect the optimal controls to have this structure, this means that our coupling argument does not give the upper bound by itself.

Nevertheless, the argument is still useful. Using the F\"ollmer drift, we can build admissible controls which have these special properties and achieve a finite cost (see Lemma \ref{lem.follmer}). These controls then serve as the basis for a more indirect proof, organized as follows:
\begin{enumerate}
    \item First, by using the coupling procedure mentioned above, we prove for each $p > 4$ a non-asymptotic upper bound on the value function $\cV^n : \{0,1,\dots,n\} \times \R^n \to \R$ for the $n$-dimensional problem, of the form
    \begin{align} \label{nvalue.upper.intro}
        \cV^n(k,\bx) \leq \cU_p(k/n, m_{\bx}^n) + o_n(1), \qquad m_{\bx}^n \coloneqq \frac{1}{n}\sum_{i=1}^n\delta_{x_i},
    \end{align}
    where $\cU_p : [0,1] \times \cP(\R) \to \R$ is, roughly speaking, the value function for the limiting problem when controls $\alpha$ are restricted to a class of ``$L^p$ F\"ollmer drifts". The upper bound \eqref{nvalue.upper.intro} is the most technically demanding step in our argument, and is carried out in Section \ref{subsec.coupling}. The precise definition of $\cU_p$ is given in \eqref{def.up}. We expect that $\cU_p$ is strictly larger than $\cV^{\infty}$, so \eqref{nvalue.upper.intro} is not tight, but it provides a tractable non-asymptotic upper bound. The F\"ollmer drift has become a common tool for proving functional inequalities since \cite{lehec2013representation}; for our purposes, it just provides a convenient way to construct controls which lead to a given target distribution $m=\cL(X(1))$ and whose moments can be controlled  tractably in terms of $m$.
    \item For some $p > 4$, we consider the ``upper half-relaxed limit" of $\cV^n$, with respect to the $p$-Wasserstein metric $\bd_p$, i.e., the function $\cV^+_p : [0,1) \times \cP_p(\R) \to \R$ given by
    \begin{align*}
      \cV^+_p(t,m) \coloneqq \sup_{(k^n, \bx^n)} \limsup_{n \to \infty} \cV^n(k^n, \bx^n),
    \end{align*}
    where the first supremum is taken over all $k^n \in \{0,\dots,n\}$ and $\bx^n \in \R^n$ such that
    \begin{align*}
        k^n/n \to t, \quad m_{\bx^n}^n \xrightarrow{ \bd_p } m.
    \end{align*}
    This upper half-relaxed limit is a natural object in the theory of viscosity solutions, and is automatically upper semicontinuous. The upper bound \eqref{nvalue.upper.intro} from the first step allows us to see that
    \begin{equation*}
        \cV^+_p(t,m) \leq \cU_p(t,m),
    \end{equation*}
    which is useful because we can obtain tractable upper bounds for $\cU_p$ (see Lemmas \ref{lem.follmervalue} and \ref{lem.updet.upperbound}). In particular, we can deduce that for each $t \in [0,1)$, $\cV^+_p(t,\cdot)$ is locally bounded on $\cP_p(\R)$, and satisfies
    \begin{align} \label{vplusp.term}
        \cV^+_p(1 - \eps, m) \leq [m]_{\infty} + \kappa(\eps),
    \end{align}
    for some $\kappa$ with $\kappa(\eps) \to 0$ as $\eps \to 0$.
    \item We show in Proposition \ref{prop.dpi} that $\cV_p^+$ satisfies a ``dynamic programming inequality", at least for \textit{bounded controls}. In particular, for each $t_0 \in [0,1)$ and $h \in (0,1 - t_0)$, any \textit{bounded} process $\alpha$ satisfying $\|\alpha(t)\|_2^2 \leq 2/\pi$, and any process $X$ with dynamics of the form
    \begin{align*}
        dX(t) = \alpha(t)\,dt + dB(t), \qquad t_0 \leq t \leq t_0 + h,
    \end{align*}
    we have
    \begin{align} \label{dpp.intro}
        \cV_p^+\Big( t_0, \cL\big(X(t_0)\big) \Big) \leq \cV_p^+\Big( t_0 + h, \cL\big(X(t_0 + h)\big) \Big).
    \end{align}
    We note that the proof of \eqref{dpp.intro} also uses in a crucial way the a priori estimate \eqref{nvalue.upper.intro}.
    \item We fix a control $\alpha^{\eps}$ which is optimal for the problem defining $V^{\infty}$, but on a slightly shorter time horizon $[0,1- \eps]$. This means that
    \begin{align} \label{xeps.value}
        \cV^{\infty}(\eps, \delta_0) &= \big\| X^{\eps}(1- \eps) \big\|_{\infty},
    \end{align}
    where $X^\eps$ denotes the corresponding state process. We note that the need to restrict to a shorter time horizon comes from the fact that $\cV_p^+$ only admits the bound \eqref{vplusp.term} on $[0,1) \times \cP_p(\R)$. Next, we use some properties of $\alpha^{\eps}$ proved in Section \ref{sec:duality-and-optimizers} to approximate $\alpha^{\eps}$ by admissible controls $\alpha^{\eps, \delta}$ which are \textit{bounded}, and such that the corresponding state processes $X^{\eps, \delta}$ satisfy
    \begin{align*}
        \cL\big( X^{\eps, \delta}(1-\eps) \big) \xrightarrow{\bd_p} \cL\big(X^{\eps}(1-\eps)\big),\qquad \text{ as $\delta\to0$}.
    \end{align*}
    We apply the definition of $\cV^+_p$, then \eqref{dpp.intro} to get
    \begin{align*}
        \limsup_{n \to \infty} V^n \leq \cV^+_p(0, \delta_0) \leq \cV_p^+\big(1- \eps, \cL\big(X^{\eps, \delta}(1-\eps)\big) \big).
    \end{align*}
    We then use upper semicontinuity of $\cV_p^+$ with respect to $\bd_p$ to send $\delta \to 0$, apply \eqref{vplusp.term}, and finally \eqref{xeps.value} to obtain
    \begin{align*}
        \limsup_{n \to \infty} V^n &\leq \cV_p^+\big(1 - \eps, \cL\big(X^{\eps}(1-\eps)\big) \big) \\
        &\leq \big\| X^{\eps}(1-\eps) \big\|_{\infty} + \kappa(\eps) = \cV^{\infty}(\eps, \delta_0) + \kappa(\eps).
    \end{align*}
    Finally, we use the convexity of the limiting problem to show in Lemma \ref{lem.nondecreasing} that $\eps \mapsto \cV^{\infty}(\eps, \delta_0)$ is nonincreasing. This yields
    \begin{align*}
        \limsup_{n \to \infty} V^n \leq \cV^{\infty}(0,\delta_0) + \kappa(\eps) = V^{\infty} + \kappa(\eps),
    \end{align*}
    and we send $\eps \to 0$ to complete the proof.
\end{enumerate}

\subsection{Further discussion and open problems}

There are several interesting questions left open by this work. The most immediate is the question of universality of the main convergence result Theorem \ref{th:intro}. As mentioned above, we establish the lower bound $\liminf_{n \to \infty} V^n \ge V^{\infty}$ for any increment distribution with mean zero, unit variance, and finite fourth moment, but prove the upper bound $\limsup_{n \to \infty} V^n \le V^{\infty}$ only in the Gaussian case. Generalizing the upper bound would be extremely interesting, but seems to be beyond the techniques of the present paper. The main challenge is that our argument relies on the coupling in Section \ref{sec.upperbounds} (specifically, Lemma~\ref{lem.coupling.discrete}), and it is not clear how to generalize it beyond the Gaussian case.

There are also several natural questions about the limiting control problem \eqref{def.vinf.intro}. The most ambitious goal would be to solve the problem explicitly. Even short of this, it would be interesting to better understand the structure of optimal controls. For example, while we prove in Proposition \ref{lem.optproperties} that there exist optimizers $\alpha$ satisfying several natural properties, we cannot at present rule out the possibility that there exist other optimizers which do not have these properties. Moreover, one heuristically expects optimizers to take the form $\alpha(t) = \hat{\alpha}(t,X(t))$ where $\hat\alpha : [0,1] \times \R \to \R$ is smooth away from time $1$. Verifying that this is the case, and characterizing the blow-up at time $1$, could provide further insight.

Finally, it would be interesting to study the long-time behavior of $V_T^{\infty}$, as defined in \eqref{def.vTinf}. As discussed above, we know that $T \mapsto V_T^{\infty}$ is nondecreasing, and converges as $T \to \infty$ to a constant which is between 1 and 2. But at present, we do not have a more exact description (e.g., in a variational form) of this constant.

\subsection{Acknowledgements}

We would like to thank Brice Huang, Mark Sellke, and Nike Sun for discussions about their forthcoming work~\cite{HuangSellkeSun2026,HuangSellke2026}. DL and JJ also thank Pierre-Louis Lions for several enlightening discussions, and in particular for his interesting suggestions about a possible PDE approach to the problem solved in the present paper.
JNW is grateful for conversations with Robert V.\ Kohn about an early version of this project.

\subsection{Organization of the paper}

In Section~\ref{sec:prelim}, we define dynamic value functions associated with the optimization problems studied here and develop their basic properties. In Section~\ref{sec:lower-bound} we prove the lower bound of Theorem \ref{th:intro}, and in Section~\ref{sec:upper-bound} we prove the upper bound. The final step of the proof of the upper bound relies on a lemma, whose proof is deferred to the subsequent Section~\ref{sec:duality-and-optimizers}, which establishes that there exists an optimizer of the problem~\eqref{def.vinf.intro} satisfying certain nice properties. In order to prove this lemma in Section~\ref{sec:duality-and-optimizers}, we must first develop duality results for the limiting problem and approximations of it. Section~\ref{sec:bounds_and_asymptotics} is devoted to upper and lower bounds on the value $V_T^{\infty}$ and related asymptotics. Finally, Section~\ref{sec:l2-problem} explicitly solves an $L^2$ variant of the problem $V_T^\infty$, which arises as the scaling limit of the $\ell^2$-version of the vector balancing problem. The appendices collect technical results and proofs used in the body of the paper.

\subsection{Notation}
For a Polish space $\mathcal{X}$, we denote by $\cP(\mathcal{X})$ the set of Borel probability measures on $\mathcal{X}$, endowed with the topology of weak convergence. For $p \geq 1$, we let $\cP_p(\mathcal{X})$ denote the subset of $\cP(\mathcal{X})$ consisting of measures with finite $p$th moment, endowed with the $p$-Wasserstein distance $\bd_p$. For a random variable $X$, defined on some probability space $(\Omega,\mathcal{F},\mathbb{P})$, we let $\cL(X)$ denote its law. If $X$ is real-valued, we use $\|X\|_p$ for its $L^p$ norm, or $\|X\|_{L^p(\mathbb{P})}$ to emphasize the underlying probability space. We use the shorthand $\E[X;A]$ for $\E[X\mathbf{1}_A]$. We reserve boldface for vectors and denote their components by subscripts; for example, $\bx = (x_1,\dots,x_n)\in\R^n$. The $\ell^p$ norm of a vector $\bx\in\R^n$ is denoted by $|\bx|_p$.

\section{Preliminaries}
\label{sec:prelim}

\subsection{Definitions of the value functions}

Fix a probability measure $\nu \in \cP(\R)$, representing the law of the coordinates of the vectors $\bzeta(1),\dots,\bzeta(n)$ from \eqref{def.vn.intro}. In our main results, we either take $\nu= \gamma$, where $\gamma$ denotes the standard Gaussian on $\R$, or we assume that
\begin{align} \label{nu.conditions}
    \int_{\R} x\, \nu(dx) = 0, \qquad \int_{\R} x^2\, \nu(dx) = 1, \qquad \int_{\R} x^4\, \nu(dx) < \infty.
\end{align}
For notational simplicity, we set
\begin{align*}
    \nu_{\lambda} \coloneqq (x \mapsto \lambda^{1/2} x)_{\#} \nu.
\end{align*}
In other words, if $\zeta \sim \nu$, then $\sqrt{\lambda} \zeta \sim \nu_{\lambda}$. In particular, if \eqref{nu.conditions} holds, then $\nu_{\lambda}$ has mean zero and variance $\lambda$.

We are now going to define value functions
\begin{align*}
    \cV^n : \{0,1,\dots,n\} \times \R^n \to \R, \qquad  \cV^{\infty} : [0,1] \times \cP_2(\R) \to \R.
\end{align*}
We start with $\cV^n$. For technical reasons, we will first define $\cV^n$ through a dynamic programming equation, and then verify below that it is the value function of the vector balancing problem \eqref{def.vn.intro}, regardless of the particular filtered probability space used to set up this optimization problem. More precisely, we define $\cV^n$ via the recursion
\begin{align}
    \cV^n(k,\bx) &= \E\Big[ \min \Big\{ \cV^n\big(k+1,\bx + \bZ \big), \cV^n\big(k+1, \bx -  \bZ \big) \Big\} \Big], \quad \bZ \sim \nu_{1/n}^{\otimes n}  \label{def.vn.dpp}   \\
    &= \int_{\R^n} \min \big\{ \cV^n\big(k+1,\bx + \bz \big), \cV^n\big(k+1, \bx - \bz \big) \big\}\, \nu_{1/n}^{\otimes n} (d\bz), \quad \text{for $k = 0,\dots,n-1$}, \nonumber
\end{align}
together with the terminal condition
\begin{align} \label{def.vn.term}
    \cV^n(n,\bx) = | \bx |_{\infty} = \max_{i = 1,\dots,n} |x_i|.
\end{align}

We next turn to the definition of $\cV^{\infty}$. We begin with a PDE formulation, which might be considered the \emph{Eulerian} or \emph{closed-loop} formulation; the control is chosen as a function of time and the current state. In contrast, below we will show that this is equivalent to a \emph{Lagrangian} or \emph{open-loop} formulation, in which controls are adapted processes on a given filtered probability space. For each $(t_0,m_0) \in [0,1] \times \cP_2(\R)$, we denote by $\cA(t_0,m_0)$ the set of all pairs $(m, \alpha)$ consisting of a (continuous) curve $[t_0,1] \ni t \mapsto m_t \in \cP_2(\R)$ and a measurable function $\alpha : [t_0,1] \times \R \to \R$, satisfying the Fokker--Planck equation
\begin{align}\label{eq.prelim.vinf.fpe}
    \partial_t m = \frac{1}{2}\partial_{xx} m - \partial_x ( m \alpha ), \quad (t,x) \in (t_0,1) \times \R, \quad m_{t_0} = m_0
\end{align}
in the sense of distributions, as well as the constraint\footnote{It makes no difference if this constraint \eqref{eq.prelim.vinf.l2constraint} is imposed for \emph{every} $t$ or just for \emph{almost every} $t$. In the latter case, we could simply modify $\alpha$ on a null set such that \eqref{eq.prelim.vinf.l2constraint} holds for every $t$, and the Fokker--Planck equation \eqref{eq.prelim.vinf.fpe} still holds (in the sense of distributions) for the same $(m_t)_{t_0 \le t \le 1}$.}
\begin{align}\label{eq.prelim.vinf.l2constraint}
    \int_{\R} \alpha(t,x)^2\, m_t(dx) \leq \frac{2}{\pi}, \qquad t_0 \leq t \leq 1,\text{ a.e.}
\end{align}
We then define
\begin{equation}
    \cV^{\infty}(t_0,m_0) \coloneqq \inf \Big\{[ m_1 ]_{\infty}: (m,\alpha) \in \cA(t_0, m_0)\Big\}, \label{eq:Vinf-originaldef}
\end{equation}
where for $\mu \in \cP(\R)$, we set
\begin{align*}
    [\mu]_{\infty} \coloneqq \inf\big\{ r > 0 : \text{supp}(\mu) \subset [-r,r] \big\} = \inf \big\{ r > 0 : \mu\big([-r,r]^c\big) = 0\big\}.
\end{align*}
While we introduce the value functions in order to benefit from dynamic programming, we are primarily interested in the values $\cV^n(0,\bzero)$ and $\cV^{\infty}(0,\delta_0)$, which correspond to the numbers $V^n$ and $V^\infty$ defined in \eqref{def.vn.intro} and \eqref{def.vinf.intro}, respectively. Indeed, it is explained in Lemmas \ref{lem.dpp} and \ref{lem.equivstrong} below that
\[
V^n = \cV^n(0,\bzero), \qquad V^{\infty} = \cV^{\infty}(0,\delta_0).
\]
Here, $\bzero$ denotes the zero vector in $\R^n$.

\subsection{First properties of \protect\texorpdfstring{$\cV^n$}{Vn}}

We now state and prove two important lemmas about the value functions $\cV^n$. The first states the $\cV^n$ inherits a Lipschitz bound from its terminal condition.

\begin{lemma} \label{lem.vn.1lip}
    For each $n \in \N$ and $k \in \{0,1,\dots,n\}$, $\cV^n(k,\cdot)$ is $1$-Lipschitz with respect to $|\cdot|_{\infty}$, i.e., for all $\bx,\by\in\R^n$
    \begin{align*}
        |\cV^n(k,\bx) - \cV^n(k,\by)| \leq |\bx - \by|_{\infty}.
    \end{align*}
\end{lemma}

\begin{proof}
    Obviously $\cV^n(n,\cdot) = |\cdot|_{\infty}$ is 1-Lipschitz with respect to $|\cdot|_{\infty}$. Suppose that $\cV^n(k,\cdot)$ is 1-Lipschitz for some $k \in \{1,\dots,n\}$. Then, from \eqref{def.vn.dpp}, we have
    \begin{align*}
        &\cV^n(k-1,\bx) - \cV^n(k-1,\by)
        \\
        &\quad = \int_{\R^n} \Big(  \min \big\{ \cV^n(k,\bx + \bz), \cV^n(k, \bx - \bz) \big\} - \min \big\{ \cV^n(k,\by + \bz), \cV^n(k, \by - \bz) \big\} \Big)\, \nu_{1/n}^{\otimes n}(d\bz)
        \\
        &\quad \leq \int_{\R^n} \max\Big\{ \big| \cV^n(k,\bx + \bz) - \cV^n(k,\by + \bz)\big|, \, \big| \cV^n(k,\bx - \bz) - \cV^n(k,\by - \bz)\big| \Big\} \, \nu_{1/n}^{\otimes n}(d\bz)
        \\
        &\quad \leq \int_{\R^n} |\bx - \by|_{\infty}\, \nu_{1/n}^{\otimes n}(d\bz) = |\bx - \by|_{\infty}.
    \end{align*}
    The proof is completed by induction.
\end{proof}
Next, we explain that $\cV^n$ represents the value function of \eqref{def.vn.intro}, regardless of the exact filtered probability space on which this problem is defined.

\begin{lemma} \label{lem.dpp}
     Fix integers $0\leq k_0 < k_1 \leq n$ and $\bx \in \R^n$. Let $(\Omega, \cF, \bP)$ be a probability space, and $\bbG = (\cG(k))_{k = k_0,\dots,k_1}$ be a filtration of $\cF$. Furthermore, let $(\bZ(k))_{k = k_0+1,\dots,k_1}$ be a sequence of $\bbG$-adapted, $\R^n$-valued random vectors such that $\bZ(k) \sim \nu_{1/n}^{\otimes n}$ is independent of $\cG(k-1)$ for all $k = k_0+1,\dots,k_1$.
    Then, we have
    \begin{align*}
        \cV^n(k_0,\bx) =  \inf_{\eps(k_0+1),\dots,\eps(k_1)} \E\bigg[ \cV^n\Big(k_1, \bx + \sum_{k = k_0+1}^{k_1} \eps(k) \bZ(k)\Big) \bigg],
    \end{align*}
    where the infimum is taken over all $\bbG$-adapted $\{\pm 1\}$-valued processes $(\eps(k))_{k = k_0+1,\dots,k_1}$. In particular, if $k_1 = n$, we have
    \begin{align*}
    \cV^n(k_0,\bx) = \inf_{\eps(k_0+1),\dots,\eps(n)} \E\,\Big| \bx + \sum_{k = k_0+1}^n \eps(k) \bZ(k) \Big|_{\infty},
    \end{align*}
    and thus, $V^n = \cV^n(0,\bzero)$, where $V^n$ was defined in \eqref{def.vn.intro}.
\end{lemma}

\begin{proof}
    Let $(\eps(k))_{k = k_0+1,\dots,k_1}$ be an arbitrary $\bbG$-adapted $\{\pm 1\}$-valued process. Define an $\R^n$-valued process $(\bY(k))_{k = k_0,\dots,k_1}$ by
    \begin{align*}
       \bY(k) = \bx + \sum_{\ell = k_0+1}^k \eps(\ell) \bZ(\ell), \qquad k = k_0,\dots,k_1.
    \end{align*}
    We claim that for each $k = k_0,\dots,k_1-1$,
    \begin{align} \label{inductive}
        \E\big[ \cV^n\big(k, \bY(k)\big) \big] \leq \E\big[ \cV^n\big(k + 1, \bY(k + 1)\big) \big].
    \end{align}
    Indeed, using the independence of $\bZ(k+1)$ from $\cG(k)$, together with the dynamic programming relation \eqref{def.vn.dpp}, we have
    \begin{align} \label{comp.verification}
        \E\big[ \cV^n\big(k, \bY(k) \big) \big]
        &=
        \E\Big[ \int_{\R^n} \Big(\min_{\eta\in\{\pm 1\}} \cV^n\big(k+1,\bY(k) + \eta\bz \big)\Big)\, \nu_{1/n}^{\otimes n}(d\bz) \Big]
       \nonumber \\
        &= \E\Big[ \E\Big[ \min_{\eta\in\{\pm 1\}} \cV^n\big(k+1,\bY(k) + \eta\bZ(k+1) \big) \, \Big|\, \cG(k) \Big] \Big]
      \nonumber   \\
        &= \E\Big[ \min_{\eta\in\{\pm 1\}} \cV^n\big(k+1,\bY(k) + \eta\bZ(k+1) \big) \Big]
      \nonumber   \\
        &\leq \E\big[ \cV^n\big(k+1, \bY(k+1) \big) \big].
    \end{align}
    Iterating \eqref{inductive}, we obtain
    \begin{align*}
        \cV^n(k_0,\bx_0) = \E\big[ \cV^n\big(k_0,\bY(k_0)\big) \big] \leq \E\big[  \cV^n\big(k_1, \bY(k_1) \big) \big],
    \end{align*}
    and so taking an infimum over $\eps$ and recalling the definition of $\bY$, we get
    \begin{align} \label{ver.upperbound}
         \cV^n(k_0,\bx) \leq \inf_{\eps(k_0+1),\dots,\eps(k_1)}  \E\bigg[ \cV^n\Big(k_1, \bx + \sum_{k = k_0+1}^{k_1} \eps(k) \bZ(k) \Big)\bigg].
    \end{align}
    On the other hand, define for $k \in \{k_0+1,\dots,k_1\}$ and $\by, \bz \in \R^n$
    \begin{equation*}
        \hat{\eps}(k,\by, \bz) =
        \begin{cases}
            1 &\text{if } \cV^n( k, \by + \bz) \leq \cV^n( k, \by - \bz), \\
            -1 &\text{if } \cV^n( k, \by + \bz) > \cV^n( k, \by - \bz).
        \end{cases}
    \end{equation*}
    Then, $\hat{\eps}(k,\cdot,\cdot) : \R^n \times \R^n \to \{ \pm 1 \}$ is measurable and
    \begin{align*}
        \hat{\eps}(k,\by, \bz)  \in  \mathrm{argmin} \Big( \{\pm 1\} \ni \eta \mapsto \cV^n\big( k, \by + \eta \bz \big)  \Big).
    \end{align*}
    Set $\bY^*(k_0) = \bx$, and define $(\bY^*(k))_{k = k_0+1,\dots,k_1}$ and $(\eps^*(k))_{k = k_0+1,\dots,k_1}$ inductively via
    \begin{align*}
       \eps^*(k+1) = \hat{\eps}\big(k+1, \bY^*(k), \bZ(k+1) \big), \quad \bY^*(k+1) = \bY^*(k) + \eps^*(k+1) \bZ(k+1).
    \end{align*}
    Repeating the computations in \eqref{comp.verification} shows that this particular control satisfies
    \begin{align*}
        \E\big[ \cV^n\big(k, \bY^*(k)\big) \big] = \E\big[ \cV^n\big(k + 1, \bY^*(k + 1)\big) \big], \qquad k = k_0,\dots,k_1 - 1,
    \end{align*}
    and iterating this equality gives
     \begin{align} \label{ver.lowerbound}
         \cV^n(k_0,\bx) = \E\bigg[ \cV^n\Big(k_1, \bx + \sum_{k = k_0 + 1}^{k_1} \eps^*(k) \bZ(k) \Big)\bigg].
    \end{align}
    Combining \eqref{ver.upperbound} and \eqref{ver.lowerbound} completes the proof.
\end{proof}

\subsection{First properties of \protect\texorpdfstring{$\cV^{\infty}$}{V-infinity}}

We now turn to some key properties of $\cV^{\infty}$. First, it is worth recording the fact that $\cV^{\infty}$ is locally bounded on (i.e., bounded on each compact subset of) $[0,1) \times \cP_2(\R)$, a fact which is not obvious at first glance. This is a consequence of Proposition \ref{prop.ulequp} below, so we do not give a proof now.
\begin{proposition} \label{prop.upperbound.uinf}
    The function $\cV^\infty$ is locally bounded on $[0,1) \times \cP_2(\R)$.
\end{proposition}

The following lemma shows, as was previously announced, that the PDE definition of $\cV^\infty$ admits an equivalent probabilistic formulation. This is an instance of a folklore result in stochastic control theory, that the optimal value is typically the same over \emph{open-loop} versus \emph{Markovian} controls. In other terms, the Eulerian and Lagrangian formulations are equivalent.

\begin{lemma} \label{lem.equivstrong}
   Fix $(t_0,m_0) \in [0,1) \times \cP_2(\R)$. Let $\big(\Omega, \cF, \bbF = (\cF(t))_{0 \leq t \leq 1}, \mathbb{P}\big)$ be any filtered probability space satisfying the usual conditions, which supports an $\cF(t_0)$-measurable random $X_0$ with law $m_0$, and an $\bbF$-Brownian motion $(B(t))_{0\leq t \leq 1}$. Then, we have
   \begin{align}
       \cV^{\infty}(t_0,m_0) = \inf_{\alpha} \Big\| X_0 + \int_{t_0}^1 \alpha(t) \, dt + B(1) - B(t_0) \Big\|_{\infty},
   \end{align}
   where the infimum is taken over all $\bbF$-progressive processes $\alpha$ satisfying the constraint
   \begin{align} \label{alpha.l2constraint}
       \E \, \alpha(t)^2 \leq 2/\pi, \quad \text{ for a.e.\ } t_0 \leq t \leq 1.
   \end{align}
   In particular, $V^{\infty} = \cV^{\infty}(0,\delta_0)$, where $V^{\infty}$ was defined  in \eqref{def.vinf.intro}.
\end{lemma}
\begin{proof}
    First, given a process $\alpha$ satisfying \eqref{alpha.l2constraint}, set $X(t) = X_0 + \int_{t_0}^t \alpha(s)\,ds + B(t) - B(t_0)$. Then we can argue as in \cite[Proposition 5.1]{brunickshreve} to show that there is a measurable function $\hat{\alpha} : [t_0,1] \times \R \to \R$ such that
    \begin{equation}\label{eq.equivstrong.defhatalpha}
        \hat{\alpha}(t,X(t)) = \E[ \alpha(t) \,|\, X(t)],\qquad\text{a.s.\ for a.e.\ } t.
    \end{equation}
    We claim that the pair $(m,\hat{\alpha})$ with $m_t\coloneqq\cL(X(t))$ lies in $\cA(t_0,m_0)$. First, the fact that $(m,\hat{\alpha})$ is a weak solution to the Fokker--Planck equation \eqref{eq.prelim.vinf.fpe} is a simple consequence of It\^{o}'s formula and \eqref{eq.equivstrong.defhatalpha}. Indeed, for any smooth test function $\phi:\R\to\R$ and all $t\in[0,1]$
    \begin{align*}
        \int_\R \phi\,dm_t = \E \, \phi(X(t))
        &= \int_0^t \E\Big[\alpha(s)\phi'(X(s)) + \frac{1}{2} \phi''(X(s))\Big]\,ds \\
        &= \int_0^t \E\Big[\hat{\alpha}(s,X(s))\phi'(X(s)) + \frac{1}{2} \phi''(X(s))\Big]\,ds \\
        &= \int_0^t \int_\R \Big(\hat{\alpha}(s,x) \phi'(x) + \frac{1}{2}\phi''(x)\Big)\,m_s(dx)\,ds.
    \end{align*}
    The constraint \eqref{eq.prelim.vinf.l2constraint} is also satisfied, since for any non-negative $\phi\in C[0,1]$, Jensen's inequality gives
    \begin{align*}
        \int_0^1 \phi(t)\int_\R \hat{\alpha}(t,x)^2\, m_t(dx)\,dt
        \leq \int_0^1 \phi(t) \E\big[\alpha(t)^2\big]\,dt \leq \frac{2}{\pi}\int_0^1\phi(t)\,dt.
    \end{align*}
    As a result, $(m,\hat{\alpha})$ lies in $\cA(t_0, m_0)$, and thus
    \begin{align*}
        \cV^{\infty}(t_0,m_0) \leq [m_1]_{\infty} = \| X(1) \|_{\infty}.
    \end{align*}
    Taking an infimum over $\alpha$ shows that
    \begin{align}\label{eq.equivstrong.toshow.leq}
       \cV^{\infty}(t_0,m_0) \leq \inf_{\alpha} \Big\| X_0 + \int_{t_0}^1 \alpha(t) \, dt + B(1) - B(t_0) \Big\|_{\infty},
   \end{align}
   with the infimum over $\bbF$-progressive processes $\alpha$ satisfying \eqref{alpha.l2constraint}.

   On the other hand, let $(m,\hat{\alpha}) \in \cA(t_0, m_0)$. Then, by the superposition principle  \cite[Theorem 2.5]{trevisan2016}, there exists a filtered probability space $(\Omega', \cF', \bbF', \bP')$ hosting a Brownian motion $B'$, together with a continuous, $\bbF'$-adapted process $(X'(t))_{t_0\leq t \leq 1}$ satisfying
   \begin{align*}
       dX'(t) = \hat{\alpha}\big(t,X'(t)\big) \, dt + dB'(t), \quad \cL\big(X'(t)\big) = m_{t}, \ \text{ for all } t \in [t_0,1].
   \end{align*}
   Now define processes $\alpha'$ and $\beta'$ via
   \begin{align*}
        \alpha'(t) = \hat{\alpha}\big( t,X'(t) \big), \ \beta'(t) = \E\Big[ \alpha'(t) \, \Big| \, \cF^{X'(t_0), B'}(t) \Big], \ \cF^{X'(t_0), B'}(t) = \sigma\big(X'(t_0), (B'(s))_{t_0 \leq s \leq t}\big).
   \end{align*}
   We note that \cite[Proposition 5.1]{brunickshreve} implies that it is possible to choose a progressively measurable version of $\beta'$, which we fix here. Moreover, the $L^2$ constraint on $\hat\alpha$ carries over to $\beta'$ in the sense of \eqref{alpha.l2constraint}.
   Noting that $\beta'(s) = \E[ \alpha'(s) \, | \, \cF^{X'(t_0), B'}(t) ]$ a.s.\ for $s < t$, we have
   \begin{equation*}
       X'(t_0) + \int_{t_0}^1 \beta'(t) \, dt + B'(1) - B'(t_0) = \E\Big[ X'(1) \,\Big|\, \cF^{X'(t_0), B'}(1) \Big],
   \end{equation*}
   and thus
    \begin{equation*}
      \Big\| X'(t_0) + \int_{t_0}^1 \beta'(t) \, dt + B'(1) - B'(t_0) \Big\|_{L^{\infty}(\bP')} \leq \| X'(1) \|_{L^{\infty}(\bP')} \leq [m_1]_\infty.
   \end{equation*}
   Now, since $\beta'$ is measurable with respect to the filtration generated by $X'(t_0)$ and $B'$, it follows that we can find a process $\alpha$ defined on $(\Omega, \cF, \bbF, \bP)$ such that the laws of $\big( X_0, \alpha, B\big)$ and $\big(X'(t_0), \beta', B'\big)$ under $\bP$ and $\bP'$, respectively, coincide. In particular,
   \begin{align*}
        \Big\| X_0 + \int_{t_0}^1 \alpha(t) \, dt + B(1) - B(t_0) \Big\|_{L^{\infty}(\bP)} &=  \Big\| X'(t_0) + \int_{t_0}^1 \beta'(t) \, dt + B'(1) - B'(t_0) \Big\|_{L^{\infty}(\bP')}
        \\
        &\leq [m_1]_{\infty}.
   \end{align*}
   Taking an infimum over $(m,\hat{\alpha}) \in \cA(t_0, m_0)$, we obtain
   \begin{equation*}
       \cV^{\infty}(t_0,m_0) \geq \inf_{\alpha} \Big\| X_0 + \int_{t_0}^1 \alpha(t) \, dt + B(1) - B(t_0) \Big\|_{L^{\infty}(\bP)},
   \end{equation*}
   which together with \eqref{eq.equivstrong.toshow.leq} completes the proof.
\end{proof}

A final fact about $\cV^\infty$ that we will need is a time-monotonicity property. It will not be used until a final step of the proof of the upper bound in Section~\ref{sec:upper-bound}.

\begin{lemma} \label{lem.nondecreasing}
The map $t \mapsto \cV^{\infty}(t,\delta_0)$ is nonincreasing on $[0,1]$.
\end{lemma}

To prove this lemma, we will use two other facts  that are worth recording separately.
The first states that $\cV^{\infty}$ satisfies a dynamic programming principle. This is classical because the control problem is deterministic, and so the proof is omitted.

\begin{lemma} \label{lem.vinf.dpp}
    For each $t_0 \in [0,1)$, $h \in (0,1 - t_0]$, we have
    \begin{align*}
        \cV^{\infty}(t_0,m_0) = \inf_{(m,\alpha) \in \cA(t_0, m_0)} \cV^{\infty}(t_0+h,m_{t_0+h}).
    \end{align*}
\end{lemma}

Finally, we will need to know that $\cV^{\infty}$ has a convexity property.

\begin{lemma}
    \label{lem.vinf.convex}
    For each $t_0 \in [0,1]$, $\cV^{\infty}(t_0,\cdot)$ is $L$-convex, meaning that for any square-integrable random variables $X_0,X_0'$ defined on some common probability space $(\Omega, \cF, \bP)$, we have
    \begin{align*}
        \cV^{\infty}\Big(t_0, \cL\big(\lambda X_0 + (1-\lambda) X'_0 \big) \Big) \leq \lambda \cV^{\infty}\big(t_0,\cL(X_0)\big) + (1-\lambda) \cV^{\infty}\big(t_0,\cL(X'_0)\big),
    \end{align*}
    for $0 \leq \lambda \leq 1$. Moreover, $\cV^{\infty}(t_0,\cdot)$ is even, in the sense that
    \begin{align*}
        \cV^{\infty}(t_0, \cL(X_0)) = \cV^{\infty}\big(t_0,\cL(-X_0)\big).
    \end{align*}
    As a consequence, $\cV^{\infty}(t_0,\cdot)$ obtains a global minimum at $\delta_0$.
\end{lemma}
\begin{proof}
   It is clear that $\cV^{\infty}(1,\cdot) = [\,\cdot\,]_{\infty}$ is $L$-convex. Fix $t_0 \in [0,1)$ and square-integrable random variables $X_0,X'_0$ on some probability space $(\Omega, \cF, \bP)$. By enlarging the probability space if necessary, we can assume that $(\Omega, \cF, \bP)$ also supports a filtration $\bbF = (\cF(t))_{0 \leq t \leq 1}$ and an $\bbF$-Brownian motion $B = (B(t))_{0 \leq t \leq 1}$ such that $X_0$ and $X_0'$ are $\cF(0)$-measurable. By Lemma \ref{lem.equivstrong}, for any $\eps > 0$ we can find $\bbF$-progressively measurable processes $\alpha$ and $\alpha'$ such that, setting
    \begin{align*}
        X(t) = X_0 + \int_{t_0}^t \alpha(s)\, ds + B(t) - B(t_0), \quad X'(t) = X_0' + \int_{t_0}^t \alpha'(s)\, ds + B(t) - B(t_0),
    \end{align*}
    for $t_0 \leq t \leq 1$, we have
    \begin{align*}
        \|X(1)\|_{\infty} \leq \cV^{\infty}(t_0,m_0) + \eps, \quad \|X'(1)\|_{\infty} \leq \cV^{\infty}(t_0,m_0') + \eps.
    \end{align*}
    Now fix $\lambda \in (0,1)$, and set $X^{\lambda}(t) = \lambda X(t) + (1-\lambda) X'(t)$ and $m_0^\lambda = \mathcal{L}(\lambda X_0 + (1-\lambda)X_0')$. By convexity, the process $\alpha^{\lambda} = \lambda \alpha + (1-\lambda) \alpha'$ satisfies the $L^2$-constraint \eqref{alpha.l2constraint}, and so using Lemma~\ref{lem.equivstrong}
    \begin{align*}
        \cV^{\infty}\big(t_0, m_0^{\lambda}) &\leq \|X^{\lambda}(1)\|_{\infty} = \| \lambda X(1) + (1-\lambda) X'(1) \|_{\infty}
        \\
        &\leq \lambda \|X(1)\|_{\infty} + (1-\lambda)  \|X'(1)\|_{\infty} \leq \lambda \cV^{\infty}(t_0,m_0) + (1-\lambda) \cV^{\infty}(t_0,m_0') + 2\eps.
    \end{align*}
    Sending $\eps \to 0$ completes the proof that $\cV^{\infty}(t_0,\cdot)$ is convex.

    The fact that $\cV^{\infty}\big(t_0, \cL(X_0)) = \cV^{\infty}(t_0,\cL(-X_0))$ can again be proved using the probabilistic formulation appearing in Lemma \ref{lem.equivstrong}.  In particular, given $(t_0,m_0)$, we first choose a filtered probability space $(\Omega, \cF, \bbF, \bP)$, a Brownian motion $B$ and a random variable $X_0$ as in the statement of that lemma. Noting that $B' = -B$ is also a Brownian motion, we can apply Lemma \ref{lem.equivstrong} and make the change of variables $\beta = - \alpha$
    \begin{align*}
      \cV^{\infty}(t_0,\cL(X_0)) &= \inf_{\|\alpha(t)\|_2^2 \leq 2/\pi}  \Big\| X_0 + \int_{t_0}^1 \alpha(t) \, dt + B(1) - B(t_0) \Big\|_{\infty}
      \\
      &= \inf_{\|\beta(t)\|_2^2 \leq 2/\pi}  \Big\| - X_0 + \int_{t_0}^1 \beta(t) \, dt + B'(1) - B'(t_0) \Big\|_{\infty} = \cV^{\infty}\big(t_0,\cL(-X_0)\big).
    \end{align*}
    Finally, to prove the last claim that $\cV^{\infty}(t_0,\cdot)$ is minimized at $\delta_0$, we fix $m_0 \in \cP_2(\R)$ and a random variable $X_0$ with $\cL(X_0) = m_0$, and then use symmetry and convexity to conclude that
    \begin{align*}
        \cV^{\infty}(t_0,m_0) &= \cV^{\infty}\big(t_0,\cL(X_0)\big) = \frac{1}{2} \cV^{\infty}\big(t_0, \cL(X_0) \big) + \frac{1}{2} \cV^{\infty}\big(t_0,\cL(-X_0)\big)
        \\
        &\geq \cV^{\infty}\Big(t_0, \cL\Big(\frac{1}{2} X_0 - \frac{1}{2} X_0\Big) \Big) = \cV^{\infty}(t_0,\delta_0). \qedhere
    \end{align*}
\end{proof}

\begin{proof}[Proof of Lemma \ref{lem.nondecreasing}]
    Fix $0\leq s < t \leq 1$ and $\eps>0$. Let $(m,\alpha) \in \cA(s,\delta_0)$ be an $\eps$-optimizer for the problem defining $\cV^{\infty}(s,\delta_0)$. Then,
    \begin{align*}
        \cV^{\infty}(s,\delta_0) \geq [m_1]_\infty - \eps \geq \cV^{\infty}(t, m_t) - \eps.
    \end{align*}
    By Lemma \ref{lem.vinf.convex}, $\cV^{\infty}(t,\cdot)$ obtains a global minimum at $\delta_0$, so $\cV^{\infty}(t,m_t) \geq \cV^{\infty}(t,\delta_0)$. Letting $\eps\to0$ proves $\cV^\infty(s,\delta_0) \geq \cV^\infty(t,\delta_0)$.
\end{proof}

\subsection{On the convexity and monotonicity of \protect\texorpdfstring{$V^\infty_T$}{V-T-infinity} in time}

By the same arguments as in Lemma~\ref{lem.nondecreasing}, the values $V_T^\infty$ defined in \eqref{def.vTinf} are nondecreasing in $T$; recall that the $T$ in $V_T^{\infty}$ is the time horizon, whereas the time horizon corresponding to $\cV^\infty(t,\delta_0)$ is $1-t$. We record one more interesting result about the time-dependence of $V_T^{\infty}$, though it is not used anywhere else in the paper. 

\begin{proposition} \label{pr:scaling-dec+convex}
The function $r \mapsto \frac{1}{r}V_{r^2}^{\infty}$ is nonincreasing and convex on $(0,\infty)$. In particular, $T \mapsto V^\infty_T$ is continuous.
\end{proposition}
\begin{proof}
The continuity follows from convexity and the fact that $V^\infty_T < \infty$ for all $T>0$, as shown below in Proposition~\ref{prop.firstupperbound}.
Let $h(r) = \frac{1}{r}V_{r^2}^{\infty}$.
Start from the definition \eqref{def.vTinf} of $V_T^{\infty}$.
By scaling time via $t=sT$ and using Brownian scaling,
we have
\begin{align*}
h(\sqrt{T})=\frac{1}{\sqrt{T}}V_T^{\infty} &= \inf \bigg\{ \Big\| \sqrt{T}\int_0^1 \alpha(sT) \, ds + \frac{1}{\sqrt{T}}B(T) \Big\|_{\infty} : \sup_{0 \le t \le T}\|\alpha(t)\|_2 \leq \sqrt{\frac{2}{\pi}} \bigg\} \\
	&= \inf \bigg\{ \Big\|  \int_0^1 \alpha(s) \, ds + B(1) \Big\|_{\infty} : \sup_{0\leq s \leq 1} \|\alpha(s)\|_2 \leq \sqrt{\frac{2T}{\pi}} \bigg\}.
\end{align*}
This is clearly nonincreasing in $T$. To prove convexity, fix $r_1,r_2 > 0$ and $0 < u < 1$, and set $r\coloneqq ur_1+(1-u)r_2$. For $i=1,2$ let $\alpha_i$ be optimal for the above control problem with $T$ replaced by $r_i^2$; in particular, $\|\alpha_i(s)\|_2 \le r_i\sqrt{2/\pi}$ for all $s$. Then $\alpha \coloneqq u\alpha_1+(1-u)\alpha_2$ satisfies $\|\alpha(s)\|_2 \le r\sqrt{2/\pi}$  for all $s$, and by convexity of the $L^\infty$ norm we deduce
\begin{align*}
h(r) &\le \Big\|  \int_0^1 \alpha(s) \, ds + B(1) \Big\|_{\infty} \\
	&\le u\,\Big\|  \int_0^1 \alpha_1(s) \, ds + B(1) \Big\|_{\infty} + (1-u)\Big\|  \int_0^1 \alpha_2(s) \, ds + B(1) \Big\|_{\infty} \\
	&= u h(r_1) + (1-u)h(r_2). \qedhere
\end{align*}
\end{proof}

\section{Lower Bound}
\label{sec:lower-bound}

In this section, we prove the lower bound of Theorem \ref{th:intro} for a general coordinate distribution $\nu$ satisfying \eqref{nu.conditions}. That is, we show that
\begin{equation}
V^\infty \le \liminf_{n\to\infty} V^n. \label{pf:lowerbound-final}
\end{equation}
For the remainder of this section, let us define the control problems corresponding to the $V^n$'s on a common probability space. Let $(\Omega,\mathcal{F},\mathbb{P})$ be a fixed probability space that supports an array of i.i.d.\ random variables $(\zeta_i(k))_{i,k\in\N}$ with $\zeta_{i}(k)\sim \nu$. For $1\leq i,k\leq n$, we set
\begin{gather*}
    Z_i^n(k) \coloneqq n^{-1/2} \zeta_{i}(k),\quad \bm{Z}^n(k) \coloneqq  (Z_1^n(k),\dots,Z_n^n(k)),\\
    \G^n(k) \coloneqq  \sigma\big(\bm{Z}^n(\ell): 1\leq \ell \leq k\big),\quad \G^n(0) \coloneqq \{\emptyset, \Omega\}.
\end{gather*}
In particular, $\sqrt{n}\bm{Z}^n(1),\dots, \sqrt{n}\bm{Z}^n(n)$ are independent random vectors with distribution $\nu^{\otimes n}$. Denote $\mathbb{G}^n=(\cG^n(k))_{k=0,1,\dots,n}$. This setup satisfies the assumptions of Lemma~\ref{lem.dpp}, and so
\begin{equation*}
    V^n = \inf_{\epsilon^n(1),\dots,\epsilon^n(n)} \E\Big|\sum_{k=1}^n \epsilon^n(k) \bm{Z}^n(k)\Big|_{\infty}
\end{equation*}
where the infimum is taken over all $n$-state controls, i.e., $\mathbb{G}^n$-adapted $\{\pm 1\}$-valued processes $\epsilon^n=(\epsilon^n(k))_{k=1,\dots,n}$. We recall the ``weak'' formulation of the control problem $V^\infty$, discussed already in the introduction. We let the set of ``weak'' controls $\mathcal{A}^\mathrm{w}$ be given by all probability measures $m\in\P(\C^2)$ such that, under $m$, the canonical processes $(A,B)$ on $\cC^2$ satisfy \eqref{eq:thm:limit-points:ac0-paths}, \eqref{eq:thm:limit-points:l2-constraint}, and \eqref{eq:thm:limit-points:bm}.
Then, by Lemma~\ref{lem.equivstrong}, we have
\begin{equation*}
	V^\infty = \inf_{m \in \mathcal{A}^\mathrm{w}} \|A(1) + B(1)\|_{L^\infty(m)}.
\end{equation*}

\subsection{Doob decomposition} For the rest of Section~\ref{sec:lower-bound}, we fix a sequence $(\eps^n)_{n\in\N}$ such that $\eps^n$ is an $n$-state control, and we consider the corresponding state processes
\begin{equation*}
	\bm{Y}^n(k) \coloneqq \sum_{\ell=1}^k \eps^n(\ell)\bZ^n(\ell),\qquad k=0,1,\dots,n.
\end{equation*}
The goal is to show
\begin{equation*}
	V^\infty \leq \liminf_{n\to\infty} \E|\bY^n(n)|_\infty.
\end{equation*}

For each $n$, we start by embedding the state processes $\bm{Y}^n$ in the time interval $[0,1]$ by time rescaling and linear interpolation:
\begin{alignat*}{2}
	\hat{\bm{Y}}^n(\tfrac{k}{n})&\coloneqq \bm{Y}^n(k),&&k=0,1,\dots,n,\\
	\hat{\bm{Y}}^n(t)&\coloneqq \hat{\bm{Y}}^n(\tfrac{k}{n}) + (\bm{Y}^n(k+1)-\bm{Y}^n(k))(nt-k),\qquad &&k < nt \leq k+1.
\end{alignat*}
We decompose the continuous-time process $\hat{\bm{Y}}^n=(\hat{\bm{Y}}^n(t))_{t\in[0,1]}$ into two terms, $\bm{A}^n$ and $\bm{M}^n$, based on the Doob decomposition of $\bm{Y}^n$:
\begin{alignat*}{2}
	\bm{A}^n(\tfrac{k}{n})&\coloneqq \sum_{\ell=1}^k \bm{b}^n(\ell-1),\qquad &&k=0,1,\dots,n,\\
	\bm{A}^n(t)&\coloneqq \bm{A}^n(\tfrac{k}{n}) + \bm{b}^n(k)(nt-k),\qquad &&k < nt \leq k+1,\\
	\bm{M}^n(\tfrac{k}{n})&\coloneqq \sum_{\ell=1}^k \bm{\sigma}^n(\ell),&&k=0,1,\dots,n,\\
	\bm{M}^n(t)&\coloneqq \bm{M}^n(\tfrac{k}{n}) + \bm{\sigma}^n(k+1)(nt-k),\qquad &&k < nt \leq k+1,
\end{alignat*}
where for $\ell=1,\dots,n$
\begin{align}
	\bm{b}^n(\ell-1) &\coloneqq \E[\epsilon^n(\ell)\bm{Z}^n(\ell)\,|\,\mathcal{G}^n(\ell-1)],\label{eq:discrete-drift}\\
	\bm{\sigma}^n(\ell) &\coloneqq \epsilon^n(\ell)\bm{Z}^n(\ell) - \E[\epsilon^n(\ell)\bm{Z}^n(\ell)\,|\,\mathcal{G}^n(\ell-1)]\label{eq:discrete-mart-incr}.
\end{align}
We call $\bm{A}^n$ the drift term and $\bm{M}^n$ the martingale term. Note that the processes $\hat{\bm{Y}}^n$, $\bm{A}^n$, and $\bm{M}^n$ are continuous and take values in $\R^n$.

We begin with two elementary estimates on the drift and martingale increments.

\begin{lemma}\label{lemma:drift-vector:a-priori-bound}
    For any $k=0,\dots,n-1$, we have $|\bm{b}^n(k)|_2 \leq 1/\sqrt{n}$. If $\nu = \mathcal{N}(0,1)$, then $|\bm{b}^n(k)|_2 \leq \sqrt{2/(\pi n)}$.
\end{lemma}
\begin{proof}
    Take a unit vector $\bm{u}\in\R^n$. Since $\bm{Z}^n(k+1)$ is independent of $\G^n(k)$,
    \begin{equation*}
        \bm{u}\cdot \bm{b}^n(k) = \E[\epsilon^n(k+1) \bm{u}\cdot \bm{Z}^n(k+1)\,|\,\G^n(k)] \leq \E|\bm{u}\cdot\bm{Z}^n(k+1)|.
    \end{equation*}
    Using Cauchy--Schwarz and $\mathrm{Cov}(\bm{Z}^n(k+1)) = n^{-1} I_n$, where $I_n$ is the identity matrix, this quantity is bounded by $1/\sqrt{n}$. In the case $\nu = \mathcal{N}(0,1)$, the random variable $\bm{u}\cdot \bm{Z}^n(k+1)$ is a standard Gaussian, and so the right-hand side of the inequality above is equal to $\sqrt{2/(\pi n)}$. Taking the supremum over $\bm{u}$ concludes the proof.
\end{proof}

\begin{lemma}\label{lemma:sigma-moments}
	Let $k=1,\dots,n$.
	\begin{enumerate}
		\item We have
		\begin{equation*}
			\mathrm{Cov}(\bm{\sigma}^n(k)\,|\,\G^n(k-1)) = \frac{1}{n} I_n - \bm{b}^n(k-1)\bm{b}^n(k-1)^\top.
		\end{equation*}
		In particular, $\mathrm{Cov}(\bm{\sigma}^n(k)\,|\,\G^n(k-1))\leq n^{-1} I_n$ in semidefinite order.
		\item For every $i=1,\dots,n$,
		\begin{equation}\label{eq:lemma:sigma-fourth-moment}
			\E[\sigma_i^n(k)^4\,|\,\G^n(k-1)] \leq \frac{16}{n^2}\int x^4\,\nu(dx).
		\end{equation}
	\end{enumerate}
\end{lemma}

\begin{proof}
	For the first part, recall that $\bm{\sigma}^n(k) = \epsilon^n(k) \bm{Z}^n(k) - \bm{b}^n(k-1)$ and $\bm{b}^n(k-1) = \E[\epsilon^n(k)\bm{Z}^n(k)\,|\,\G^n(k)]$. Thus,
	\begin{align*}
		\mathrm{Cov}(\bm{\sigma}^n(k)\,|\,\G^n(k-1)) &= \E\big[\epsilon^n(k)^2 \bm{Z}^n(k)\bm{Z}^n(k)^\top\,\big|\,\G^n(k-1)\big] - \bm{b}^n(k-1)\bm{b}^n(k-1)^\top \\
	&= \frac{1}{n} I_n - \bm{b}^n(k-1)\bm{b}^n(k-1)^\top.
	\end{align*}
	For the second part, use Jensen's inequality and the independence of $\bm{Z}^n(k)$ and $\G^n(k-1)$ to obtain
	\begin{align*}
		\sigma_i^n(k)^4 &= \big(\epsilon^n(k) Z_i^n(k) - \E[\epsilon^n(k) Z_i^n(k)\,|\,\G^n(k-1)]\big)^4\\
		&\leq 8\big(Z_i^n(k)^4 + \E[\epsilon^n(k)Z_i^n(k)\,|\,\G^n(k-1)]^4\big) \leq 8\big(Z_i^n(k)^4 + \E[Z_i^n(k)^4]\big)
	\end{align*}
	Taking conditional expectations, we get $\E[\sigma_i^n(k)^4\,|\,\G^n(k-1)]\leq 16\E[Z_i^n(k)^4]$.
\end{proof}

\subsection{Tightness of empirical measures}
\label{sec:lower-bound:tightness}

We define the joint empirical distribution $\hat{\mu}_n$ of the coordinates of $\bm{A}^n$ and $\bm{M}^n$ by
\begin{equation*}
	\hat{\mu}_n(S)\coloneqq \frac{1}{n}\sum_{i=1}^{n} \delta_{(A_i^n,M_i^n)}(S),\qquad S\in\B\big(\C^2\big).
\end{equation*}
Then, $\hat{\mu}_n$ is a random element of $\mathcal{P}(\C^2)$. We define the corresponding marginals by
\begin{equation*}
	\hat{\mu}_n^A \coloneqq \frac{1}{n}\sum_{i=1}^n \delta_{A^n_i},\qquad \hat{\mu}_n^M \coloneqq \frac{1}{n}\sum_{i=1}^n \delta_{M^n_i}.
\end{equation*}
In this section, we show that the laws $\{\cL(\hat{\mu}_n):n\in\N\}$ of the empirical measures form a tight sequence in $\P(\P(\C^2))$. To do this, we will look at the marginals separately and prove that the sequences $\{\cL(\hat{\mu}_n^A):n\in\N\}$ and $\{\cL(\hat{\mu}_n^M):n\in\N\}$ are each tight in $\P(\P(\C))$.

Let us denote by $H_0^1\subset\C$ the space of absolutely continuous functions $\phi:[0,1]\to\R$ with $\phi(0)=0$ and $\int_0^1 \dot{\phi}(t)^2\,dt <\infty$.

\begin{lemma}\label{lemma:tightness-drift}
	For each $n$, the empirical measure $\hat{\mu}_n^A$ is $\PP$-almost surely contained in the compact set
	\begin{equation*}
		\mathcal{K}_1 \coloneqq \bigg\{\mu\in\P(\C): \mu(H_0^1)=1,\ \esssup_{t\in[0,1]} \int_{H_0^1} \dot{\phi}(t)^2\,\mu(d\phi) \leq 1\bigg\}.
	\end{equation*}
	In particular, the sequence $\big\{\cL(\hat{\mu}_n^A): n\in\N\big\}\subset \P(\P(\C))$ is tight.
	If $\nu=\mathcal{N}(0,1)$, the statement holds true if we replace the upper bound $1$ in the definition of $\mathcal{K}_1$ by $2/\pi$.
\end{lemma}

\begin{proof}
	By definition, the processes $A_i^n$ are absolutely continuous with $A_i^n(0)=0$ and $\dot{A}_i^n(t)=n b_i^n(\floor{nt})$, and we have $\PP$-almost surely for a.e.\ $t\in[0,1]$
	\begin{equation*}
		\frac{1}{n} \sum_{i=1}^n |\dot{A}_i^n(t)|^2 = \frac{1}{n} \sum_{i=1}^n |n b_i^n(\floor{nt})|^2 = n |\bm{b}^n(\floor{nt})|_2^2.
	\end{equation*}
	By Lemma~\ref{lemma:drift-vector:a-priori-bound}, the right-hand side is bounded by 1, and it follows that $\PP$-almost surely, the empirical measure $\hat{\mu}_n^A$ is contained in $\mathcal{K}_1$. That $\mathcal{K}_1$ is compact is essentially a consequence of the Arzel\`a--Ascoli theorem; see Lemma~\ref{lemma:compactness-K} in the appendix for details. The results for the special case $\nu=\mathcal{N}(0,1)$ follow from those of Lemma~\ref{lemma:drift-vector:a-priori-bound}.
\end{proof}

\begin{lemma}\label{lemma:tightness-mart}
	The sequence $\big\{\cL(\hat{\mu}_n^M): n\in\N\big\}\subset \P(\P(\C))$ is tight.
\end{lemma}
\begin{proof}
	For each $n$, we consider the \emph{mean measure} $m_n\in \P(\C)$ of $\cL(\hat{\mu}_n^M)$, defined by
	\begin{equation*}
		m_n(S) \coloneqq \E\big[\hat{\mu}_n^M(S)\big] = \frac{1}{n}\sum_{i=1}^n \PP[M_i^n\in S],\qquad S \in \mathcal{B}(\C).
	\end{equation*}
	Note in particular that if $U_n$ is uniformly distributed on $\{1,\dots,n\}$ and independent of $\{\bm{Z}^n(k): 1 \leq k \leq n\}$, then $m_n$ is the law of $M_{U_n}^n$ as a random element in $\C$. By extending our probability space if necessary, we can assume that such random variables $U_n$ exist.

	By a standard result, the tightness of $\{\cL(\hat{\mu}_n^M): n\in\N\}$ in $\P(\P(\C))$ is equivalent to the tightness of $\{m_n :n\in\N\}$ in $\P(\C)$; see, e.g., \cite[Theorem 2.11]{budhiraja-dupuis}. We will prove the latter statement by showing that $m_n$ converges to Wiener measure in $\P(\C)$, using the martingale central limit theorem in the form of \cite[Theorem 18.2]{billingsley-convergence}. Setting $\xi_{nk}=\sigma_{U_n}^n(k)$ and $s_{nk}^2 = \mathrm{Var}(\xi_{nk}\,|\,U_n, \G^n(k-1))$, the process $M_{U_n}^n$ converges in law to Wiener measure, provided that the following two conditions hold as $n\to\infty$:\footnote{To be exact, \cite[Theorem 18.2]{billingsley-convergence} states that under \eqref{eq:lemma:tightness-mart:cond-1} and \eqref{eq:lemma:tightness-mart:cond-2}, the càdlàg process $\sum_{k=1}^{\floor{nt}} \xi_{nk}$ converges in distribution to Brownian motion in the Skorohod space $D[0,1]$. Since Brownian motion is continuous, \cite[Theorem 13.4]{billingsley-convergence} then implies that $\sup_t |M_{U_n}^n(t) - \sum_{k=1}^{\floor{nt}} \xi_{nk} | \leq \max_k |\xi_{nk}|$ converges to zero in distribution. Together, this implies that $M_{U_n}^n$ converges in distribution to Brownian motion in $\C$.}
	\begin{alignat}{2}
		\sum_{k=1}^{\floor{nt}} s_{nk}^2 \longrightarrow t,\qquad&\text{in probability for all $t\in[0,1]$,}\label{eq:lemma:tightness-mart:cond-1}\\
		\sum_{k=1}^{\floor{nt}} \E[\xi_{nk}^2;\,|\xi_{nk}|\geq \delta]\longrightarrow 0,\qquad&\text{for all $t\in[0,1]$ and $\delta>0$.} \label{eq:lemma:tightness-mart:cond-2}
	\end{alignat}

	To see \eqref{eq:lemma:tightness-mart:cond-1}, note that $s_{nk}^2= n^{-1} - b_{U_n}^n(k-1)^2$ by the first part of Lemma~\ref{lemma:sigma-moments} and therefore
	\begin{equation*}
		\sum_{k=1}^{\floor{nt}} s_{nk}^2 = \frac{\floor{nt}}{n} - \sum_{k=1}^{\floor{nt}} b_{U_n}^n(k-1)^2.
	\end{equation*}
	The first term converges (surely) to $t$ as $n\to\infty$, and the second term converges to zero in probability since
	\begin{equation*}
		\sum_{k=1}^{\floor{nt}} \E\big[b_{U_n}^n(k-1)^2\big] = \frac{1}{n}\sum_{k=1}^{\floor{nt}}\E|\bm{b}^n(k-1)|_2^2 \leq \frac{1}{n}
	\end{equation*}
	by Lemma~\ref{lemma:drift-vector:a-priori-bound}.

	For \eqref{eq:lemma:tightness-mart:cond-2}, observe that
	\begin{equation*}
		|\sigma_i^n(k)| \leq |Z_i^n(k)| + \E\big[ |Z_i^n(k)| \,\big|\,\G^n(k-1)\big] \leq |Z_i^n(k)| + \frac{1}{\sqrt{n}},
	\end{equation*}
	and thus
	\begin{align*}
		\sum_{k=1}^{\floor{nt}}\E[\xi_{nk}^2;\,|\xi_{nk}|\geq \delta]
		&= \frac{1}{n} \sum_{k=1}^{\floor{nt}}\sum_{i=1}^n \E[|\sigma_i^n(k)|^2;\,|\sigma_i^n(k)|\geq \delta] \\
		&\leq \frac{2}{n} \sum_{k=1}^{\floor{nt}}\sum_{i=1}^n \E\bigg[|Z_i^n(k)|^2 + \frac{1}{n}; |Z_i^n(k)| \geq \delta - \frac{1}{\sqrt{n}}\bigg] \\
		&= \frac{2\floor{nt}}{n}\int_{\{|z|\geq \delta\sqrt{n} - 1\}} (z^2 + 1)\,\nu(dz)\longrightarrow 0.
	\end{align*}
\end{proof}

\color{black}

The previous lemmas show that  $\{\cL(\hat{\mu}_n^A):n\in\N\}$ and $\{\cL(\hat{\mu}_n^M):n\in\N\}$ are each tight sequences in $\P(\P(\C))$. As noted in the proof of Lemma \ref{lemma:tightness-mart}, this is equivalent to the tightness of the corresponding sequences of mean measures in $\P(\C)$, which are the marginals of the mean measures associated with $\big\{\cL(\hat{\mu}_n): n\in\N\big\}$. The latter sequence is thus tight in $\P(\P(\C \times \C))$.

\subsection{Characterization of limit points}
We have shown that $\{\cL(\hat{\mu}_n):n\in\N\}$ forms a tight sequence in $\P(\P(\C^2))$, and we now analyze its limit points. Let us start with the drift term.
\begin{proposition}\label{prop:limit-points-drift}
    Any limit point $Q\in\P(\P(\C))$ of the sequence $\big\{\cL(\hat{\mu}_n^A): n\in\N\big\}$ assigns probability 1 to the set
	\begin{equation*}
		\mathcal{K}_{2/\pi} \coloneqq \bigg\{\mu\in\P(\C): \mu(H_0^1)=1,\ \esssup_{t\in[0,1]} \int_{H_0^1} \dot{\phi}(t)^2\,\mu(d\phi) \leq \frac{2}{\pi} \bigg\}.
	\end{equation*}
\end{proposition}
In the Gaussian setting, $\nu=\mathcal{N}(0,1)$, the proposition follows immediately from Lemma~\ref{lemma:tightness-drift} as in this case the laws $\cL(\hat{\mu}_n^A)$ are concentrated on $\mathcal{K}_{2/\pi}$. In the general case, however, we can only conclude from Lemma~\ref{lemma:tightness-drift} that $Q[\mathcal{K}_1] = 1$. Using a central-limit-type argument, we now show that the improved bound $2/\pi$ is recovered in the limit.

We first prove a lemma for a fixed realization of the process $\bm{A}^n=(\bm{A}^n(t))_{t\in[0,1]}$. Let us denote such a path by $\bm{a}^n=(\bm{a}^n(t))_{t\in[0,1]}$. We will obtain Proposition~\ref{prop:limit-points-drift} as a corollary using Skorohod representation.

\begin{lemma}\label{lemma:limit-points-drift:fixed-path}
	For each $n$, let $\bm{a}^n:[0,1]\to\R^n$ be such that $a_i^n\in H_0^1$ and for almost every $t\in[0,1]$
	\begin{equation}\label{eq:lemma:limit-points-drift:fixed-path:assumption}
		n^{-1}\dot{\bm{a}}^n(t) \in \big\{\E[g(\bm{Z}^n) \bm{Z}^n] : \sqrt{n}\bm{Z}^n\sim \nu^{\otimes n},\ g:\R^n\to\{\pm 1\}\text{ measurable}\big\}.
	\end{equation}
	Then, any limit point $\mu\in\P(\C)$ of the sequence
    \begin{equation*}
        \mu_n = \frac{1}{n} \sum_{i=1}^n \delta_{a_i^n}\in\P(\C)
    \end{equation*}
    is contained in $\mathcal{K}_{2/\pi}$.
\end{lemma}

\begin{proof}
	Note first that by the same arguments as in Lemma~\ref{lemma:drift-vector:a-priori-bound}, $\mu_n\in\mathcal{K}_1$ for every $n$. Let us assume that $\mu$ is the limit of some subsequence $(\mu_n)_{n\in J}$ for $J\subset \N$. Since $\mathcal{K}_1$ is compact (and thereby closed), we know that $\mu\in \mathcal{K}_1$. To prove that $\mu\in\mathcal{K}_{2/\pi}$, it suffices to show that for any fixed $h:(0,1)\to\R$ with $h\geq 0$ and $\int_0^1 h(t)\,dt=1$
    \begin{equation}\label{eq:drift-limit-points:proof:ts}
        \int_0^1 h(t) \int_{H_0^1} \dot{\phi}(t)^2\,\mu(d\phi)\,dt \leq \frac{2}{\pi}.
    \end{equation}

    To start, let us define the modified paths
    \begin{equation*}
        \phi_i^n(t) \coloneqq \int_0^t \dot{a}_i^n(s)\mathbf{1}_{\{|\dot{a}_i^n(s)|\leq \log n\}}\,ds,\qquad t\in[0,1].
    \end{equation*}
    Then, $|\dot{\phi}_i^n(t)|\leq \log n$ and $(\dot{\phi}_i^n)^2=\dot{\phi}_i^n\dot{a}_i^n$, and along $n\in J$, the empirical law $\bar{\mu}_n=n^{-1}\sum_{i}\delta_{\phi_i^n}$ converges to the same limit $\mu$ since
    \begin{align*}
        \frac{1}{n}\sum_{i=1}^n \sup_{t\in[0,1]}|\phi_i^n(t) - a_i^n(t)| &\leq \frac{1}{n}\sum_{i=1}^n \int_0^1 |\dot{a}_i^n(t)|\mathbf{1}_{\{|\dot{a}_i^n(t)| > \log n\}}\,dt \\
        &\leq \frac{1}{\log n} \frac{1}{n}\sum_{i=1}^n \int_0^1 |\dot{a}_i^n(t)|^2\,dt  \leq \frac{1}{\log n}\longrightarrow 0,
    \end{align*}
    where the last inequality uses $\mu_n\in\mathcal{K}_1$.

    The expression on the left-hand side in \eqref{eq:drift-limit-points:proof:ts} is lower semicontinuous in $\mu$, and it is therefore enough to prove that
    \begin{equation}\label{eq:drift-limit-points:proof:ts-reduction}
        \liminf_{n\in J} r_n \leq \sqrt{2/\pi},\qquad \text{ where } r_n \coloneqq \bigg(\frac{1}{n}\sum_{i=1}^n \int_0^1 h(t) (\dot{\phi}_i^n(t))^2\,dt\bigg)^{1/2}.
    \end{equation}
    The claim is trivial if $\liminf_{n\in J} r_n = 0$, and so we may assume that $r_n \geq \delta > 0$ for all $n\in J$ and some $\delta >0$. Using $(\dot{\phi}_i^n)^2=\dot{\phi}_i^n\dot{a}_i^n$ and \eqref{eq:lemma:limit-points-drift:fixed-path:assumption}, we have
    \begin{equation}\label{eq:drift-limit-points:proof:rn-bound-1}
        r_n = \frac{1}{r_n}\int_0^1 h(t) \Big(\frac{1}{n}\sum_{i=1}^n \dot{\phi}_i^n(t) \dot{a}_i^n(t)\Big)\,dt \leq \frac{1}{r_n} \int_0^1 h(t)\, \E\Big|\sum_{i=1}^n \dot{\phi}_i^n(t) Z_i^n\Big|\,dt
    \end{equation}
    where $\bm{Z}^n$ is a random vector with $\sqrt{n}\bm{Z}^n\sim\nu^{\otimes n}$. For a fixed $\eta>0$, we define
    \begin{equation*}
        w_i^n(t) \coloneqq \frac{\dot{\phi}_i^n(t)}{\big(\eta^2 r_n^2 + n^{-1}\sum_{j=1}^n (\dot{\phi}_j^n(t))^2\big)^{1/2}},\qquad t\in[0,1].
    \end{equation*}
    Rewriting the right-hand side in \eqref{eq:drift-limit-points:proof:rn-bound-1} in terms of $w_i^n(t)$ and using Cauchy--Schwarz gives
    \begin{align}
        r_n
		&\leq \frac{1}{r_n} \bigg(\int_0^1 h(t) \Big(\E\Big|\sum_{i=1}^n w_i^n(t) Z_i^n\Big|\Big)^2\,dt\bigg)^{1/2}\bigg(\int_0^1 h(t) \Big(\eta^2 r_n^2 + \frac{1}{n}\sum_{j=1}^n (\dot{\phi}_j^n(t))^2\Big)\,dt\bigg)^{1/2} \notag\\
		&=\sqrt{1 + \eta^2} \bigg(\int_0^1 h(t)\Big(\E\Big|\sum_{i=1}^n w_i^n(t) Z_i^n\Big|\Big)^2\,dt\bigg)^{1/2}. \label{eq:drift-limit-points:proof:rn-bound-2}
    \end{align}

    We now reduce the proof to the following pointwise estimate:
	\begin{equation}\label{eq:drift-limit-points:proof:pointwise-estimate}
		\limsup_{n\in J}\E\Big|\sum_{i=1}^n w_i^n(t)Z_i^n\Big|
		\leq \sqrt{\frac{2}{\pi}}
		\qquad\text{for a.e.\ } t\in[0,1].
	\end{equation}
	Indeed, assume this estimate and take the $\limsup$ in \eqref{eq:drift-limit-points:proof:rn-bound-2}. The bound $\E|\sum_i w_i^n(t)Z_i^n| \leq 1$ gives the domination needed for the reverse Fatou lemma, yielding
	\begin{equation*}
		\limsup_{n\in J} r_n \leq \sqrt{1+\eta^2} \bigg(\int_0^1 h(t) \frac{2}{\pi} \,dt\bigg)^{1/2}
		= \sqrt{1+\eta^2} \sqrt{\frac{2}{\pi}}.
	\end{equation*}
	Letting $\eta\to0$ then gives \eqref{eq:drift-limit-points:proof:ts-reduction} and completes the proof.

    Let us prove the estimate \eqref{eq:drift-limit-points:proof:pointwise-estimate}. By construction of $\phi_i^n$ and $w_i^n$, together with the assumption that $r_n\geq\delta>0$, the set $\{t\in[0,1]: |w_i^n(t)|\leq(\eta\delta)^{-1}\log n\ \forall i,n\}$ has Lebesgue measure 1. Fix such a $t$, and set
	\begin{equation*}
		\Phi_n \coloneqq \sum_{i=1}^n w_i^n(t)Z_i^n.
	\end{equation*}
	Then $\Phi_n$ is a sum of independent mean-zero random variables, and
	\begin{equation*}
		\E\,\Phi_n^2
		= \frac{1}{n}\sum_{i=1}^n (w_i^n(t))^2
		= \frac{n^{-1}\sum_{i=1}^n(\dot{\phi}_i^n(t))^2}
		{\eta^2r_n^2+n^{-1}\sum_{i=1}^n(\dot{\phi}_i^n(t))^2}
		\leq 1.
	\end{equation*}
	Choose a subsequence $J'\subset J$ such that
	\begin{equation*}
		\lim_{n\in J'}\E|\Phi_n|
		= \limsup_{n\in J}\E|\Phi_n|,
		\qquad
		\lim_{n\in J'}\E\,\Phi_n^2 = v
	\end{equation*}
	for some $v\in[0,1]$. If $v=0$, then Cauchy--Schwarz gives $\lim_{n\in J'}\E|\Phi_n|=0$ and the bound in \eqref{eq:drift-limit-points:proof:pointwise-estimate} holds trivially. So let us assume $v>0$. We claim that $\Phi_n$ converges in law to $\mathcal{N}(0,v)$ along $n\in J'$. Since
	the variables $\Phi_n$ are bounded in $L^2$, this would imply
	\begin{equation*}
		\lim_{n\in J'} \E|\Phi_n|
		= \sqrt{\frac{2v}{\pi}}
		\leq \sqrt{\frac{2}{\pi}}
	\end{equation*}
	which proves \eqref{eq:drift-limit-points:proof:pointwise-estimate}. To apply the Lindeberg central limit theorem, it only remains to verify the Lindeberg condition. For $\epsilon>0$
    \begin{align*}
        \sum_{i=1}^n \E\big[(w_i^n(t) Z_i^n)^2;\, |w_i^n(t) Z_i^n|\geq \epsilon\big]
        &= \sum_{i=1}^n (w_i^n(t))^2\, \E[(Z_i^n)^2; |w_i^n(t) Z_i^n|\geq \epsilon ]\\
        &= \frac{1}{n}\sum_{i=1}^n (w_i^n(t))^2\, \psi\Big(\frac{|w_i^n(t)|}{\epsilon \sqrt{n}}\Big)
    \end{align*}
    where, for $\zeta\sim\nu$, we set $\psi(y)\coloneqq \E[\zeta^2;\,|\zeta|\geq y^{-1}]$ when $y>0$, and $\psi(0)\coloneqq 0$. Since $|w_i^n(t)|\leq (\eta\delta)^{-1} \log n$ and $n^{-1}\sum_i(w_i^n(t))^2\leq1$, the quantity above is bounded by $\psi(\log n/(\eta \delta \epsilon \sqrt{n}))$. This converges to zero as $n\to\infty$, and the Lindeberg condition follows.
\end{proof}

\begin{proof}[Proof of Proposition~\ref{prop:limit-points-drift}]
	Write $Q_n=\cL(\hat{\mu}_n^A)$ and let $Q$ be the limit of some subsequence $(Q_n)_{n\in J}$, $J\subset \N$. By the Skorohod representation theorem \cite[Theorem II.86.1]{rogers-williams-i}, we can extend our probability space $(\Omega,\mathcal{F},\mathbb{P})$ if necessary to find $\P(\C)$-valued random elements $(m_n)_{n\in J}$ and $m$ such that
	\begin{equation*}
		\cL(m_n) = Q_n,\quad \cL(m) = Q,\quad m_n\to m\text{ a.s.}
	\end{equation*}
	For each $n$, since $\cL(m_n)=\cL(\hat{\mu}_n^A)$, we may find a continuous $\R^n$-valued process $\tilde{\bm{A}}^{n}=(\tilde{\bm{A}}^n(t))_{t\in[0,1]}$ such that $m_n = n^{-1} \sum_{i=1}^n \delta_{\tilde{A}_i^n}$ and such that the unordered sets $\{A^n_1,\dots,A^n_n\}$ and $\{\tilde{A}^n_1,\dots,\tilde{A}^n_n\}$ have the same law. The assumption \eqref{eq:lemma:limit-points-drift:fixed-path:assumption} of Lemma~\ref{lemma:limit-points-drift:fixed-path} is invariant under coordinate permutations and it therefore carries over from the paths of $\bm{A}^n$ to those of $\tilde{\bm{A}}^n$. In particular, for almost every $\omega$, \eqref{eq:lemma:limit-points-drift:fixed-path:assumption} holds for $\bm{a}^n(t)=\tilde{\bm{A}}^n(t,\omega)$, and $m_n(\omega)\to m(\omega)$ weakly. For all such $\omega$, we obtain $m(\omega)\in \mathcal{K}_{2/\pi}$, and thus $Q[\mathcal{K}_{2/\pi}] = \PP[m\in \mathcal{K}_{2/\pi}] = 1$.
\end{proof}

\begin{theorem}\label{thm:limit-points}
	Any limit point $Q$ of the sequence $\big\{\cL(\hat{\mu}_n): n\in\N\big\}\subset\P(\P(\C^2))$ is a relaxed control for the problem $V^\infty$, in the sense that $Q[\mathcal{A}^\mathrm{w}] = 1$.
\end{theorem}

\begin{proof}
	Let $Q$ be a limit point of $\big\{\cL(\hat{\mu}_n): n\in\N\big\}$. Consider the canonical space $\C^2$, with $A$ and $B$ denoting the canonical processes and $\mathbb{F}^{A,B}$ denoting the filtration generated by $(A,B)$. By the definition of $\mathcal{A}^\mathrm{w}$, we need to verify that for $Q$-almost every $m\in\P(\C^2)$, under $m$
	\begin{gather*}
		A=(A(t))_{t\in[0,1]} \text{ has absolutely continuous paths and } A(0)=0,\\
		\E\,\dot{A}(t)^2 \leq \frac{2}{\pi},\quad \text{for a.e.\ } t\in[0,1],\\
		B=(B(t))_{t\in[0,1]} \text{ is an $\mathbb{F}^{A,B}$-Brownian motion}.
	\end{gather*}
	The first two properties follow immediately from Proposition~\ref{prop:limit-points-drift}. It remains to prove that the third property holds as well.\par

	We use the following version of Lévy's characterization of Brownian motion. For $0\leq s<t\leq 1$, $\phi\in C_c^\infty(\R)$, and $h\in C_b(C[0,s])$, we define a map $G_{s,t,\phi,h}:\P(\C^2)\to \R$ by
	\begin{equation*}
		G_{s,t,\phi,h}(m) \coloneqq \int_{\C^2} h(a|_{[0,s]},x|_{[0,s]})\left(\phi(x(t)) - \phi(x(s)) - \frac{1}{2}\int_s^t \phi''(x(u))\,du\right)\, m(da\,dx)
	\end{equation*}
	Then, $B$ is an $\mathbb{F}^{A,B}$-Brownian motion under a fixed $m$ if and only if $G_{s,t,\phi,h}(m)=0$ for all choices of $(s,t,\phi,h)$. In particular, this holds true for $Q$-almost every $m$ if and only if $\int |G_{s,t,\phi,h}(m)|\,Q(dm)=0$ for all $(s,t,\phi,h)$.\footnote{Here, we use that we can restrict our attention to a suitable countable and dense set of $(s,t,\phi,h)$.} Since $|G_{s,t,\phi,h}|$ is a bounded continuous function on $\P(\C^2)$ and $\cL(\hat{\mu}_n)$ converges to $Q$ along some subsequence, it suffices to show that as $n\to\infty$
	\begin{equation*}
		\E\big|G_{s,t,\phi,h}(\hat{\mu}_n)\big| = \int |G_{s,t,\phi,h}|\,d\cL(\hat{\mu}_n) \to 0,\quad\text{ for all } (s,t,\phi,h).
	\end{equation*}

	So let us fix $(s,t,\phi,h)$ and write $G_n=G_{s,t,\phi,h}(\hat{\mu}_n)$. Then, by the definition of $\hat{\mu}_n$,
	\begin{equation*}
		G_n = \frac{1}{n}\sum_{i=1}^n h_i^n \bigg(\phi\big(M_i^n(t)\big) - \phi\big(M_i^n(s)\big) - \frac{1}{2}\int_s^t \phi''\big(M_i^n(u)\big)\,du\bigg),
	\end{equation*}
	where $h_i^n\coloneqq h(A_i^n|_{[0,s]}, M_i^n|_{[0,s]})$. Write $s_n = (\floor{ns} + 1)/n$ and $t_n = \floor{nt}/n$, and let $n$ be large enough so that $s\leq s_n < t_n \leq t$. We also define
	\begin{align*}
		G_n' &= \frac{1}{n}\sum_{i=1}^n h_i^n \bigg(\phi\big(M_i^n(t_n)\big) - \phi\big(M_i^n(s_n)\big) - \frac{1}{2}\int_{s_n}^{t_n} \phi''\big(M_i^n(u)\big)\,du\bigg).
	\end{align*}
	We prove that $\E|G_n-G_n'| \to 0$ and $\E|G_n'| \to 0$.

	To see that $\E|G_n-G_n'| \to 0$, note that for any $i$
	\begin{align*}
		&\bigg|\bigg(\phi(M_i^n(t)) - \phi(M_i^n(s)) - \frac{1}{2}\int_s^t \phi''(M_i^n(u))\,du \bigg) \\
		&\quad -\bigg(\phi(M_i^n(t_n)) - \phi(M_i^n(s_n)) - \frac{1}{2}\int_{s_n}^{t_n} \phi''(M_i^n(u))\,du\bigg)\bigg|\\
		&\leq \|\phi'\|_{\infty}\big(|M_i^n(t) - M_i^n(t_n)| + |M_i^n(s_n) - M_i^n(s)|\big) + \frac{1}{2}\|\phi''\|_{\infty}\big(|t-t_n| + |s_n - s|\big) \\
		&\leq \|\phi'\|_{\infty}\big(|\sigma_i^n(\floor{nt}+1)| + |\sigma_i^n(\floor{ns}+1)|\big) + \frac{\|\phi''\|_{\infty}}{n}.
	\end{align*}
	It follows that
	\begin{align*}
		\E|G_n - G_n'|
		&\leq \frac{\|h\|_{\infty}\|\phi'\|_{\infty}}{n} \sum_{i=1}^n \E\big[|\sigma_i^n(\floor{nt}+1)| + |\sigma_i^n(\floor{ns}+1)|\big] + \frac{\|h\|_{\infty}\|\phi''\|_{\infty}}{n} \longrightarrow 0
	\end{align*}
	since $\E|\sigma_i^n(k)| \leq n^{-1/2}$ for any $k$ by Lemma~\ref{lemma:sigma-moments}.

	We now show that $\E|G_n'| \to 0$. By Taylor's theorem, we have for $1\leq i,k\leq n$
	\begin{align*}
		\phi\big(M_i^n(\tfrac{k}{n})\big) - \phi\big(M_i^n(\tfrac{k-1}{n})\big)
		&= \phi'\big(M_i^n(\tfrac{k-1}{n})\big)\big(M_i^n(\tfrac{k}{n}) - M_i^n(\tfrac{k-1}{n})\big) \\
		&\qquad + \frac{1}{2} \phi''(\xi_{ik}^n) \big(M_i^n(\tfrac{k}{n}) - M_i^n(\tfrac{k-1}{n})\big)^2 \\
		&= \phi'\big(M_i^n(\tfrac{k-1}{n})\big)\sigma_i^n(k) + \frac{1}{2} \phi''(\xi_{ik}^n) \sigma_i^n(k)^2,
	\end{align*}
	where $\xi_{ik}^n$ is some intermediate value between $M_i^n(\tfrac{k}{n})$ and $M_i^n(\tfrac{k-1}{n})$. Summing over $k$ from $\floor{ns}+2$ to $\floor{nt}$ gives
	\begin{align*}
		\phi\big(M_i^n(t_n)\big) - \phi\big(M_i^n(s_n)\big)
		&= \sum_{k=\floor{ns}+2}^{\floor{nt}} \phi'\big(M_i^n(\tfrac{k-1}{n})\big)\sigma_i^n(k) + \frac{1}{2} \sum_{k=\floor{ns}+2}^{\floor{nt}}\phi''(\xi_{ik}^n) \sigma_i^n(k)^2.
	\end{align*}
	We can thus decompose $G_n'$ as follows:
	\begin{equation*}\label{eq:thm:limit-points:g-decomposition}
		G_n' = G_{n,1}' + G_{n,2}' + G_{n,3}' + G_{n,4}'
	\end{equation*}
	where
	\begin{align*}
		G_{n,1}' &\coloneqq \frac{1}{n}\sum_{i=1}^n h_i^n \Bigg(\sum_{k=\floor{ns}+2}^{\floor{nt}} \phi'\big(M_i^n(\tfrac{k-1}{n})\big)\sigma_i^n(k)\Bigg), \\
		G_{n,2}' &\coloneqq \frac{1}{2n} \sum_{i=1}^n h_i^n \Bigg(\sum_{k=\floor{ns}+2}^{\floor{nt}}\big(\phi''(\xi_{ik}^n)- \phi''\big(M_i^n(\tfrac{k-1}{n})\big)\big) \sigma_i^n(k)^2\Bigg), \\
		G_{n,3}' &\coloneqq \frac{1}{2n} \sum_{i=1}^n h_i^n \Bigg(\sum_{k=\floor{ns}+2}^{\floor{nt}}\phi''\big(M_i^n(\tfrac{k-1}{n})\big) \big(\sigma_i^n(k)^2 - 1/n\big)\Bigg), \\
		G_{n,4}' &\coloneqq \frac{1}{2n} \sum_{i=1}^n h_i^n \Bigg(\frac{1}{n}\sum_{k=\floor{ns}+2}^{\floor{nt}}\phi''\big(M_i^n(\tfrac{k-1}{n})\big) - \int_{s_n}^{t_n} \phi''\big(M_i^n(u)\big)\,du\Bigg).
	\end{align*}
	We will show that each of the terms converges to zero in $L^1(\mathbb{P})$.\smallskip

    \noindent
	\textit{The term $G_{n,1}'$:} Taking squares, we have
	\begin{equation*}
		(G_{n,1}')^2 = \frac{1}{n^2} \sum_{i,j=1}^n \sum_{k,\ell=\floor{ns}+2}^{\floor{nt}} g_{ik}^n g_{j\ell}^n \sigma_i^n(k)\sigma_j^n(\ell),
	\end{equation*}
	where $g_{ik}^n = h_i^n \phi'(M_i^n(\tfrac{k-1}{n}))$. Since $h_i^n$ is $\G^n(\floor{ns}+1)$-measurable, $g_{ik}^n$ is $\G^n(k-1)$-measurable, and for $k<\ell$, we get
	\begin{equation*}
		\E\big[g_{ik}^n g_{j\ell}^n \sigma_i^n(k)\sigma_j^n(\ell)\big] = \E\big[g_{ik}^n g_{j\ell}^n \sigma_i^n(k) \E\big[\sigma_j^n(\ell)\,|\,\G^n(\ell-1)\big]\big] = 0,
	\end{equation*}
	where we recall that $(\sigma_i^n(k))_{k=1,\dots,n}$ forms a martingale difference sequence in the filtration $(\G^n(k))_{k=0,\dots,n}$. Similarly, the terms with $k>\ell$ vanish in expectation, and we obtain
	\begin{equation*}
		\E[(G_{n,1}')^2] = \frac{1}{n^2} \sum_{i,j=1}^n \sum_{k=\floor{ns}+2}^{\floor{nt}} \E\big[g_{ik}^n g_{jk}^n \sigma_i^n(k)\sigma_j^n(k)\big].
	\end{equation*}
	Now, for each $k=\floor{ns}+2,\dots,\floor{nt}$,
	\begin{align*}
		\sum_{i,j=1}^n \E\big[g_{ik}^n g_{jk}^n \sigma_i^n(k)\sigma_j^n(k)\,\big|\,\G^n(k-1)\big]
		&= \sum_{i,j=1}^n g_{ik}^n g_{jk}^n \E\big[\sigma_i^n(k)\sigma_j^n(k)\,\big|\,\G^n(k-1)\big] \leq \frac{1}{n} \sum_{i=1}^n \big(g_{ik}^n\big)^2
	\end{align*}
	because $\Cov(\bm{\sigma}^n(k)\,|\,\G^n(k-1))\leq n^{-1} I_n$ in semidefinite order by Lemma~\ref{lemma:sigma-moments}. The right-hand side is bounded by $C=\|h\|_{\infty}^2 \|\phi'\|_{\infty}^2$, and we obtain
	\begin{equation*}
		\E\big[(G_{n,1}')^2\big] \leq \frac{C}{n^2}(\floor{nt} - \floor{ns} - 1) \longrightarrow 0.
	\end{equation*}\smallskip

    \noindent
	\textit{The term $G_{n,2}'$:} We have
	\begin{align*}
		|G_{n,2}'| &\leq \frac{1}{2n} \sum_{i=1}^n |h_i^n|\sum_{k=\floor{ns}+2}^{\floor{nt}}\big|\phi''(\xi_{ik}^n)- \phi''\big(M_i^n(\tfrac{k-1}{n})\big)\big| \sigma_i^n(k)^2 \\
		&\leq \frac{\|h\|_\infty \|\phi'''\|_\infty}{2n} \sum_{i=1}^n \sum_{k=\floor{ns}+2}^{\floor{nt}} \sigma_i^n(k)^4
	\end{align*}
	where we used that
	\begin{equation*}
		\big|\xi_{ik}^n - M_i^n(\tfrac{k-1}{n})\big| \leq \big|M_i^n(\tfrac{k}{n}) - M_i^n(\tfrac{k-1}{n})\big| = \sigma_i^n(k)^2.
	\end{equation*}
	By assumption, $\nu$ has a fourth moment, and so $\E[\sigma_i^n(k)^4] \leq C/n^2$ by Lemma~\ref{lemma:sigma-moments}. It follows that $\E|G_{n,2}'| \to 0$.\smallskip

	\noindent
	\textit{The term $G_{n,3}'$:}
	Taking squares, we have
	\begin{align*}
		(G_{n,3}')^2 = \frac{1}{n^2} \sum_{i,j=1}^n \sum_{k,\ell=\floor{ns} + 2}^{\floor{nt}} g_{ik}^n g_{j\ell}^n \big(\sigma_i^n(k)^2 - 1/n\big)\big(\sigma_j^n(\ell)^2 - 1/n\big)
	\end{align*}
	where we denote $g_{ik}^n = \tfrac{1}{2} h_i^n \phi''(M_i^n(\tfrac{k-1}{n}))$.
	Let us consider the summands for the cases $k=\ell$ and $k\neq \ell$. When $k=\ell$, note that
	\begin{align*}
		&\E\Big[\big(\sigma_i^n(k)^2 - 1/n\big)\big(\sigma_j^n(k)^2 - 1/n\big)\,\Big|\, \G^n(k-1)\Big] \\
		&\qquad= \mathrm{Cov}\big(\sigma_i^n(k)^2, \sigma_j^n(k)^2\,\big|\,\G^n(k-1)\big) + b_i^n(k-1)^2 b_j^n(k-1)^2
	\end{align*}
	since $\E\big[\sigma_i^n(k)^2 - 1/n\,\big|\,\G^n(k-1)\big] = -b_i^n(k-1)^2$ by the first part of Lemma~\ref{lemma:sigma-moments}. Thus,
	\begin{align*}
		&\hphantom{=}\frac{1}{n^2} \sum_{i,j=1}^n g_{ik} ^n g_{jk}^n \E\Big[\big(\sigma_i^n(k)^2 - 1/n\big)\big(\sigma_j^n(k)^2 - 1/n\big)\,\Big|\, \G^n(k-1)\Big] \\
		&= \frac{1}{n^2} \sum_{i,j=1}^n g_{ik} ^n g_{jk}^n \mathrm{Cov}\big(\sigma_i^n(k)^2, \sigma_j^n(k)^2\,\big|\,\G^n(k-1)\big) + \frac{1}{n^2}\bigg(\sum_{i=1}^n g_{ik}^n b_i^n(k-1)^2\bigg)^2.
	\end{align*}
    The final term is bounded by $C^2/n^4$ where  $C \ge \|h\|_{\infty}\|\phi''\|_{\infty} \ge |g^n_{ik}|$, because  $|\bm{b}^n(k-1)|_2^2 \leq 1/n$ by Lemma~\ref{lemma:drift-vector:a-priori-bound}. We also have
	\begin{equation*}
		\big|\Cov\big(\sigma_i^n(k)^2, \sigma_j^n(k)^2\,\big|\,\G^n(k-1)\big)\big| \leq \sqrt{\E[\sigma_i^n(k)^4]}\sqrt{ \E[\sigma_j^n(k)^4]} \leq \frac{C}{n^2}
	\end{equation*}
	by Lemma~\ref{lemma:sigma-moments} for some $C>0$. It follows that
	\begin{equation*}
		\frac{1}{n^2} \sum_{i,j=1}^n g_{ik} ^n g_{jk}^n \E\Big[\big(\sigma_i^n(k)^2 - 1/n\big)\big(\sigma_j^n(k)^2 - 1/n\big)\,\big|\, \G^n(k-1)\Big] \leq \frac{C}{n^2} + \frac{C}{n^4}
	\end{equation*}
	and in turn
	\begin{equation*}
		\E\bigg[\frac{1}{n^2} \sum_{i,j=1}^n \sum_{k=\floor{ns} + 2}^{\floor{nt}} g_{ik}^n g_{jk}^n \big(\sigma_i^n(k)^2 - 1/n\big)\big(\sigma_j^n(k)^2 - 1/n\big)\bigg] \longrightarrow 0.
	\end{equation*}
	It remains to show that the off-diagonal terms of $(G_{n,3}')^2$ vanish in expectation, i.e.,
	\begin{equation}
		\E\bigg[\frac{1}{n^2} \sum_{i,j=1}^n \sum_{\substack{k,\ell=\floor{ns} + 2 \\ k<\ell}}^{\floor{nt}} g_{ik}^n g_{j\ell}^n \big(\sigma_i^n(k)^2 - 1/n\big)\big(\sigma_j^n(\ell)^2 - 1/n\big)\bigg] \longrightarrow 0. \label{pf:Gn3'-1}
	\end{equation}
	For $k<\ell$, we have
	\begin{equation*}
		\E\Big[\big(\sigma_i^n(k)^2 - 1/n\big)\big(\sigma_j^n(\ell)^2 - 1/n\big)\,\big|\, \G^n(\ell-1)\Big] = - \big(\sigma_i^n(k)^2 - 1/n\big) b_j^n(\ell-1)^2,
	\end{equation*}
	and
	\begin{align*}
		-\frac{1}{n^2} \sum_{i,j=1}^n g_{ik}^n g_{j\ell}^n (\sigma_i^n(k)^2 - 1/n) b_j^n(\ell-1)
		&= -\frac{1}{n^2} \bigg(\sum_{i=1}^n g_{ik}^n(\sigma_i^n(k)^2 - 1/n)\bigg)\bigg(\sum_{j=1}^n g_{j\ell}^n b_j^n(\ell-1)^2\bigg) \\
		&\leq \frac{C}{n^3} \sum_{i=1}^n \big|\sigma_i^n(k)^2 - 1/n\big|,
	\end{align*}
	using again that $g$ is bounded and $|\bm{b}^n(\ell-1)|_2^2\leq 1/n$. This is bounded in expectation by $C/n^3$ because $\E\big|\sigma_i^n(k)^2 - 1/n\big|\leq 2/n$, and we deduce that the left-hand side of \eqref{pf:Gn3'-1} is bounded by $C/n$. We have thus shown that $\E[(G_{n,3}')^2]\to 0$.\smallskip

	\noindent
	\textit{The term $G_{n,4}'$:}
	We have
	\begin{align*}
		G_{n,4}'
		= \frac{1}{2n} \sum_{i=1}^n h_i^n \Bigg(\sum_{k=\floor{ns}+2}^{\floor{nt}}\int_{\frac{k-1}{n}}^{\frac{k}{n}} \phi''\big(M_i^n(\tfrac{k-1}{n})\big) - \phi''\big(M_i^n(u)\big)\,du\Bigg).
	\end{align*}
	For $(k-1)/n\leq u\leq k/n$, we have
	\begin{equation*}
		\E\big|\phi''\big(M_i^n(\tfrac{k-1}{n})\big) - \phi''\big(M_i^n(u)\big)\big| \leq \|\phi'''\|_{\infty} (nu+1-k)\E|\sigma_i^n(k)| \leq \|\phi'''\|_{\infty}C/\sqrt{n}
	\end{equation*}
	by the definition of $M_i^n(u)$ and Lemma~\ref{lemma:sigma-moments}. It easily follows that $\E|G_{n,4}'|\to 0$, completing the proof.
\end{proof}

\subsection{Completion of the proof}

We are now ready to prove the lower bound \eqref{pf:lowerbound-final} of Theorem~\ref{th:intro}.
	Consider the canonical space $\C^2$ of continuous paths $[0,1]\to\R^2$, with $A$ and $B$ denoting the canonical processes. We define a map $\mathcal{E}:\P(\C^2)\to[0,\infty]$ by sending $m\in\P(\C^2)$ to $\|A(1)+B(1)\|_{L^\infty(m)}=[\cL(A(1)+B(1))]_{\infty}$. The map $[\,\cdot\,]_{\infty} : \cP(\R) \to [0,\infty]$ is lower semicontinuous, as it can be written as the supremum of the lower semicontinuous functions $\mu\mapsto r\mathbf{1}_{A_r}(\mu)$, $r\geq 0$, where $A_r\coloneqq \{\mu\in\cP(\R): \mu([-r,r]^c) > 0\}$. Hence, the map $\mathcal{E}$ is lower semicontinuous, and in turn, $Q\mapsto \int \mathcal{E}\,dQ$ defines a lower semicontinuous map from $\P(\P(\C^2))$ to $[0,\infty]$.\par
	We have
	\begin{equation*}
		\E|\bm{Y}^n(n)|_\infty = \E\big[\mathcal{E}(\hat{\mu}_{n})\big] = \int \mathcal{E}\,d\cL(\hat{\mu}_{n}).
	\end{equation*}
	Now, to see that
	\begin{equation*}
		V^\infty \leq \liminf_{n\to\infty} \E|\bm{Y}^n(n)|_\infty,
	\end{equation*}
	we take a subsequence ${n_k}$ such that $\E|\bm{Y}^{n_k}(n_k)|_\infty$ converges to the right-hand side. As shown in Section~\ref{sec:lower-bound:tightness}, the sequence $\{\cL(\hat{\mu}_n):n\in\N\}$ is tight, and so we may assume without loss of generality that $\cL(\hat{\mu}_{n_k})$ converges to some $Q\in\P(\P(\C^2))$. By Theorem~\ref{thm:limit-points}, we have $Q[\mathcal{A}^\mathrm{w}]=1$, and so
	\begin{equation*}
		\lim_{k\to\infty} \E|\bm{Y}^{n_k}(n_k)|_\infty = \lim_{k\to\infty} \int_{\P(\C^2)} \mathcal{E}\,d\cL(\hat{\mu}_{n_k}) \geq \int_{\P(\C^2)} \mathcal{E} \,dQ \geq \inf_{m\in\mathcal{A}^\mathrm{w}} \mathcal{E}(m) = V^\infty.
	\end{equation*}
This completes the proof of \eqref{pf:lowerbound-final}.

\section{Upper Bound}
\label{sec:upper-bound}

In this section, we prove the upper bound of Theorem \ref{th:intro}, namely that
\begin{equation}
V^\infty \ge \limsup_{n\to\infty} V^n. \label{pf:upperbound-final}
\end{equation}
Throughout this section, the increment distribution appearing in the definition of $V^n$ is Gaussian; $\nu = \gamma$, and we will write $\cstar\coloneqq \sqrt{2/\pi}$ to ease notation.

\subsection{Growth bounds for \texorpdfstring{$\cV^n$}{Vn}} \label{sec.upperbounds}

The main goal of this section is to obtain an upper bound on the functions $\cV^n$ that is not necessarily sharp, but which is tractable and uniform in $n$. Ultimately, we will prove that for each $p \in (4,\infty)$, we have the bound
\begin{align}\label{vn.upperbound.intro}
    \cV^n(k,\bx) \leq \cU_p(k/n,m_{\bx}^n) + o_n(1),
\end{align}
where $m_{\bx}^n \coloneqq n^{-1}\sum_{i=1}^n \delta_{x_i}$, $\cU_p : [0,1] \times \cP_2(\R) \to \R$, and $o_n(1)$ denotes a constant which is independent of $(k,\bx)$ (but may depend on the parameter $p$), and vanishes as $n \to \infty$.

The function $\cU_p(t,m)$ is constructed from a specific admissible strategy for the continuous control problem $\cV^\infty(t,m)$. More precisely, instead of working with the original uniform $L^2$ bound $\cstar=\sqrt{2/\pi}$, we restrict attention to controls satisfying a stronger uniform $L^p$ bound for $p>2$, and we split the resulting $L^p$ budget equally between two parts.

One part of the budget is used in a deterministic way to shrink the initial law $m$ towards the origin. This motivates the quantity $\cU_p^{\mathrm{det}}(t,m)$, which measures the smallest essential supremum that can be achieved in this way. The other part of the budget is used in a construction based on the F\"ollmer drift; see Lemma \ref{lem.follmer}. We introduce a class of terminal laws that can be realized using F\"ollmer's construction, starting from the origin and with only half of the $L^p$ budget. We define \(\cU_p^{\mathrm{F}}(t)\) to be the smallest essential supremum among laws in this class. Proposition~\ref{prop.ulequp} shows that this budget-splitting strategy is admissible for the problem \(\cV^\infty(t,m)\), and therefore yields
\[
   \cV^\infty(t,m) \leq \cU_p(t,m) \coloneqq \cU_p^{\mathrm{det}}(t,m) + \cU_p^{\mathrm{F}}(t).
\]

The same function $\cU_p$ will then serve as an upper bound for $\cV^n$, up to an error term $o_n(1)$, in the sense of \eqref{vn.upperbound.intro}. This is established in Proposition~\ref{prop.vnupperbound}, and constitutes one of the main difficulties in the proof of \eqref{pf:upperbound-final}. In the next two subsections, we introduce $\cU_p^{\mathrm{det}}$ and $\cU_p^{\mathrm{F}}$ precisely and establish their basic properties before turning to the proof of Proposition~\ref{prop.vnupperbound}.

Throughout Section \ref{sec.upperbounds}, it will be convenient to fix a filtered probability space $\big(\Omega, \cF, \bbF = (\cF(t))_{0 \leq t \leq 1}, \bP\big)$ satisfying the usual conditions such that $\cF(0)$ is rich enough (i.e., for any $\mu \in \cP(\R)$, we can find a $\cF(0)$-measurable random variable $X$ with $\cL(X) = \mu$), and hosting independent $\bbF$-Brownian motions $B$ and $(B_i)_{i \in \N}$.

\subsubsection{The F\"ollmer drift construction and $\cU_p^{\mathrm{F}}$}

Let $t_0 \in [0,1]$. Define $\cA_p^{\mathrm{F}}(t_0)$ to be the set of absolutely continuous, even, compactly supported functions $f:\R\to[0,\infty)$ such that
\begin{equation*}
    \int_\R f\, d \gamma_{1-t_0} = 1, \qquad \bigg(\int_{\{f>0\}} \frac{|f'(x)|^p}{f(x)^{p-1}}\,\gamma_{1-t_0}(dx)\bigg)^{1/p} \leq \frac{\cstar}{2}.
\end{equation*}
We remind the reader that $\gamma_v$ is the Gaussian measure on $\R$ with mean $0$ and variance $v$. For each $p \geq 2$, we then define $\cU^{\mathrm{F}}_p : [0,1] \to \R$ by
\begin{align*}
    \cU^{\mathrm{F}}_p (t_0) \coloneqq \inf \Big\{[f]_{\infty} : {f \in \cA_p^{\mathrm{F}}(t_0)}\Big\}, \qquad [f]_{\infty} \coloneqq \inf \big\{ r > 0 : f = 0 \text{ on } [-r,r]^c \big\}.
\end{align*}
We note that we interpret $\gamma_0$ as $\delta_0$, and it is easy to check that $\cU_p^{\mathrm{F}}(1) = 0$.
Each $f \in \cA_p^{\mathrm{F}}(t_0)$ determines a drift process $\alpha^f = (\alpha^f(t))_{t_0 \leq t \leq 1}$ through the following procedure. First, we define
\begin{align} \label{def.v}
 v : [t_0,1] \times \R \to \R, \quad v(t,x) \coloneqq \log P_{1-t} f(x),
\end{align}
where $P_s$ is the heat semigroup, defined by $P_s f(x) \coloneqq \E[f(x + B(s))]$.

\begin{lemma} \label{lem.follmer}
    Fix $t_0 \in [0,1)$, $p \geq 2$, and $f \in \cA_p^{\mathrm{F}}(t_0)$.
Define $v$ by \eqref{def.v}. Then, the SDE
\begin{align*}
    dX^f(t) = \partial_x v(t,X^f(t))\,dt + dB(t), \quad t_0 \leq t \leq 1, \quad X^f(t_0) = 0,
\end{align*}
has a unique strong solution, which satisfies $X^f(1) \sim f \, d\gamma_{1- t_0}$. In particular,
\begin{align} \label{schrodinger.term}
     \|X^f(1)\|_{\infty} = [f]_{\infty}.
\end{align}
Moreover, the drift
\begin{align} \label{def.alpha}
    \alpha^f(t) \coloneqq \partial_x v(t,X^f(t)), \qquad t_0 \leq t \leq 1,
\end{align}
is a mean-zero martingale such that
\begin{align} \label{schrodinger.lp}
    \|\alpha^f(t)\|_p \leq \frac{\cstar}{2}, \qquad t_0 \leq t \leq 1.
\end{align}
\end{lemma}
\begin{proof}
These are all fairly standard properties of the F\"ollmer drift, but we sketch the proof for $t_0=0$ to ease notation; the general case is no different. We will also write $X=X^f$ and $\alpha=\alpha^f$.
Differentiating the heat kernel yields
\[
\partial_x v(t,x) =  \frac{\int_\R \frac{y-x}{1-t }f(y)\exp\!\Big(\frac{y^2}{2 } - \frac{(y-x)^2}{2(1-t )}\Big)dy}{\int_\R  f(y) \exp\!\Big(\frac{y^2}{2} - \frac{(y-x)^2}{2(1-t )}\Big)dy},\qquad (t,x)\in[0,1)\times\R.
\]
The existence and uniqueness of the SDE is shown in Theorem 25 of \cite{baudoin2002conditioned}, and the idea is roughly as follows: The pathwise uniqueness follows from the fact that $v$ is smooth and thus locally Lipschitz on $[0,1) \times \R$. The weak existence comes from a direct construction, defining a probability measure on $C[0,1]$ by specifying its Radon--Nikodym derivative with respect to Wiener measure to be given by $C[0,1] \ni x \mapsto f(x(1))$. That is, $\E\, \varphi(X)=\E[f(B(1))\varphi(B)]$ for any bounded measurable $\varphi : C[0,1] \to \R$. The fact that $X(1) \sim f \, d\gamma_1$ is a consequence of this identity (see also Proposition 26 of \cite{baudoin2002conditioned}): for all bounded measurable $\varphi : \R \to \R$, it implies
\[
\E\, \varphi(X(1))= \E \big[f(B(1)) \varphi(B(1))\big] = \int_\R f(x)\varphi(x)\,\gamma_{1}(dx).
\]
To justify that $\alpha(t)=\partial_x v(t,X(t))$ is a martingale, we note first that
\begin{equation}\label{eq.lem.follmer.identity}
\E[\varphi(X(1))\,|\,\F(t)] = \frac{P_{1-t} (\varphi f)(X(t))}{P_{1-t}f(X(t))}, \quad \text{a.s.}, \ 0 \le t \le 1,
\end{equation}
for any $\varphi\in L^1(\mu)$, where $(\F(t))_{t \in [0,1]}$ is the Brownian filtration. To see this, take $h : C[0,t] \to \R$ bounded and measurable, and use the Markov semigroup formula $P_{1-t}f(B(t))=\E[f(B(1))\,|\,\F(t)]$ to compute
\begin{align*}
\E\big[h(X|_{[0,t]}) \varphi(X(1))\big] &= \E\big[h(B|_{[0,t]}) f(B(1))\varphi(B(1))\big] \\
    &= \E\bigg[h(B|_{[0,t]}) f(B(1)) \frac{P_{1-t} (\varphi f)(B(t))}{P_{1-t}f(B(t))} \bigg] \\
    &= \E\bigg[h(X|_{[0,t]})  \frac{P_{1-t} (\varphi f)(X(t))}{P_{1-t}f(X(t))} \bigg].
\end{align*}
Differentiating under the integral sign in \eqref{def.v} and noting that differentiation commutes with  $P_t$, we deduce
\[
\alpha(t) = \partial_x v(t,X(t)) = \frac{P_{1-t}  f'(X(t))}{P_{1-t}f(X(t))} = \E\big[ (\log f)'(X(1))\,|\,\F(t)\big],
\]
where we used \eqref{eq.lem.follmer.identity} and that $(\log f)' \in L^p(\mu) \subset L^1(\mu)$ by the definition of $\mathcal{A}_p^\mathrm{F}(0)$. This shows that $\alpha$ is a martingale. As $\E\alpha(t)$ is constant in time, we have for any $s \in [0,1]$
\[
\E \alpha(s) = \E\bigg[ \int_0^1\alpha(t)\,dt + B(1)\bigg] = \E[X(1)]= \int_\R xf(x)\,\gamma_1(dx)=0
\]
where we use that $f$ is even by assumption. Finally, to deduce the $L^p$-norm bound, use Jensen's inequality and the assumption that $f \in \cA_p^\mathrm{F}(0)$:
\[
\E|\alpha(t)|^p \le \E|\alpha(1)|^p = \E|(\log f)'(X(1))|^p = \int_\R |(\log f)'(x)|^p f(x)\,\gamma_{1}(dx) \le \bigg(\frac{\cstar}{2}\bigg)^p. \qedhere
\]
\end{proof}

We call the process $\alpha^f$ appearing in Lemma \ref{lem.follmer} the Föllmer drift corresponding to $f$, and $X^f$ the Föllmer process corresponding to $f$. We note that \eqref{schrodinger.term} shows that $\cU_p^{\mathrm{F}}(t)$ may be viewed as the value of the optimization problem defining $\V^\infty(t,\delta_0)$, restricted to controls arising from the F\"ollmer drift with $f\in \cA_p^\mathrm{F}(t)$.

\begin{lemma} \label{lem.follmervalue}
Fix $p\geq 2$. The function $\cU_p^\mathrm{F}$ is upper semicontinuous on $[0,1]$. In particular, it is bounded and
    \begin{equation*}
        \lim_{t \to 1} \cU_p^\mathrm{F}(t) = 0.
    \end{equation*}
\end{lemma}

\begin{proof}
    Write $K \coloneqq \cstar/(2p)$. After the change of variables \(h = f^{1/p}\), and using that \(h \mapsto [h]_\infty\) is invariant under positive scaling, we can equivalently write
	\begin{equation}\label{eq:v_functional_h}
		\cU_p^{\mathrm{F}}(t) = \inf \left\{ [h]_{\infty} : h \geq 0,\ \|h'\|_{L^p(\gamma_{1-t})} \leq K\|h\|_{L^p(\gamma_{1-t})},\ \|h\|_{L^p(\gamma_{1-t})} > 0\right\}
	\end{equation}
    where the functions $h:\R\to[0,\infty)$ are assumed to be absolutely continuous, even, and compactly supported.
	We first establish that $\cU_p^\mathrm{F}(t)$ is bounded and $\lim_{t \to 1} \cU_p^\mathrm{F}(t) = 0$.

	Fix $r_t > 0$ to be specified, and define
	\begin{equation*}
		h_t(x) \coloneqq
		\begin{cases}
			1 & |x|\le r_t/2,\\
			(r_t/2)^{-1} (r_t - |x|) & |x|\in[r_t/2, r_t],\\
			0 & |x|>r_t.
		\end{cases}
	\end{equation*}
	To ensure that $h_t$ is feasible in~\eqref{eq:v_functional_h}, it suffices to choose $r_t$ such that
	\begin{equation}\label{eq:jon_need}
		\gamma_{1-t}([r_t/2, r_t]) \leq (Kr_t/2)^p \gamma_{1-t}([0, r_t/2])\,,
	\end{equation}
	since~\eqref{eq:jon_need} implies
	\begin{align*}
		\|h'_t\|_{L^p(\gamma_{1-t})} = \left(2 (r_t/2)^{-p}\int_{[r_t/2, r_t]}d\gamma_{1-t}\right)^{1/p} & \leq K \left(2 \int_{[0, r_t/2]}d\gamma_{1-t}\right)^{1/p} \\
		& \leq K \|h_t\|_{L^p(\gamma_{1-t})}\,.
	\end{align*}

	Let $r_t \coloneqq 2\sqrt{p (1-t) \log(4K^{-2}/(1-t))}$.
	By a standard Gaussian tail bound,
	\begin{equation*}
		\gamma_{1-t}([r_t/2, \infty)) \leq e^{-r_t^2/8(1-t)} = (K/2)^p(1-t)^{p/2}\,.
	\end{equation*}
	Therefore for this choice of $r_t$,
	\begin{align*}
		\gamma_{1-t}([r_t/2, r_t]) & \leq (K/2)^p(1-t)^{p/2}\,.
	\end{align*}
	By the same token, since $\gamma_{1-t}([0, r_t/2]) \geq \tfrac 12 - (K/2)^p \geq \tfrac 14$ and $\log(4K^{-2}/(1-t)) \geq 1$ for all $t \in [0, 1)$,
	\begin{align*}
		(Kr_t/2)^p \gamma_{1-t}([0, r_t/2]) \geq \frac 14 K^p (1-t)^{p/2} \geq \gamma_{1-t}([r_t/2, r_t])\,.
	\end{align*}
	Hence \eqref{eq:jon_need} holds for all $t \in [0, 1)$.
	We obtain that $\sup_{t \in [0, 1)}\cU_p^{\mathrm{F}}(t) \leq \sup_{t \in [0, 1)} r_t  = r_0$ and $\lim_{t \to 1} \cU_p^\mathrm{F}(t) \leq \lim_{t \to 1} r_t = 0$, as claimed.

	We next show that $\cU_p^{\mathrm{F}}(t)$ is upper semicontinuous on $[0, 1)$.
	For any $t \in [0, 1)$ and $\eps > 0$, pick an admissible $h$ such that $\operatorname{supp}(h) \subseteq [-r, r] \subseteq (-R, R)$, where $R = \cU_p^{\mathrm{F}}(t) + \eps$.
	We may assume by rescaling that $h \leq 1$.
	Let $\varphi$ be a smooth even cutoff function with $\varphi \equiv 1$ on $[-r,r]$ and $\varphi = 0$ on $[-R, R]^c$, and set $h_\eta = (1-\eta)h + \eta \varphi$ for small $\eta > 0$.
	Then $\operatorname{supp}(h_\eta) \subseteq [-R, R]$, and $\|h_\eta\|_{L^p(\gamma_{1-t})} \geq \|h\|_{L^p(\gamma_{1-t})}$.
	The derivatives $h'$ and $\varphi'$ are supported on disjoint sets, and thus
	\begin{equation*}
		\|h'_\eta\|_{L^p(\gamma_{1-t})}^p = (1-\eta)^p \|h'\|_{L^p(\gamma_{1-t})}^p + O(\eta^p)\,.
	\end{equation*}
	By the admissibility of $h$, we obtain for sufficiently small $\eta$
	\begin{equation}
		\|h'_\eta\|_{L^p(\gamma_{1-t})}^p \leq (1-\eta)K^p \|h_\eta\|_{L^p(\gamma_{1-t})}^p\,.
	\end{equation}

	As $s \to t$, the density $\gamma_{1-s}$ converges uniformly to $\gamma_{1-t}$ on $[-R, R]$; hence
	\begin{equation}
		\|h'_\eta\|_{L^p(\gamma_{1-s})}^p \leq (1-\eta)K^p \|h_\eta\|_{L^p(\gamma_{1-s})}^p + o(1) \quad \text{as $s \to t$}\,,
	\end{equation}
	and in particular for $s$ sufficiently close to $t$, the function $h_\eta$ will satisfy the constraints in the definition of $ \cU_p^{\mathrm{F}}(s)$.
	Therefore, $\limsup_{s \to t}  \cU_p^{\mathrm{F}}(s) \leq R = \cU_p^{\mathrm{F}}(t) + \eps$.
	Taking $\eps \to 0$ yields the claim.
\end{proof}

\subsubsection{The deterministic control problem and $\cU_p^\mathrm{det}$}

For $p\ge2$, we define $\cU_p^\mathrm{det} : [0,1] \times \cP_p(\R) \to \R$ via
\begin{align*}
    \cU_p^\mathrm{det}(t,m) \coloneqq \inf \Big\{ [m']_{\infty} : m'\in\cP_p(\R),\ \bd_p(m,m') \leq \frac{(1-t)\cstar}{2} \Big\},
\end{align*}
where we recall that $\bd_p$ denotes the $p$-Wasserstein distance.
While it will not be needed for any of our arguments, we note that one can show that
\begin{align*}
    \cU_p^{\mathrm{det}}(t_0,m_0) = \inf_{\alpha} \Big\| X_0 + \int_{t_0}^1 \alpha(t) \, dt \Big\|_{\infty},
\end{align*}
where $X_0 \sim m_0$ is $\cF(t_0)$-measurable, and the infimum is taken over all progressively measurable processes satisfying
\begin{align*}
    \| \alpha(t) \|_p \leq \frac{\cstar}{2},\qquad t_0 \leq t \leq 1,
\end{align*}
so $\cU_p^{\mathrm{det}}$ can be interpreted as the value of a deterministic analogue of the limiting optimization problem, but with the $L^2$ constraint replaced by an $L^p$ constraint. We next record some important facts about $\cU_p^{\mathrm{det}}$.

\begin{lemma} \label{lem.updet.upperbound}
    The function $\cU_p^{\mathrm{det}}$ is continuous on $[0,1) \times \cP_p(\R)$. For each $t \in [0,1]$ and $m \in \cP(\R)$ with bounded support, we have
    \begin{align} \label{vpd.est1}
        \cU_p^{\mathrm{det}}(t,m) \leq [m]_{\infty}.
    \end{align}
    Moreover, for each $q > p$, $t \in [0,1)$ and $m \in \cP_q(\R)$, we have
    \begin{align} \label{vpd.est2}
    \cU_p^{\mathrm{det}}(t,m) \leq C (1-t)^{-p/(q-p)} \Big(\int_\R |x|^q\,m(dx)\Big)^{1/(q-p)},  
    \end{align}
    where $C=C(p,q)$ is a constant depending only on $p$ and $q$.
\end{lemma}

\begin{proof}
    First, consider the lift of $\cU_p^{\mathrm{det}}$, i.e., the function
    \begin{align*}
       \hat{\cU}_p^{\mathrm{det}} : [0,1) \times L^p(\Omega) \to \R, \quad \hat{\cU}_p^{\mathrm{det}}(t,X) \coloneqq \cU_p^{\mathrm{det}}\big(t,\cL(X)\big).
    \end{align*}
   Using the definition of the Wasserstein distance in terms of couplings, one obtains
    \begin{align*}
        \hat{\cU}_p^{\mathrm{det}}(t,X) = \inf \Big\{ \|Y\|_{\infty} : \|X - Y\|_p \leq \frac{(1-t)\cstar}{2} \Big\}.
    \end{align*}
    Moreover, making the change of variables $\wt{X} = \frac{X}{1-t}$ gives
    \begin{align} \label{vpd.scaling}
        \hat{\cU}_p^{\mathrm{det}}(t,X) = (1-t)\, \hat{\cU}_p^{\mathrm{det}}\Big(\frac{X}{1-t}\Big),
    \end{align}
    where we write $\hat{\cU}_p^{\mathrm{det}}(X) = \hat{\cU}_p^{\mathrm{det}}(0,X)$ for simplicity. Now, we notice that $\hat{\cU}_p^{\mathrm{det}} : L^p(\Omega) \to \R$ is convex and locally bounded, in the sense that for each $X \in L^p$ there exists an $\eps > 0$ such that $\hat{\cU}_p^{\mathrm{det}}$ is bounded on
    \begin{align*}
        \{ X' \in L^p : \|X' - X\|_p < \eps \}.
    \end{align*}
    It follows that $\hat{\cU}_p^{\mathrm{det}}$ is locally Lipschitz continuous with respect to $L^p$, and then the continuity of $\hat{\cU}_p^{\mathrm{det}} : [0,1) \times L^p \to \R$ follows from \eqref{vpd.scaling}. We deduce that ${\cU}_p^{\mathrm{det}}$ is also continuous with respect to $\cP_p$.

    Next, the upper bound \eqref{vpd.est1} follows directly from the definition of $\hat{\cU}_p^{\mathrm{det}}$ (by taking $m' = m$). For the upper bound \eqref{vpd.est2}, we fix $X \in L^q$, and for $R > 0$ we set $X_R = X \mathbf{1}_{\{|X|\leq R\}}$. Then we estimate
    \begin{align*}
        \|X - X_R\|^p_p &= \E\big[ |X|^p \mathbf{1}_{\{|X| > R\}} \big]
        \\
        &\leq \|X\|_q^{p} \cdot  \bP[|X| > R]^{(q-p)/q}
        \\
        &\leq \|X\|_q^q R^{-(q-p)}.
    \end{align*}
    So, to ensure that $\|X - X_R\|_p \leq \cstar/2$, we need
    \begin{align*}
        \|X\|_q^q R^{-(q-p)} \leq (2\pi)^{-p/2} \iff R \geq C \|X\|_q^{q/(q-p)}
    \end{align*}
    where $C=(2\pi)^{p/(2(q-p))}$. We deduce that
    \begin{align*}
        \hat{\cU}_p^{\mathrm{det}}(X) \leq C\|X\|_q^{q/(q-p)},
    \end{align*}
    and so by \eqref{vpd.scaling},
    \begin{align*}
        \hat{\cU}_p^{\mathrm{det}}(t,X) \leq C (1-t)^{-p/(q-p)} \|X\|_q^{q/(q-p)}.
    \end{align*}
    Recalling the relation $\hat{\cU}_p^{\mathrm{det}}(t,X) = {\cU}_p^{\mathrm{det}}(t,\cL(X))$ yields \eqref{vpd.est2}.
\end{proof}

It will also be useful to understand the restriction of the function $\cU_p^{\mathrm{det}}$ to empirical measures.
\begin{lemma} \label{lem.updet.emp}
    For $t \in [0,1]$ and $\bx \in \R^n$, we have
    \begin{align} \label{updet.emp}
        \cU_p^{\mathrm{det}}(t,m_{\bx}^n) =
    \inf \bigg\{ | \by |_{\infty} : \Big(\frac{1}{n} \sum_{i = 1}^n |y_i - x_i|^p \Big)^{1/p} \leq \frac{(1-t)\cstar}{2} \bigg \}.
    \end{align}
\end{lemma}

\begin{proof}
    First, for any $\by \in \R^n$ satisfying the constraint in \eqref{updet.emp}, the empirical measure $m_{\by}^n$ satisfies
 $\bd_p(m_{\bx}^n, m_{\by}^n) \leq (1-t)\cstar/2$, and so $\cU_p^{\mathrm{det}}(t,m_{\bx}^n) \leq [m_{\by}^n]_{\infty} = |\by|_{\infty}$. Taking an infimum over $\by$ gives
 \begin{align} \label{updet.emp.upper}
        \cU_p^{\mathrm{det}}(t,m_{\bx}^n) \leq
    \inf \bigg\{ | \by |_{\infty} : \Big(\frac{1}{n} \sum_{i = 1}^n |y_i - x_i|^p \Big)^{1/p} \leq \frac{(1-t)\cstar}{2} \bigg \}.
    \end{align}
    Now let $m'$ be optimal for the problem defining $\cU_p^{\mathrm{det}}$, and let $(X,Y)$ be an optimal coupling of $m_{\bx}^n$ and $m'$ with respect to the $p$-Wasserstein distance. Thus we have
    \begin{align} \label{updetcomps}
        \cU_p^{\mathrm{det}}(t,m_{\bx}^n) = \| Y \|_{\infty}, \quad \| X - Y\|_p \leq \frac{(1- t)\cstar}{2}.
    \end{align}
    Note that because $X \sim m_{\bx}^n$, we can find disjoint events $\Omega_1,\dots,\Omega_n \subset \Omega$ such that $\bP[\Omega_i] = 1/n$, and $X = \sum_{i = 1}^n x_i \mathbf{1}_{\Omega_i}$. Now consider the random variable $Y' = \E[Y \,|\, X]$ which, since $X$ is discrete, must take the form $Y' = \sum_{ i = 1}^n y_i \mathbf{1}_{\Omega_i}$ for some $\by = (y_1,\dots,y_n)$. Notice that by \eqref{updetcomps}, we have $| \by |_{\infty} \leq \| Y\|_{\infty} = \cU_p^{\mathrm{det}}(t,m_{\bx}^n)$, as well as
     \begin{align*}
         \Big(\frac{1}{n} \sum_{i = 1}^n |y_i - x_i|^p \Big)^{1/p} = \|X - Y'\|_p \leq \|X - Y\|_p = \frac{(1-t)\cstar}{2}.
     \end{align*}
     This implies that
     \begin{align*}
        \cU_p^{\mathrm{det}}(t,m_{\bx}^n) \geq
        \inf \bigg\{ | \by |_{\infty} : \bigg(\frac{1}{n} \sum_{i = 1}^n |y_i - x_i|^p \bigg)^{1/p} \leq \frac{(1-t)\cstar}{2}  \bigg \},
    \end{align*}
    completing the proof.
\end{proof}

\subsubsection{Coupling procedure and the upper bound on $\cV^n$}
\label{subsec.coupling}
For $p \geq 2$, we now define $\cU_p : [0,1] \times \cP_p(\R) \to \R$ by
\begin{align} \label{def.up}
    \cU_p(t,m) = \cU_p^{\mathrm{det}}(t,m) + \cU_p^{\mathrm{F}}(t),
\end{align}
where $\cU_p^{\mathrm{det}}$ and $\cU_p^{\mathrm{F}}$ are as defined in the previous two subsections. Note that $\cU_p$ is upper semicontinuous on $[0,1)\times\P_p(\R)$ by Lemmas~\ref{lem.follmervalue} and \ref{lem.updet.upperbound}.

Let us first establish $\cU_p$ as an upper bound for the value function $\cV^\infty$, by showing that the budget-splitting control from the previous sections is admissible for the limiting problem. This in particular proves Proposition \ref{prop.upperbound.uinf}.

\begin{proposition} \label{prop.ulequp}
    For each $(t,m) \in [0,1] \times \cP_2(\R)$, we have
    \begin{align*}
        \cV^\infty(t,m) \leq \cU_2(t,m).
    \end{align*}
    In particular, $\cV^\infty$ is locally bounded on $[0,1) \times \cP_2(\R)$.
\end{proposition}

\begin{proof}
    Fix $(t_0,m_0)\in[0,1]\times \cP_2(\R)$. The case $t_0=1$ is trivial since both sides of the inequality are then $[m]_\infty$, which can be infinite. So assume $t_0<1$. Let $f \in \cA_2^{\mathrm{F}}(t_0)$ and define $\alpha^f$ as the corresponding control in the sense of Lemma \ref{lem.follmer}. In particular, we know that
    \begin{align*}
        \|\alpha^f(t)\|_2 \leq \frac{\cstar}{2},\qquad t_0\leq t\leq 1,
    \end{align*}
    and that the process $X^f(t) = \int_{t_0}^t \alpha^f(s)\, ds + B(t) - B(t_0)$ satisfies $\|X^f(1)\|_{\infty} = [f]_{\infty}$. Next, let $m_0' \in \cP_2(\R)$ be such that $\bd_2(m_0,m_0') \leq (1-t_0)\cstar/2$. We then let $X_0$ and $Y_0$ be square-integrable random variables, measurable with respect to $\cF(t_0)$, such that $(X_0,Y_0)$ is an optimal coupling of $m_0$ and $m_0'$. For $t_0\leq t\leq 1$, we define
    \begin{align*}
        \alpha(t) \coloneqq \alpha^f(t) + \alpha^\mathrm{det},\qquad \alpha^\mathrm{det} \coloneqq \frac{Y_0 - X_0}{1 - t_0},
    \end{align*}
    and
    \begin{align*}
        X(t) \coloneqq X_0 + \int_{t_0}^t \alpha(s)\,ds + B(t) - B(t_0) = X_0 + X^f(t) + (t - t_0) \alpha^\mathrm{det}.
    \end{align*}
    Note that $X(1) = X^f(1) + Y_0$ and
    \begin{align*}
        \| \alpha(t)\|_2 &\leq \|\alpha^f(t)\|_2 +  \|\alpha^\mathrm{det}\|_2
        \\
        &= \|\alpha^f(t)\|_2 + \frac{1}{1-t_0} \bd_2(m_0,m_0') \leq \frac{\cstar}{2} + \frac{\cstar}{2} = \cstar,
    \end{align*}
    so by Lemma \ref{lem.equivstrong},
    \begin{align*}
        \cV^\infty(t_0,m_0) \leq \|X(1)\|_{\infty} \leq \|X^f(1)\|_{\infty} + \|Y_0\|_{\infty} = [f]_{\infty} + [m_0']_{\infty}.
    \end{align*}
    Taking an infimum over $f \in \cA_2^{\mathrm{F}}(t_0)$ and $m_0'$ satisfying $\bd_2(m_0,m_0') \leq (1-t_0)\cstar/2$, we obtain the result.
\end{proof}

The main objective of this subsection is to prove the following non-asymptotic upper bound on $\cV^n$.

\begin{proposition} \label{prop.vnupperbound}
    For each $p > 4$, there is a constant $C_p > 0$ such that for all $n \in \N$, $k_0 \in \{0,1,\dots,n\}$, and $\bx \in \R^n$, we have
    \begin{align} \label{vn.upperbound.main}
        \cV^n(k_0,\bx) \leq \cU_p(k_0/n, m_{\bx}^n) + C_p n^{\frac{1}{2p} - \frac{1}{4}}.
    \end{align}
\end{proposition}

We now explain the strategy of the proof of Proposition~\ref{prop.vnupperbound}. Fix throughout $n\in\N$, $k_0\in\{0,1,\dots,n\}$, and $\bx\in\R^n$, and set $t_0\coloneqq k_0/n$. The proof is trivial when $k_0=n$, since $\V^n(n,\bm{x})=|\bx|_\infty=\cU_p(1,m_{\bx}^n)$, so let us assume $k_0<n$. Starting from an arbitrary $f\in\cA_p^{\mathrm F}(t_0)$, we construct a coupling of the following two $n$-dimensional systems:
\begin{enumerate}
    \item a continuous system $\bm X$ consisting of $n$ independent realizations of the state processes for the limiting problem, driven by the budget-splitting control associated with $f$ and the initial condition $m_{\bx}^n$;
    \item the discrete state process $\bY$ for the $n$-particle problem, started from $\bx$ and constructed from a specifically chosen admissible control $\eps$, independent Gaussian increments $\bZ$, and filtration $\mathbb{G}$ satisfying the assumptions of Lemma~\ref{lem.dpp}.
\end{enumerate}
The coupling is constructed so that, at terminal time, the two systems are close in the sense that
\begin{equation*}
\E|\bY(n)-\bm X(1)|_\infty \le C_p n^{\frac{1}{2p}-\frac14},
\end{equation*}
and by the specific choice of $\bm X$, we will have
\begin{equation*}
\E|\bm X(1)|_\infty \le [f]_\infty + \cU_p^{\mathrm{det}}(t_0,m_{\bx}^n).
\end{equation*}
Therefore,
\begin{equation*}
    \cV^n(k_0, \bx) \leq \E|\bY(n)|_\infty \leq [f]_\infty + \cU_p^\mathrm{det}(t_0, m_{\bm{x}}^n) + C_p n^{\frac{1}{2p} - \frac{1}{4}}.
\end{equation*}
Since $f\in\cA_p^{\mathrm F}(t_0)$ was arbitrary, taking the infimum over $f$ yields \eqref{vn.upperbound.main}.

The following lemma gives the construction of $\bm{Y}$ underlying the proof. Given a Brownian motion $\bm{B}$ and an adapted sequence of vectors $(\bm{b}(k))$, it constructs an admissible triple $(\bbG,\bm{Z},\epsilon)$ such that the associated discrete state process has drift $\bm{b}(k)$, and such that its martingale increments are close to the increments of $\bm{B}$. Recall here from the introduction that the vectors in the $\ell^2$-ball of radius $\cstar/\sqrt{n}$ with $\cstar=\sqrt{2/\pi}$ are precisely the admissible drifts.

\begin{lemma}\label{lem.coupling.discrete}
    Fix $n\in\N$, $\bx\in\R^n$, and $k_0 \in \{0,1,\dots,n - 1\}$, and write $t_0=k_0/n$. Let $\bm{B}=(\bm{B}(t))_{t_0\leq t \leq 1}$ be an $n$-dimensional Brownian motion and denote its canonical filtration by $\mathbb{F}=(\cF(t))_{t_0\leq t\leq 1}$. Also, define the filtration $\bbG = (\cG(k))_{k=k_0,\dots,n}$ by $\cG(k)=\cF(\tfrac{k}{n})$. Suppose we are given random vectors $\bm{b}(k_0),\dots,\bm{b}(n-1)\in\R^n$ such that
    \begin{equation*}
        |\bm{b}(k)|_2\leq \frac{\cstar}{\sqrt{n}},\qquad \bm{b}(k)\text{ is $\G(k)$-measurable,}\qquad\text{for all $k=k_0,\dots,n-1$.}
    \end{equation*}
    Then, there exist processes $\bm{Z}=(\bm{Z}(k))_{k=k_0+1,\dots,n}$ and $\eps=(\eps(k))_{k={k_0+1,\dots,n}}$ such that the following properties hold:
    \begin{enumerate}
        \item The triple $(\bbG, \bZ, \eps)$ satisfies the assumptions of Lemma~\ref{lem.dpp} with $\nu=\gamma$ and $k_1=n$.

        \item We have
        \begin{align} \label{b.increment}
            \bm{b}(k) = \E\big[\epsilon(k+1) \bm{Z}(k+1)\,\big|\,\G(k)\big], \qquad k = k_0,\dots,n-1.
        \end{align}
        In particular, we have the martingale decomposition
        \begin{align}\label{y.martdecomp}
            \bm{Y}(k) \coloneqq \bx + \sum_{\ell=k_0+1}^k \epsilon(\ell)\bm{Z}(\ell) = \bx + \sum_{\ell=k_0 + 1}^k \bm{b}(\ell - 1) + \sum_{\ell= k_0 + 1}^k \bm{\sigma}(\ell),
        \end{align}
        where $\bm{\sigma}(\ell) \coloneqq \eps(\ell) \bZ(\ell) - \E\big[\eps(\ell) \bZ(\ell)\,\big|\,\mathcal{G}(\ell-1)\big]$ and $k=k_0,\dots,n$.

        \item For every $p > 2$ and $k_1=k_0+1,\dots,n$,
        \begin{align}
            \E\bigg|\sum_{k=k_0+1}^{k_1}\bm{\sigma}(k) - \big(\bm{B}(\tfrac{k_1}{n}) - \bm{B}(\tfrac{k_0}{n})\big)\bigg|_\infty
            &\leq C_p n^{\frac14} \max_{k_0 \leq k < k_1} \bigg(\sum_{i=1}^n \|b_i(k)\|_p^p\bigg)^{\frac{1}{2p}}, \label{eq.coupling.discrete.marterror}
        \end{align}
        where $C_p>0$ is a constant that depends only on $p$.

        \end{enumerate}
\end{lemma}

The main subtlety is the choice of the Gaussian increments $\bm Z(k)$. While the
Brownian increments $\bm{B}(\frac{k}{n})-\bm{B}(\frac{k-1}{n})$ are the natural
first candidates, this choice does not yield an estimate of the form
\eqref{eq.coupling.discrete.marterror}. We instead construct $\bm Z(k)$ as the
increments of an auxiliary Brownian motion, obtained from $\bm B$ through a
sign-flipping procedure.

\begin{proof}
    For $k = k_0,\dots,n-1$, let $\bm{v}(k)\in\R^n$ denote an $\ell^2$-normalized version of $\bm{b}(k)$, i.e., $\bm{v}(k) = \bm{b}(k)/|\bm{b}(k)|_{2}$ on $\{\bm{b}(k)\neq 0\}$ and $\bm{v}(k)=(1,0,\dots,0)$ on $\{\bm{b}(k) = 0\}$. We define
    \begin{equation*}
        \epsilon(k+1) \coloneqq
        \begin{cases}
            -1 & \text{if } -\Psi\big(\sqrt{n}|\bm{b}(k)|_2\big) < \sqrt{n} \big\langle \bm{B}(\tfrac{k+1}{n}) - \bm{B}(\tfrac{k}{n}),\bm{v}(k)\big\rangle \leq 0, \\
            1  & \text{otherwise,}
        \end{cases}
    \end{equation*}
    where we write $\langle x,y\rangle$ for the inner product on $\R^n$ and
    \begin{equation} \label{def.psi}
        \Psi(r) \coloneqq \sqrt{-2\log\Big(1 - \frac{r}{\cstar}\Big)},\qquad r\in[0,\cstar].
    \end{equation}
    Here, we use the convention $\Psi(\cstar)\coloneqq \infty$. This choice of $\Psi$ satisfies the following identity, which we will use below:
    \begin{equation}\label{psi.property}
        \E\Big[\eta\big(1-2 \mathbf{1}_{\{-\Psi(r)<\eta\leq 0\}}\big)\Big] = r,
        \qquad r\in[0,\cstar],\quad \eta\sim\mathcal{N}(0,1).
    \end{equation}

    We next construct an auxiliary $n$-dimensional Brownian motion $\bm{W}$ whose increments will be used to define $\bm{Z}$. Let $\bm{P}(k) \coloneqq I_n - \bm{v}(k)\bm{v}(k)^\top \in \R^{n\times n}$ be the orthogonal projection onto the subspace $\bm{v}(k)^\perp \coloneqq \{\bm{y}\in\R^n : \langle \bm{v}(k), \bm{y} \rangle = 0\}$. We then inductively define $(\bm{W}(t))_{t_0\leq t \leq 1}$ by setting $\bm{W}(t_0) \coloneqq \bm{B}(t_0)$ and
    \begin{align}
        \bm{W}(t)
        & \coloneqq \bm{W}(\tfrac{k}{n}) + \big\langle \bm{v}(k), \bm{B}(t) - \bm{B}(\tfrac{k}{n})\big\rangle \bm{v}(k) \nonumber \\
        & \qquad\qquad + \epsilon(k+1) \bm{P}(k)\big(\bm{B}(t) - \bm{B}(\tfrac{k}{n})\big),
        \label{eq:upper-bound:bm-construction}
    \end{align}
    for $\tfrac{k}{n} < t \leq \frac{k+1}{n}$ and $k=k_0,\dots,n-1$. Thus, this construction modifies the Brownian increment of $\bm{B}$ on each interval $(\frac{k}{n},\frac{k+1}{n}]$ only through a possible sign flip in directions orthogonal to $\bm v(k)$, while leaving the component along $\bm v(k)$ unchanged. This defines a new Brownian motion $\bm{W}$ with the property that at the discrete time points $t=\frac{k}{n}$, the filtrations generated by $\bm{B}$ and $\bm{W}$ coincide. A detailed proof of this is left to Appendix~\ref{sec:appendix:bm-construction}.

    We now define
    \begin{equation*}
        \cG(k) \coloneqq \sigma\big(\bm{W}(t): t \leq \tfrac{k}{n}\big) = \sigma\big(\bm{B}(t): t \leq \tfrac{k}{n}\big),\qquad k=k_0,\dots,n,
    \end{equation*}
    as well as
    \begin{equation*}
        \bZ(k) \coloneqq \bm{W}(\tfrac{k}{n}) - \bm{W}(\tfrac{k-1}{n}),\qquad k=k_0+1,\dots,n.
    \end{equation*}
    The process $\epsilon$ can now be written as
    \begin{equation}\label{eq:upper-bound:epsilon:identity}
        \epsilon(k) =
        \begin{cases}
            -1 & \text{if } -\Psi\big(\sqrt{n}|\bm{b}(k-1)|_2\big) < \sqrt{n} \big\langle \bm{Z}(k),\bm{v}(k-1)\big\rangle \leq 0, \\
            1  & \text{otherwise,}
        \end{cases}
    \end{equation}
    since $\langle \bm{B}(\frac{k}{n}) - \bm{B}(\frac{k-1}{n}),\bm{v}(k-1)\rangle = \langle \bm{Z}(k), \bm{v}(k-1)\rangle$. It is immediate that $(\bbG, \bm{Z}, \epsilon)$ satisfies the assumptions of Lemma~\ref{lem.dpp}.

    To prove \eqref{b.increment}, let us fix $k=k_0+1,\dots,n$, and introduce the random variables
    \begin{align}\label{def.eta.r}
        \eta \coloneqq  \sqrt{n} \langle \bm{v}(k-1), \bZ(k) \rangle, \quad R \coloneqq \sqrt{n} |\bm{b}(k-1)|_2.
    \end{align}
    Since $\bm{v}(k-1)$ has unit length and $\bZ(k) \sim \gamma_{1/n}^{\otimes n}$ is independent of $\cG(k-1)$, $\eta$ is a standard Gaussian, independent of $\cG(k-1)$. Using \eqref{eq:upper-bound:epsilon:identity} and \eqref{psi.property}, we obtain
    \begin{align} \nonumber
        \big\langle \bm{v}(k-1),\, \E\big[\epsilon(k) \bm{Z}(k)\,\big|\,\G(k-1)\big] \big\rangle
        & = \frac{1}{\sqrt{n}} \E\big[\eta\,\eps(k) \,\big|\, \cG(k-1) \big] \\ \nonumber
        & = \frac{1}{\sqrt{n}} \E\Big[ \eta \big(1 - 2 \mathbf{1}_{\{-\Psi(R) < \eta \leq 0\}} \big) \,\Big|\, \cG(k-1) \Big] \\
        & = |\bm{b}(k-1)|_2. \label{etar.comp}
    \end{align}
    On the other hand, if $\bm{u}$ is a $\cG(k-1)$-measurable random vector which is orthogonal to $\bm{v}(k-1)$, then $\langle \bm{u}, \bm{Z}(k) \rangle$ and $\eps(k)$ are conditionally independent with mean zero given $\cG(k-1)$, and so
    \begin{align} \label{cond.indep}
        \big\langle \bm{u},\, \E\big[ \eps(k) \bZ(k) \,\big|\, \cG(k-1)\big] \big\rangle = \E\big[ \eps(k) \langle \bm{u},\bZ(k)\rangle
        \big|\, \cG(k-1)\big] = 0.
    \end{align}
    Together, \eqref{etar.comp} and \eqref{cond.indep} imply that \eqref{b.increment} holds.

    It remains to prove the estimate \eqref{eq.coupling.discrete.marterror} for some $p>2$. For $k=k_0+1,\dots,n$, notice that
    \begin{align*}
        \langle \bm{v}(k-1),\bZ(k)\rangle & = \big\langle \bm{v}(k-1),\, \bm{B}(\tfrac{k}{n}) - \bm{B}(\tfrac{k-1}{n})\big\rangle, \\
        \bm{P}(k-1) \bZ(k)                 & = \epsilon(k) \bm{P}(k-1) \big(\bm{B}(\tfrac{k}{n}) - \bm{B}(\tfrac{k-1}{n})\big).
    \end{align*}
    It follows that
    \begin{equation*}
        \epsilon(k) \bm{Z}(k) - \big(\bm{B}(\tfrac{k}{n}) - \bm{B}(\tfrac{k-1}{n})\big) = (\epsilon(k) - 1) \langle \bm{v}(k-1),\bZ(k)\rangle \bm{v}(k-1)
    \end{equation*}
    and
    \begin{align}
        \sigma(k) - \big(\bm{B}(\tfrac{k}{n}) - \bm{B}(\tfrac{k-1}{n})\big)
        & = \epsilon(k) \bZ(k) - \bm{b}(k-1) - \big(\bm{B}(\tfrac{k}{n}) - \bm{B}(\tfrac{k-1}{n})\big) \notag\\
        & = \frac{1}{\sqrt{n}}\xi(k) \bm{v}(k-1) \label{eq:upper-bound:mart-incr}
    \end{align}
    where
    \begin{align}
        \xi(k)
        & \coloneqq - 2 \mathbf{1}_{\{\epsilon(k)=-1\}} \sqrt{n} \langle \bZ(k), \bm{v}(k-1)\rangle - \sqrt{n}|\bm{b}(k-1)|_2 \label{eq:upper-bound:mart:def-xi} \\
        & = -2\mathbf{1}_{\{-\Psi(R) < \eta \leq 0\}}\eta - R \label{eq:upper-bound:mart:def-xi:simplified},
    \end{align}
    with $\eta$ and $R$ as defined in \eqref{def.eta.r}.

    By the martingale decomposition \eqref{y.martdecomp}, both $\bm{\sigma}(k)$ and $\bm{B}(\frac{k}{n}) - \bm{B}(\frac{k-1}{n})$ have mean zero conditionally on $\G(k-1)$, and thus
    \begin{equation}\label{eq:upper-bound:mart:xi-mean-zero}
        \E[\xi(k)\,|\,\cG(k-1)] = 0.
    \end{equation}
    Next, since $\Psi(r) \leq 2\sqrt{r}$ for $r<r_0$, with $r_0$ sufficiently small, we have
    \begin{equation*}
        \mathbf{1}_{\{-\Psi(R) < \eta \leq 0\}} |\eta| \leq 2\sqrt{R}\,\mathbf{1}_{\{R < r_0\}} + |\eta|\mathbf{1}_{\{R\geq r_0\}} \leq C\sqrt{R}(1 + |\eta|)
    \end{equation*}
    for some constant $C>0$. In view of \eqref{eq:upper-bound:mart:def-xi:simplified} and \eqref{eq:upper-bound:mart:xi-mean-zero}, it follows that
    \begin{align*}
        \E\big[|\xi(k)|^{2p}\,\big|\,\G(k-1)\big]^\frac{1}{2p}
        \leq 4 \E\big[\big|\mathbf{1}_{\{-\Psi(R) < \eta \leq 0\}} \eta\big|^{2p}\,\big|\,\G(k-1)\big]^\frac{1}{2p} \leq 4C\sqrt{R}\|1+\eta\|_{2p}
    \end{align*}
    where we used that $R$ is $\cG(k-1)$-measurable and $\eta\sim\mathcal{N}(0,1)$ is independent of $\cG(k-1)$. Letting $C_p$ denote a generic constant depending only on $p$, we have thus shown that
    \begin{equation}\label{eq:upper-bound:mart:xi-lp}
        \E\big[|\xi(k)|^{2p}\,\big|\,\G(k-1)\big]^{\frac1p} \leq C_p \sqrt{n} |\bm{b}(k-1)|_2
    \end{equation}
    for each $k=k_0+1,\dots,n$.

    To prove \eqref{eq.coupling.discrete.marterror}, let us now define an $\R^n$-valued $\bbG$-martingale $\bm{N}=(\bm{N}(k))_{k=k_0,\dots,n}$ by
    \begin{equation*}
        \bm{N}(k) \coloneqq \frac{1}{\sqrt{n}}\sum_{\ell=k_0}^{k-1} \xi(\ell+1) \bm{v}(\ell) = \sum_{\ell=k_0+1}^{k}\bm{\sigma}(\ell) - \big(\bm{B}(\tfrac{k}{n}) - \bm{B}(\tfrac{k_0}{n})\big)
    \end{equation*}
    where the equality holds by \eqref{eq:upper-bound:mart-incr}. Let $k_0 < k_1 \leq n$. By the Marcinkiewicz--Zygmund inequality for martingales \cite[Theorem 2.1]{rio-2009}, we can estimate for each coordinate $i$
    \begin{align*}
        \|N_i(k_1)\|_{2p}^2 = \bigg\|\frac{1}{\sqrt{n}} \sum_{k=k_0}^{k_1-1} \xi(k+1) v_i(k)\bigg\|_{2p}^2
        \leq \frac{2p-1}{n}\sum_{k=k_0}^{k_1-1} \big\|\xi(k+1) v_i(k)\big\|_{2p}^2.
    \end{align*}
    By \eqref{eq:upper-bound:mart:xi-lp},
    \begin{align*}
        \|\xi(k+1) v_i(k)\big\|_{2p}^{2} &= \big(\E|\xi(k+1) v_i(k)|^{2p}\big)^{\frac1p} \\
        &\leq C_p \sqrt{n}\, \big(\E|\bm{b}(k)|_{2}^p\, |v_i(k)|^{2p}\big)^{\frac1p}
        \leq C_p \sqrt{n}\, \|b_i(k)\|_p,
    \end{align*}
    where in the last step we used $|v_i(k)|^p \leq 1$ and $|\bm{b}(k)|_2\, v_i(k)=b_i(k)$. By Jensen's inequality,
    \begin{align*}
        \E|\bm{N}(k_1)|_{2p}
        &\leq \bigg(\sum_{i=1}^n \|N_i(k_1)\|_{2p}^{2p}\bigg)^{\frac{1}{2p}} \\
        &\leq C_p n^{\frac14}\bigg(\sum_{i=1}^n
            \bigg(\frac{1}{n}\sum_{k=k_0}^{k_1-1} \, \|b_i(k)\|_p\bigg)^p
        \bigg)^{\frac{1}{2p}}
        \leq C_p n^{\frac14} \max_{k_0\leq k < k_1} \bigg(\sum_{i=1}^n \|b_i(k)\|_p^p\bigg)^{\frac{1}{2p}}.
    \end{align*}
    Since $\E|\bm{N}(k_1)|_{\infty} \leq \E|\bm{N}(k_1)|_{2p}$, the proof is complete.
\end{proof}

We are now ready for the proof of Proposition \ref{prop.vnupperbound}.

\begin{proof}[Proof of Proposition \ref{prop.vnupperbound}]
    We fix $n \in \N$, $k_0 \in \{0,1,\dots,n-1\}$, $\bx \in \R^n$, and $f \in \cA_p^\mathrm{F}(t_0)$ where $t_0=k_0/n$. Let $\bm{B}=(\bm{B}(t))_{t_0\leq t\leq 1}$ be a Brownian motion in $\R^n$, and let $\mathbb{F}=(\cF(t))_{t_0\leq t\leq 1}$ denote its canonical filtration. For each $i=1,\dots,n$, we now let $\alpha_i^f=(\alpha_i^f(t))_{t_0\leq t\leq 1}$ denote the Föllmer drift associated with $f$ and $B_i$, as defined in Lemma~\ref{lem.follmer}. In particular, each $\alpha_i^f$ is a mean-zero martingale satisfying
    \begin{equation*}
        \|\alpha_i^f(t)\|_p \leq \frac{\cstar}{2},\qquad t_0 \leq t \leq 1,
    \end{equation*}
    and the Föllmer process $X_i^f$ defined by
    \begin{equation*}
        X_i^f(t) \coloneqq \int_{t_0}^t \alpha_i^f(s)\,ds + (B_i(t) - B_i(t_0)),\qquad t_0 \leq t \leq 1,
    \end{equation*}
    satisfies $X_i^f(1) \sim f\,d\gamma_{1-t_0}$. Moreover, these processes are independent across $i$. Write $\bm{\alpha}^f$ and $\bm{X}^f$ for the corresponding $\R^n$-valued processes.

    Next, by Lemma~\ref{lem.updet.emp}, the infimum in the definition of $\cU_p^{\mathrm{det}}(t_0,m_{\bx}^n)$ is attained by an empirical measure. That is, there exists $\by\in\R^n$ such that
    \begin{align}\label{eq.upperbound.ydef}
    |\by|_\infty = \cU_p^{\mathrm{det}}(t_0,m_{\bx}^n),\qquad \Big(\frac1n\sum_{i=1}^n |y_i-x_i|^p\Big)^{\frac 1p}
    \le
    \frac{(1-t_0)\cstar}{2}.
    \end{align}
    For $t_0\leq t\leq 1$, we now set
    \begin{equation}\label{def.upperbound.drift.cont}
            \bm{\alpha}(t) \coloneqq \bm{\alpha}^f(t) + \bm{\alpha}^\mathrm{det},\qquad \bm{\alpha}^\mathrm{det} \coloneqq \frac{\by - \bx}{1 - t_0},
    \end{equation}
    and define the process $\bX$ by
    \begin{align*}
        \bX(t)\coloneqq \bm{x} + \int_{t_0}^t \bm{\alpha}(s)\,ds + \bm{B}(t) - \bm{B}(t_0) = \bm{x} + \bm{X}^f(t) + \frac{t-t_0}{1-t_0}(\by - \bx).
    \end{align*}
    At terminal time, $\bX(1) = \bm{y} + \bm{X}^f(1)$, so that
    \begin{equation}\label{upperbound.fixedf}
        \E|\bX(1)|_\infty \leq |\bm{y}|_\infty + \E|\bm{X}^f(1)|_\infty \leq \cU_p^{\mathrm{det}}(t_0,m_{\bx}^n)+[f]_\infty.
    \end{equation}

    Let us now define for $k=k_0,\dots,n-1$
    \begin{equation}\label{def.upperbound.drift.discrete}
    \bm b(k)\coloneqq \bm{b}^f(k)+\bm{b}^\mathrm{det},
    \qquad
    \bm b^\mathrm{det} \coloneqq \frac1n \bm{\alpha}^\mathrm{det},
    \end{equation}
    where $\bm b^f(k)$ is the $\ell^2$-projection of $\frac1n\bm\alpha^f(k/n)$ onto the (closed) $\ell^2$-ball of radius $\frac{\cstar}{2\sqrt{n}}$.
    Observe that by \eqref{eq.upperbound.ydef} and Jensen's inequality,
    \begin{align*}
        |\bm{b}(k)|_2
        &\leq |\bm{b}^f(k)|_2 + |\bm{b}^\mathrm{det}|_2 \\
        &= |\bm{b}^f(k)|_2 + \frac{1}{\sqrt{n}(1-t_0)} \Big(\frac1n\sum_{i=1}^n |y_i-x_i|^2\Big)^{\frac 12}  \leq \frac{\cstar}{2\sqrt{n}} + \frac{\cstar}{2\sqrt{n}} = \frac{\cstar}{\sqrt{n}}.
    \end{align*}
    Moreover, $\bm{b}(k)$ is measurable with respect to $\cG(k)\coloneqq \cF(\frac{k}{n})$, and thus the vectors $\bm{b}(k)$ are valid drift vectors in the sense of Lemma~\ref{lem.coupling.discrete}. Let $\bm{Z}=(\bm{Z}(k))$, $\eps=(\eps(k))$, and $\bm{Y}=(\bm{Y}(k))$ be as stated in the lemma.

    Since the triple $(\bbG,\bm{Z}, \epsilon)$ satisfies the assumptions of Lemma~\ref{lem.dpp}, we obtain from \eqref{upperbound.fixedf}
    \begin{equation*}
        \cV^n(k_0,\bm{x}) \leq \E|\bm{Y}(n)|_\infty \leq \cU^{\mathrm{det}}_{p}(t_0, m_{\bx}^n) + [f]_{\infty} + \E| \bm{Y}(n) - \bm{X}(1) |_{\infty}
    \end{equation*}
    Thus, once we show the estimate
    \begin{equation}\label{eq.coupling.bound}
        \E|\bY(n)-\bm X(1)|_\infty \le C_p n^{\frac{1}{2p}-\frac14},
    \end{equation}
    taking the infimum over $f\in \cA_p^\mathrm{F}(t_0)$ and using the definition of $\cU_p$ completes the proof.

    The remainder of this proof is devoted to the estimate \eqref{eq.coupling.bound}. Let us assume $k_0=0$ for notational simplicity; the general case is identical. Using the martingale decomposition from Lemma \ref{lem.coupling.discrete}, and the definition of $\bm{b}(k)$, we have
    \begin{equation*}
        \bm{Y}(n) = \bm{x} + \sum_{k=1}^n \bm{b}(k-1) + \sum_{k=1}^n \bm{\sigma}(k) = \by + \sum_{k=1}^n \bm{b}^f(k-1) + \sum_{k=1}^n \bm{\sigma}(k).
    \end{equation*}
    Since $\bX(1) = \by + \bX^f(1)$, a straightforward computation gives
    \begin{align*}
        \E| \bm{Y}(n) - \bm{X}(1) |_{\infty} \leq I_{\mathrm{mart}} + I_{\mathrm{drift},1} + I_{\mathrm{drift},2},
    \end{align*}
    where
    \begin{align} \label{errordefs}
        \nonumber
        I_{\text{mart}} &\coloneqq \E\bigg|\sum_{k = 1}^n \bm{\sigma}(k) - \bm{B}(1)\bigg|_\infty,
        \\
        \nonumber
        I_{\mathrm{drift},1} &\coloneqq \E \bigg|\sum_{k=0}^{n-1} \Big(\bm{b}^f(k) - \frac{1}{n}\bm{\alpha}^f(t_k)\Big)\bigg|_\infty,
        \\
        I_{\mathrm{drift},2} &\coloneqq \E \bigg|\frac{1}{n}\sum_{k=0}^{n-1} \bm{\alpha}^f(t_k) - \int_0^1 \bm{\alpha}^f(t)\,dt\bigg|_\infty.
    \end{align}
    Here, we write $t_k\coloneqq k/n$ for $k=0,\dots,n$. We now estimate each of these errors in turn. In the remainder of this proof, $C$ denotes a generic constant which can depend on $p$, but is independent of $n$, $\bx$, and $f$. In particular, $C$ may change from line to line.\smallskip

    \noindent\textit{The martingale term, \texorpdfstring{$I_\mathrm{mart}$}{}.} By the last part of Lemma~\ref{eq.coupling.discrete.marterror}, we have
    \begin{equation*}
        I_\mathrm{mart} \leq C n^{\frac14} \max_{k} \bigg(\sum_{i=1}^n \|b_i(k)\|_p^p\bigg)^{\frac{1}{2p}}.
    \end{equation*}
    By the definition of $\bm{b}(k)$ and using that $\|\alpha_i^f(t_k)\|_p\leq \cstar/2 \leq 1$, we have for each $i=1,\dots,n$
    \begin{equation*}
        \|b_i(k)\|_p \leq \frac{1}{n}\Big(\big\|\alpha_i^f(t_k)\big\|_p + \big|\alpha_i^\mathrm{det}\big|\Big) \leq \frac{1}{n}\big(1 + \big|\alpha_i^\mathrm{det}\big|\big).
    \end{equation*}
    Since $|\bm{\alpha}^\mathrm{det}|_p \leq n^{1/p}$ by definition, it follows that
    \begin{equation}\label{mart.est}
        I_\mathrm{mart} \leq C_p n^{-\frac14} \Big(n^{\frac1p} + |\bm{\alpha}^\mathrm{det}|_p\Big)^{\frac{1}{2}} \leq C n^{\frac{1}{2p} - \frac{1}{4}}.
    \end{equation}

    \noindent\textit{The projection error, \texorpdfstring{$I_{\mathrm{drift},1}$}{}.}
    Recall that $\bm{b}^f(k)$ is defined as the $\ell^2$-projection of the vector $\frac{1}{n}\bm{\alpha}^f(t_k)$ onto the $\ell^2$-ball of radius $\frac{\cstar}{2\sqrt{n}}$. Thus, writing $x_+ = \max(x,0)$, we have
    \begin{align*}
        \Big|\bm{b}^f (k) - \frac{1}{n} \bm{\alpha}^f(t_k)\Big|_\infty
        &= \frac{1}{n} \bigg(1 - \frac{\cstar/(2\sqrt{n})}{ \frac{1}{n}|\bm{\alpha}^f(t_k)|_2}\bigg)_+ \big|\bm{\alpha}^f(t_k)\big|_\infty \\
        &\leq \frac{1}{n\sqrt{n}}\bigg(\frac{2}{\cstar}\big|\bm{\alpha}^f(t_k)\big|_2 - \sqrt{n}\bigg)_+ \big|\bm{\alpha}^f(t_k)\big|_\infty.
    \end{align*}
    It follows that
    \begin{align}
        I_{\mathrm{drift},1}
        &= \E \bigg|\sum_{k = 0}^{n-1} \Big(\bm{b}^f (k) - \frac{1}{n} \bm{\alpha}^f (t_k)\Big)\bigg|_\infty \nonumber \\
        &\leq \frac{1}{\sqrt{n}} \frac{1}{n} \sum_{k = 0}^{n-1} \E\bigg[\bigg(\frac{2}{\cstar}\big|\bm{\alpha}^f(t_k)\big|_2 - \sqrt{n}\bigg)_+ \big|\bm{\alpha}^f(t_k)\big|_\infty\bigg]. \label{eq.proj.error.bound}
    \end{align}
    We estimate the expectation using Cauchy--Schwarz. First, using that the random variables $\alpha_i^f(t_k)$ are i.i.d.\ across $i$ and $L^p$-bounded by $\cstar/2 < 1$ for $p > 4$, we have
    \begin{equation*}
        \E\big|\bm{\alpha}^f(t_k)\big|_\infty^2 \leq \bigg(\sum_{i=1}^n \E|\alpha_i^f(t_k)|^4\bigg)^{\frac12} = \sqrt{n}\, \big\|\alpha_1^f(t_k)\big\|_4^2 \leq \sqrt{n}.
    \end{equation*}
    Second, the random variables $U_i=\frac{2}{\cstar}\alpha_i^f(t_k)$ are independent with mean zero and $\|U_i\|_4\leq 1$ and hence $\E\, U_i^2 \leq 1$. We can therefore use Lemma~\ref{lemma:norm-concentration} below to obtain
    \begin{equation*}
        \E\bigg[\bigg(\frac{2}{\cstar}\big|\bm{\alpha}^f(t_k)\big|_2 - \sqrt{n}\bigg)_+^2\bigg] \leq 2.
    \end{equation*}
    Combining this with \eqref{eq.proj.error.bound} gives
    \begin{equation} \label{drift1.bound}
        I_{\mathrm{drift},1} \leq \frac{\sqrt{2}}{\sqrt{n}}\, n^{\frac{1}{4}} = \sqrt{2}\, n^{-\frac{1}{4}}.
    \end{equation}

    \noindent \textit{Discretization error, \texorpdfstring{$I_{\mathrm{drift},2}$}{}.}
    We now turn our attention to $I_{\mathrm{drift},2}$. Using the crude bound $\max_i c_i \le \sqrt{\sum_i c_i^2}$, we have
    \begin{align*}
        I_{\mathrm{drift},2}
        \leq \sqrt{ \E \sum_{i=1}^n  \bigg|\frac{1}{n} \sum_{k=0}^{n-1}\alpha_i^f(t_k) - \int_{0}^1 \alpha_i^f(t)\,dt \bigg|^2}
         & = \sqrt{ n \E \bigg|\frac{1}{n}\sum_{k=0}^{n-1} \alpha_1^f(t_k) - \int_0^1 \alpha_1^f(t)\,dt \bigg|^2}                                       \\
         & = \sqrt{ n \E \bigg|\sum_{k = 0}^{n-1} \bigg(\frac{1}{n}\alpha_1^f(t_k) - \int_{t_k}^{t_{k+1}} \alpha_1^f(t)\,dt\bigg) \bigg|^2}
    \end{align*}
    where the second step used that the processes $\alpha_1^f,\dots,\alpha_n^f$ are i.i.d. Since $\alpha_1^f$ is a martingale by Lemma~\ref{lem.follmer}, the terms in the parentheses form a martingale difference sequence in $k$. Thus, if we expand the square, the cross-terms vanish, and we get
    \begin{align*}
        (I_{\mathrm{drift},2})^2
         & \le n \sum_{k = 0}^{n-1} \E \bigg( \int_{t_k}^{t_{k+1}} \big(\alpha_1^f(t) -  \alpha_1^f(t_k)\big)\,dt\bigg)^2
         \le \sum_{k = 0}^{n-1} \int_{t_k}^{t_{k+1}} \E \big(\alpha_1^f(t)-\alpha_1^f(t_k)\big)^2 \,dt,
    \end{align*}
    where the last step used Cauchy--Schwarz and that $t_{k+1} - t_k = 1/n$. Since $\alpha_1^f$ is a martingale,
    \begin{equation*}
        \E \big(\alpha_1^f(t)-\alpha_1^f(t_k)\big)^2 = \E \big(\alpha_1^f(t)\big)^2 - \E \big(\alpha_1^f(t_k)\big)^2 \le \E \big(\alpha_1^f(t_{k+1})\big)^2 - \E \big(\alpha_1^f(t_k)\big)^2,
    \end{equation*}
    for $t_k \le t \le t_{k+1}$. Hence,
    \begin{align*}
        (I_{\mathrm{drift},2})^2
         & \le \frac{1}{n} \sum_{k=0}^{n-1} \Big( \E \big(\alpha_1^f(t_{k+1})\big)^2 - \E \big(\alpha_1^f(t_k)\big)^2 \Big) = \frac{1}{n} \E \big(\alpha_1^f(1)\big)^2.
    \end{align*}
    Finally, use $\|\alpha_1^f(1)\|_2\leq \cstar/2 \leq 1$ to obtain
    \begin{align} \label{discretization.bound}
            I_{\text{drift},2} \leq n^{-\frac{1}{2}}.
    \end{align}

    Finally, combining \eqref{mart.est}, \eqref{drift1.bound} and \eqref{discretization.bound}, we obtain
    \begin{equation*}
        \E|\bm{Y}(n) - \bm{X}(1)|_\infty \leq C n^{\frac{1}{2p} - \frac{1}{4}} + \sqrt{2} n^{-\frac{1}{4}} + n^{-\frac{1}{2}} \leq C n^{\frac{1}{2p} - \frac{1}{4}}.
    \end{equation*}
    This proves the desired bound \eqref{eq.coupling.bound} and completes the proof.
\end{proof}

\begin{lemma}\label{lemma:norm-concentration}
	Let $U_1,\dots,U_n$ be i.i.d.\ random variables with $\E\,U_i=0$ and $\E\,U_i^2\leq 1$. Define $\bm{U}=(U_1,\dots,U_n)$. Then,
	\begin{align}
		\E\big[\big(|\bm{U}|_2-\sqrt{n}\big)_+^2\big]
		&\leq 2\Var(U_1^2) \leq 2\,\E\,U_1^4\label{eq:lemma:norm-concentration}.
	\end{align}
\end{lemma}

The proof is given in Appendix~\ref{sec:appendix-misc}.

\subsection{The upper half-relaxed limit}

For each $p > 4$, we now define a function $\cV^+_p : [0,1) \times \cP_p(\R) \to \R$ via
\begin{align} \label{def.halfrelaxed}
    \cV^+_p(t,m) \coloneqq \sup_{(k^n, \bx^n)} \limsup_{n \to \infty} \cV^n\big(k^n,\bx^n\big),
\end{align}
where the supremum is taken over all sequences $(k^n,\bx^n) \in \{0,1,\dots,n-1\} \times \R^n$ such that
\begin{align*}
    \frac{k^n}{n} \to t, \qquad m_{\bx^n}^n \xrightarrow{\bd_p} m, \qquad \text{  as } n \to \infty.
\end{align*}
This ``upper half-relaxed limit'' is a natural object in the theory of viscosity solutions, though we will not explicitly use any PDE arguments here. Instead, our goal will be to show that $\cV_p^+$ satisfies a ``dynamic programming inequality'' and useful upper bounds. We start by summarizing some basic properties.

\begin{lemma} \label{lem.halfrelaxed}
    For each $p > 4$, $\cV_p^+$ is upper semicontinuous on $[0,1)\times \P_p(\R)$ and satisfies
    \begin{align*}
        \cV_p^+(t,m) \leq \cU_p(t,m),\qquad (t,m)\in[0,1)\times\P_p(\R).
    \end{align*}
\end{lemma}

\begin{proof}
    The bound $\cV_p^+ \leq \cU_p$ comes directly from the definition of $\cV_p^+$, Proposition \ref{prop.vnupperbound}, and the upper semicontinuity of $\cU_p$ away from time $1$.

    To verify the upper semicontinuity of $\cV_p^+$, suppose that $(t_j, m_j)\to (t,m)$ in $[0,1)\times\cP_p(\R)$.
    Using the definition of $\cV_p^+(t_j,m_j)$, we can inductively choose integers $n_j>n_{j-1}$ along with $(k^{n_j}, \bx^{n_j}) \in \{0,1,\dots,n_j-1\} \times \R^{n_j}$ such that
    \begin{align*}
        \bigg| \frac{k^{n_j}}{n_j} - t_j \bigg
        | + \bd_p\big( m_{\bx^{n_j}}^{n_j}, m_j \big) \leq \frac{1}{j}, \qquad \Big| \cV^{n_j}\big(k^{n_j}, \bx^{n_j}\big) - \cV_p^+(t_j, m_j) \Big| \leq \frac{1}{j}.
    \end{align*}
    Since $t_j \to t$ and $m_j \to m$, we have $\frac{k^{n_j}}{n_j}\to t$ and $m_{\bx^{n_j}}^{n_j}\to m$ by the triangle inequality. By the definition of $\cV_p^+(t,m)$, we deduce that
    \begin{align*}
        \cV_p^+(t,m) \geq \limsup_{j \to \infty} \cV^{n_j}\big(k^{n_j}, \bx^{n_j}\big)
        = \limsup_{j \to \infty} \cV_p^+(t_j,m_j).
    \end{align*}
    This completes the proof.
\end{proof}

\subsubsection{Dynamic programming inequality}

The goal of this subsection is to prove that the upper half-relaxed limit $\cV^+_p$ defined above satisfies the following ``dynamic programming inequality".
\begin{proposition} \label{prop.dpi}
    Fix $p > 4$ and $(t_0,m_0) \in [0,1) \times \cP_p(\R)$. Suppose that $(\Omega, \cF, \bbF, \bP)$ is a filtered probability space hosting a Brownian motion $B$, and an $\cF(t_0)$-measurable random variable $X_0$ with $X_0 \sim m_0$. Let $\alpha$ be an $\bbF$-progressively measurable process with
    \begin{align*}
        \|\alpha(t)\|_2 \leq \cstar,\qquad \|\alpha(t)\|_{\infty} \leq C \qquad \text{ for a.e.\ $t\in[t_0,t_0+h]$},
    \end{align*}
    for $\cstar=\sqrt{2/\pi}$ and some arbitrary constant $C<\infty$. Then, for any $h\in(0,1-t_0)$, we have
    \begin{align*}
        \cV_p^+(t_0,m_0) \leq \cV_p^+\Big( t_0 + h, \cL\big( X(t_0 + h) \big)\Big),
    \end{align*}
    where $X$ is defined by
    \begin{align*}
        X(t) \coloneqq X_0 + \int_{t_0}^t \alpha(s)\, ds + B(t) - B(t_0), \quad t_0 \leq t \leq 1.
    \end{align*}
\end{proposition}
We note that, using a superposition principle as in the proof of Lemma \ref{lem.equivstrong}, we can deduce from Proposition \ref{prop.dpi} the following corollary.
\begin{corollary} \label{cor.dpi}
    Fix $p > 4$ and $(t_0,m_0) \in [0,1) \times \cP_p(\R)$. Let $(m,\alpha) \in \cA(t_0,m_0)$ be such that
    \begin{equation*}
        |\alpha(t,x)| \leq C,\qquad\text{for a.e.\ } (t,x)\in[t_0,t_0+1]\times\R,
    \end{equation*}
    where $C>0$ is an arbitrary constant. Then, for any $h\in(0,1-t_0)$, we have
    \begin{align*}
        \cV_p^+(t_0,m_0) \leq \cV_p^+\big( t_0 + h, m_{t_0 + h} \big).
    \end{align*}
\end{corollary}

Before proving Proposition \ref{prop.dpi}, we need several preliminary steps.

\begin{lemma} \label{lem.glivenkocantelli}
    Suppose that $\bx^n = (x^n_1,\dots,x^n_n) \in \R^n$, and $h^n \in \R_+$ are such that
    \begin{align*}
        m_{\bx^n}^n \xlongrightarrow{\bd_p} m, \qquad h^n \longrightarrow h,
    \end{align*}
    for some $m \in \cP_p(\R)$ and $h \in (0,\infty)$. Suppose moreover that for each $n$,
    \begin{align*}
        \bG^n = \big(G^n_1,\dots,G^n_n \big) \sim \gamma_{h^n}^{\otimes n}.
    \end{align*}
   Then,
    \begin{align*}
       \bd_p\big( m_{\bx^n + \bG^n}^n, m * \gamma_h \big) \longrightarrow 0\quad \text{  in probability.}
    \end{align*}
\end{lemma}

The proof uses the following elementary criterion for convergence in probability for random measures in $\cP_p(\R)$, whose proof we leave to the reader. For measures $m$ and integrable functions $\phi$, we write $\langle \phi, m \rangle = \int \phi\,dm$.

\begin{lemma} \label{lem.randommeasures}
    Let $(m^n)_{n\in\N}$ be a sequence of random elements of $\cP_p(\R)$, and let $m \in \cP_p(\R)$. Suppose that the following properties hold:
    \begin{enumerate}
        \item for every test function $\phi \in C_b(\R)$, $\langle \phi, m^n \rangle \to \langle \phi, m \rangle$ in probability,
        \item for all $\eps > 0$, there exist $R > 0$ and $N \in \N$ such that
        \begin{align*}\bP \bigg[ \int_{\R} |x|^p \mathbf{1}_{\{|x| > R\}} dm^n > \eps \bigg] < \eps,\qquad \text{ for all } n \geq N.
        \end{align*}
    \end{enumerate}
    Then,
    \begin{align*}
        \bd_p(m^n,m) \longrightarrow 0\quad \text{in probability.}
    \end{align*}
\end{lemma}

\begin{proof}[Proof of Lemma~\ref{lem.glivenkocantelli}]
    First, notice that if we replace $\bm{G}^n$ with $\wt{\bm{G}}^n = \sqrt{\frac{h}{h^n}} \bm{G}^n$, then $\wt{\bG}^n \sim \gamma_h^{\otimes n}$, and
    \begin{align*}
        \bd_p\Big( m_{\bx^n + \bm{G}^n}^n, m_{\bx^n + \wt{\bm{G}}^n}^n \Big) \leq \bigg(\frac{1}{n} \sum_{i = 1}^n |G_i^n - \wt{G}_i^n|^p\bigg)^{\frac{1}{p}} = \Big|\sqrt{\frac{h^n}{h}} - 1\Big| \bigg(\frac{1}{n} \sum_{i = 1}^n |\wt{G}_i^n|^p \bigg)^{\frac{1}{p}}.
    \end{align*}
    The right-hand side converges to $0$ in probability since $h^n\to h$ and
    \begin{equation*}
        \E\bigg[\frac{1}{n}\sum_{i=1}^n |\wt G_i^n|^p\bigg]
        = \E|N(0,h)|^p < \infty.
    \end{equation*}
    Here and below, $N(0,h)$ denotes a generic random variable with law $\gamma_h$. Thus, it is enough to prove the claim under the additional assumption that $h^n=h$ for all $n$.

    We now aim to apply Lemma~\ref{lem.randommeasures} to show convergence in probability of $m^n_{\bx^n + \bm{G}^n}$ to $m * \gamma_h$. Fix a test function $\phi \in C_b(\R)$, and notice that
    \begin{align*}
        \E \big[ \langle \phi, m^n_{\bx^n +\bG^n  } \rangle \big] = \frac{1}{n} \sum_{i = 1}^n \E\Big[ \phi\big(x^n_i + G^n_i\big) \Big] = \frac{1}{n} \sum_{i = 1}^n \int_\R \phi(x^n_i+ y)\,\gamma_h(dy) = \langle \phi, m_{\bx^n}^n * \gamma_h \rangle.
    \end{align*}
    Since $m_{\bx^n}^n \to m$ in $\bd_p$, and hence weakly, we deduce that $m_{\bx^n}^n * \gamma_h \to m * \gamma_h$ weakly, and hence
    \begin{align*}
         \E \big[ \langle \phi, m^n_{\bx^n +\bG^n  } \rangle \big] \to \langle \phi, m * \gamma_h \rangle,
    \end{align*}
    for each $\phi \in C_b(\R)$. In addition,
    \begin{align*}
        \text{Var}\big( \langle \phi, m^n_{\bx^n +\bG^n  } \rangle\big) = \text{Var}\Big(\frac{1}{n} \sum_{i = 1}^n \phi\big(x^n_i + G^n_i \big) \Big) \leq \frac{\|\phi\|_{\infty}^2}{n}.
    \end{align*}
    It now follows from Markov's inequality that
    \begin{align}
        \langle \phi, m^n_{\bx^n +\bG^n  } \rangle \longrightarrow \langle \phi, m * \gamma_h \rangle\qquad \text{ in probability,}
    \end{align}
    for each test function $\phi \in C_b(\R)$.

    It remains to see that $m_{\bx^n + \bm{G}^n}^n$ satisfies the tightness property (2) of Lemma~\ref{lem.randommeasures}. Fix an arbitrary $\eps>0$ and note that for every $R>0$
    \begin{align*}
        &\int_{\R} |x|^p \mathbf{1}_{\{|x| > R\}} \, m_{\bx^n + \bm{G}^n}^n(dx) = \frac{1}{n} \sum_{i = 1}^n |x^n_i + G^n_i|^p \mathbf{1}_{\{|x^n_i + G^n_i| > R\}}
        \\
        &\qquad \leq  \frac{2^p}{n}\sum_{i = 1}^n |x_i^n|^p \mathbf{1}_{\{|x_i^n| > R/2\}} + \frac{2^p}{n} \sum_{i = 1}^n |G_i^n|^p \mathbf{1}_{\{|G_i^n| > R/2\}}.
    \end{align*}
    Since $m_{\bx^n}^n$ converges in $\cP_p(\R)$, we may choose $R>0$ and $N\in\N$ such that the first term is at most $\eps/2$ for all $n\geq N$. On the other hand, by Cauchy--Schwarz
    \begin{align*}
        \E\Big[ \frac{1}{n} \sum_{i = 1}^n |G_i^n|^p \mathbf{1}_{\{|G_i^n| > R/2\}} \Big]
        &\leq \frac{1}{n} \sum_{i = 1}^n \E\big[ |G_i^n|^{2p} \big]^{1/2} \bP\big[ |G_i^n| > R/2 \big]^{1/2} \\
        &= \E\big[ |N(0,h)|^{2p} \big]^{1/2} \gamma_h\big( [-R/2,R/2]^c \big)^{1/2},
    \end{align*}
    which converges to zero as $R\to\infty$. Thus, enlarging $R$ if necessary, we may also assume that
    \begin{align*}
        \bP\bigg[\frac{2^p}{n} \sum_{i = 1}^n |G_i^n|^p \mathbf{1}_{\{|G_i^n| > R/2\}} > \frac{\eps}{2}\bigg] < \eps
    \end{align*}
    for all $n$. Combining the preceding estimates, we obtain
    \begin{align*}
        \bP\bigg[\int_{\R} |x|^p \mathbf{1}_{\{|x| > R\}} \, m_{\bx^n + \bm{G}^n}^n(dx) > \eps\bigg] < \eps
    \end{align*}
    for all $n \geq N$, as required.
\end{proof}

The next step is to prove a version of Proposition \ref{prop.dpi} when the control $\alpha$ is constant in time.

\begin{lemma} \label{lem.dpiconstant}
    Suppose that $p > 4$, $(\Omega, \cF, \bP)$ is any probability space, $t_0 \in [0,1)$, $X_0 \in L^p(\Omega)$, $\widebar\alpha \in L^{\infty}(\Omega)$ with $\|\widebar\alpha\|_2 \leq \cstar$, and $h \in (0,1-t_0)$. Then
    \begin{align*}
        \cV_p^+\big( t_0, \cL(X_0) \big) \leq \cV_p^+\big( t_0 + h, \cL(X_0 + h \widebar\alpha) * \gamma_h \big).
    \end{align*}
\end{lemma}

\begin{proof}
    We proceed as follows. First, we approximate the joint law of $(X_0,\widebar{\alpha})$ by empirical measures of a suitable sequence $(\bx^n,\ba^n)\in\R^n\times\R^n$. We then define $\R^n$-valued processes $\bX^n$ and $\bY^n$ using a coupling construction similar in spirit to that in the proof of Proposition~\ref{prop.vnupperbound}. Roughly speaking, $\bX^n$ represents $n$ realizations of the limiting state process, started from $X_0$ at time $t_0$ and under the constant control $\widebar{\alpha}$, while $\bY^n$ is the discrete state process started from $\bx^n$, driven by a control $\eps^n$ chosen so that it has constant drift equal to $\frac{1}{n}\ba^n$. We show that the empirical measures of $\bX^n$ and $\bY^n$ at time $t_0+h$ are asymptotically close, and the claim then follows by combining this with the discrete dynamic programming inequality and passing to the limit.

    \medskip
    \noindent\emph{Step 1: Approximation of $\cL(X_0, \widebar \alpha)$.} Fix $\delta>0$ and choose a sequence $(k^n, \bx^n)$ such that
    \begin{align}\label{eq.dpiconstant.pf.kx}
        \frac{k^n}{n} \to t_0, \quad m_{\bx^n}^n \xrightarrow{\bd_p} \cL(X_0),
    \end{align}
    and
    \begin{align}\label{eq.dpiconstant.pf.xn}
        \limsup_{n \to \infty} \cV^n(k^n, \bx^n) \geq \cV_p^+\big( t_0, \cL(X_0) \big) - \delta.
    \end{align}

    We now construct a sequence of vectors $\ba^n \in \R^n$ and some constant $R<\infty$ such that
    \begin{align} \label{anprops.a}
        |\ba^n|_2 \leq \cstar \sqrt{n}, \qquad | \ba^n |_{\infty} \leq R,\qquad\text{for all $n\in\N$,}
    \end{align}
    as well as
    \begin{equation}\label{anprops.convergence}
        m_{(\bx^n, \ba^n)}^n \xrightarrow{\bd_p} \cL(X_0, \widebar{\alpha}).
    \end{equation}
    To do so, let $(\wt{\bx}^n,\wt{\ba}^n)\in\R^n\times\R^n$ be any sequence such that its joint empirical measure converges to $\cL(X_0,\widebar{\alpha})$ in $\bd_p$. Since $m_{\bx^n}^n \to \cL(X_0)$ in $\bd_p$ by assumption, we have $\bd_p\big(m_{\wt{\bx}^n}^n,m_{\bx^n}^n\big)\to 0$. Therefore, after relabeling the coordinates of $(\wt{\bx}^n,\wt{\ba}^n)$ simultaneously if necessary, but without relabeling the coordinates of $\bx^n$, we may assume that
    \begin{align*}
        \frac1n\sum_{i=1}^n |\wt x_i^n-x_i^n|^p \to 0.
    \end{align*}
    This quantity bounds $\bd_p^p\big(m_{(\bx^n,\wt{\ba}^n)}^n,m_{(\wt{\bx}^n,\wt{\ba}^n)}^n\big)$ from above, while the relabeling does not change the empirical measure $m_{(\wt{\bx}^n,\wt{\ba}^n)}^n$. Hence we still have $m_{(\wt{\bx}^n,\wt{\ba}^n)}^n \to \cL(X_0,\widebar{\alpha})$ in $\bd_p$, and so the triangle inequality yields $m_{(\bx^n,\wt{\ba}^n)}^n \to \cL(X_0,\widebar{\alpha})$ in $\bd_p$.

    Now, set $R\coloneqq \|\widebar{\alpha}\|_\infty + 1$, and define vectors $\ba^{n,R}\in\R^n$ by
    \begin{equation*}
        a_i^{n,R} \coloneqq \wt{a}_i^n \mathbf{1}_{\{|\wt{a}_i^n| \leq R\}},\qquad i=1,\dots,n.
    \end{equation*}
    We define $\ba^n$ to be the $\ell^2$-projection of $\ba^{n,R}$ onto the $\ell^2$-ball of radius $\cstar\sqrt{n}$, i.e.,
    \begin{equation*}
        \ba^n \coloneqq \lambda_n\ba^{n,R},\qquad \lambda_n \coloneqq \min\Big\{1, \frac{\cstar\sqrt{n}}{|\ba^{n,R}|_2}\Big\},
    \end{equation*}
    where we use the convention $\lambda_n=1$ when $\ba^{n,R}=0$. Observe that \eqref{anprops.a} holds since $|\ba^n|_\infty \leq |\ba^{n,R}|_\infty \leq R < \infty$ and $|\ba^n|_2 \leq \sqrt{n} \cstar$. We prove the convergence \eqref{anprops.convergence} as follows. Let $\varphi$ be a continuous function with $\mathbf{1}_{(R,\infty)} \leq \varphi \leq \mathbf{1}_{(R-1,\infty)}$. Since $m_{\wt{\ba}^n}^n\to\cL(\widebar{\alpha})$ in $\bd_p$ and $\varphi(\widebar\alpha)=0$ a.s.\ by the definition of $R$, we have
    \begin{equation*}
        \frac1n \sum_{i=1}^{n} \big|a_i^{n,R} - \wt{a}_i^n\big|^p = \frac1n \sum_{i=1}^{n} |\wt{a}_i^n|^p \mathbf{1}_{\{|\wt{a}_i^n| > R\}} \leq \frac1n \sum_{i=1}^{n} |\wt{a}_i^n|^p \varphi(|\wt{a}_i^n|) \longrightarrow \E|\widebar{\alpha}|^p \varphi(|\widebar{\alpha}|) = 0.
    \end{equation*}
    In particular, $m_{\ba^{n,R}}^n \to \cL(\widebar{\alpha})$ in $\bd_p$, and in turn $n^{-1/2} |\ba^{n,R}|_2 \to \|\widebar{\alpha}\|_2 \leq \cstar$. Hence, $\lambda_n\to 1$ and so
    \begin{equation*}
        \bigg(\frac1n\sum_{i=1}^{n} \big|a_i^n - a_i^{n,R}\big|^p\bigg)^{\frac1p}
        = \bigg(\frac1n \sum_{i=1}^{n} \big|a_i^{n,R}\big|^p\bigg)^{\frac1p} (1-\lambda_n)
        \leq R(1-\lambda_n) \longrightarrow 0.
    \end{equation*}
    It follows that
    \begin{equation*}
        \bd_p(m_{(\bx^n, \ba^n)}^n, m_{(\bx^n, \wt{\ba}^n)}^n) \leq \bigg(\frac1n\sum_{i=1}^{n} \big|a_i^n - a_i^{n,R}\big|^p\bigg)^{\frac1p} + \bigg(\frac1n \sum_{i=1}^{n} \big|a_i^{n,R} - \wt{a}_i^n\big|^p\bigg)^{\frac1p} \longrightarrow 0.
    \end{equation*}
    Since $m_{(\bx^n,\wt{\ba}^n)}^n \to \cL(X_0,\widebar{\alpha})$ in $\bd_p$, we get \eqref{anprops.convergence} by the triangle inequality.

    \medskip
    \noindent\emph{Step 2: Coupling construction.} For $i\in\N$, let $(B_i(t))_{0\le t\le 1}$ be independent Brownian motions defined on a common probability space $(\Omega,\cF,\PP)$. For each $n$, write $t^n = k^n / n$, and define an $\R^n$-valued process $\bX^n$ on the interval $[t^n,1]$ by
    \begin{equation}\label{eq.dpiconstant.pf.xdef}
        \bX^n(t) \coloneqq \bx^n + (t-t^n)\ba^n + \bm{B}^n(t) - \bm{B}^n(t^n),\qquad t^n\le t\le 1,
    \end{equation}
    where $\bm{B}^n(t)\coloneqq (B_1(t),\dots,B_n(t))$. We write $\bbF^n = (\cF^n(t))_{t\in[0,1]}$ for the filtration generated by $\bm{B}^n$.

    Next, we construct a discrete state process $\bY^n$ starting from $\bx^n$ at time $k^n$. To this end, we apply Lemma~\ref{lem.coupling.discrete} to the constant sequence
    \begin{equation*}
        \bm{b}^n(k) \coloneqq \frac{1}{n}\ba^n,\qquad k=k^n,\dots,n-1,
    \end{equation*}
    and the Brownian motion $\bm{B}^n$. By the first part of \eqref{anprops.a}, the assumptions of the lemma are satisfied. We thus obtain Gaussian increments $\bZ^n = (\bZ^n(k))_{k=k^n+1,\dots,n}$ and a discrete control $\eps^n = (\eps^n(k))_{k=k^n+1,\dots,n}$ for the discrete filtration $\bbG^n = (\cG^n(k))$, defined by $\cG^n(k) \coloneqq \cF^n(k/n)$, such that $(\bbG^n,\bZ^n,\eps^n)$ satisfies the assumptions of Lemma~\ref{lem.dpp} and
    \begin{equation*}
        \frac{1}{n}\ba^n
        =
        \E\big[\eps^n(k+1)\bZ^n(k+1)\,\big|\,\cG^n(k)\big],
        \qquad k=k^n,\dots,n-1.
    \end{equation*}
    In particular, the state process $\bY^n(k) \coloneqq \bx^n + \sum_{\ell=k^n+1}^k \eps^n(\ell)\bZ^n(\ell)$ admits the decomposition
    \begin{equation}\label{eq.dpiconstant.pf.ydef}
        \bY^n(k)
        =
        \bx^n + \frac{k-k^n}{n}\ba^n + \sum_{\ell=k^n+1}^k \bm{\sigma}^n(\ell),
        \qquad k=k^n,\dots,n,
    \end{equation}
    where $\bm{\sigma}^n(\ell) \coloneqq \eps^n(\ell)\bZ^n(\ell) - \E\big[\eps^n(\ell)\bZ^n(\ell)\,\big|\,\cG^n(\ell-1)\big]$. Moreover, by the last part of Lemma~\ref{lem.coupling.discrete}, we have for every $\ell^n=0,\dots,n-k^n$ and $h^n=\ell^n / n$
    \begin{align*}
        \E\bigg|
            \sum_{k=k^n+1}^{k^n+\ell^n}\bm{\sigma}^n(k)
            -
            \big(
                \bm{B}^n(t^n+h^n)
                -
                \bm{B}^n(t^n)
            \big)
        \bigg|_\infty
        &\le
        n^{\frac14}
        \max_k \bigg(
            \sum_{i=1}^n \|b_i(k)\|_p^p
        \bigg)^{\frac{1}{2p}} \\
        &=
        n^{-\frac14}\sqrt{|\ba^n|_p} \\
        &\le
        \sqrt{R}\,n^{\frac{1}{2p}-\frac14}.
    \end{align*}
    Since the initial condition and drift terms in \eqref{eq.dpiconstant.pf.xdef} and \eqref{eq.dpiconstant.pf.ydef} agree, we can rewrite this estimate as
    \begin{equation}\label{eq.dpiconstant.pf.compstate}
        \E\big|\bX^n(t^n+h^n)-\bY^n(k^n+\ell^n)\big|_\infty
        \le
        \sqrt{R}\,n^{\frac{1}{2p}-\frac14}.
    \end{equation}

    \medskip
    \noindent\emph{Step 3: Comparison via dynamic programming.} Fix a sequence $(\ell^n)$ of integers such that
    \begin{equation*}
        0\le \ell^n \le n-k^n,
        \qquad
        h^n \coloneqq \frac{\ell^n}{n} \longrightarrow h.
    \end{equation*}
    Since $(\bbG^n,\bZ^n,\eps^n)$ satisfies the assumptions of Lemma~\ref{lem.dpp}, we have
    \begin{equation*}
        \cV^n(k^n,\bx^n)
        \le
        \E\Big[\cV^n\big(k^n+\ell^n,\bY^n(k^n+\ell^n)\big)\Big].
    \end{equation*}
    Moreover, since $p > 4$ and $\cV^n(k^n+\ell^n,\,\cdot\,)$ is $1$-Lipschitz with respect to $\ell^\infty$ by Lemma~\ref{lem.vn.1lip}, the estimate \eqref{eq.dpiconstant.pf.compstate} yields
    \begin{align*}
        \E\Big|
            \cV^n\big(k^n+\ell^n,\bX^n(t^n+h^n)\big)
            -
            \cV^n\big(k^n+\ell^n,\bY^n(k^n+\ell^n)\big)
        \Big|
        = o_n(1).
    \end{align*}
    We obtain
    \begin{equation}\label{eq.dpiconstant.pf.comparex}
        \cV^n(k^n,\bx^n)
        \le
        \E\Big[\cV^n\big(k^n+\ell^n,\bX^n(t^n+h^n)\big)\Big] + o_n(1).
    \end{equation}

    \medskip
    \noindent\emph{Step 4: Identification of the limiting terminal law.} Recalling \eqref{anprops.convergence}, we have
    \begin{equation*}
        m_{\bx^n + h^n \ba^n}^n \xlongrightarrow{\bd_p} \cL(X_0 + h\widebar{\alpha}).
    \end{equation*}
    On the other hand, by definition of $\bX^n$,
    \begin{equation}\label{eq.dpiconstant.pf.xn.hn}
        \bX^n(t^n+h^n)
        =
        \bx^n + h^n \ba^n + \bm{B}^n(t^n+h^n)-\bm{B}^n(t^n).
    \end{equation}
    Lemma~\ref{lem.glivenkocantelli} therefore implies that
    \begin{equation}\label{eq.dpiconstant.pf.inprob}
        m_{\bX^n(t^n+h^n)}^n
        \xrightarrow{\bd_p}
        \cL(X_0 + h\widebar{\alpha}) * \gamma_h
        \qquad
        \text{in probability.}
    \end{equation}

    \medskip
    \noindent\emph{Step 5: Passage to the limit and conclusion.} Let $(n_j)$ be a subsequence such that
    \begin{equation*}
        \cV^{n_j}(k^{n_j},\bx^{n_j})
        \longrightarrow
        \limsup_{n\to\infty}\cV^n(k^n,\bx^n),
    \end{equation*}
    and such that the convergence in \eqref{eq.dpiconstant.pf.inprob} holds almost surely along $(n_j)$. By the definition of $\cV_p^+$, it follows that
    \begin{equation*}
        \limsup_{j\to\infty}
        \cV^{n_j}\big(k^{n_j}+\ell^{n_j},\bX^{n_j}(t^{n_j}+h^{n_j})\big)
        \le
        \cV_p^+\big(t_0+h,\cL(X_0+h\widebar{\alpha}) * \gamma_h\big)
        \qquad
        \text{a.s.}
    \end{equation*}
    We claim that the sequence
    \begin{equation}\label{eq.dpiconstant.pf.ui}
        \cV^n\big(k^n+\ell^n,\bX^n(t^n+h^n)\big),
        \qquad n\in\N,
    \end{equation}
    is uniformly integrable. Assuming this for the moment, we may combine \eqref{eq.dpiconstant.pf.xn} and \eqref{eq.dpiconstant.pf.comparex} to obtain
    \begin{align*}
        \cV_p^+\big(t_0,\cL(X_0)\big)
        &\le
        \delta + \limsup_{n\to\infty}\cV^n(k^n,\bx^n) \\
        &\le
        \delta + \limsup_{j\to\infty}
        \E\Big[\cV^{n_j}\big(k^{n_j}+\ell^{n_j},\bX^{n_j}(t^{n_j}+h^{n_j})\big)\Big] \\
        &\le
        \delta + \E\Big[
            \limsup_{j\to\infty}
            \cV^{n_j}\big(k^{n_j}+\ell^{n_j},\bX^{n_j}(t^{n_j}+h^{n_j})\big)
        \Big] \\
        &\le
        \delta + \cV_p^+\big(t_0+h,\cL(X_0+h\widebar{\alpha}) * \gamma_h\big).
    \end{align*}
    Since $\delta>0$ was arbitrary, this proves the claim.

    It remains to justify the uniform integrability. To this end, fix $4<p'<p$. By Proposition~\ref{prop.vnupperbound}, there exists a constant $C>0$, independent of $n$, such that
    \begin{align*}
        \cV^n\big(k^n+\ell^n,\bX^n(t^n+h^n)\big)
        &\le
        \cU_{p'}^\mathrm{det}\big(t^n+h^n,m_{\bX^n(t^n+h^n)}^n\big)
        + \cU_{p'}^\mathrm{F}(t^n+h^n)
        + C n^{\frac{1}{2p'}-\frac{1}{4}}.
    \end{align*}
    By Lemma~\ref{lem.follmervalue}, the term $\cU_{p'}^\mathrm{F}$ is bounded. Moreover, by Lemma~\ref{lem.updet.upperbound}, the first term is bounded up to a multiplicative constant by
    \begin{equation*}
        \big(1-t^n-h^n\big)^{-\frac{p'}{p-p'}}
        \bigg(\frac{1}{n}\sum_{i=1}^n \big|X_i^n(t^n + h^n)\big|^p\bigg)^{\frac{1}{p-p'}}.
    \end{equation*}
    Since $t^n+h^n\to t_0+h<1$, the prefactor is bounded for all sufficiently large $n$. Using \eqref{eq.dpiconstant.pf.xn} and the triangle inequality, we can therefore bound this quantity up to a constant by
    \begin{equation*}
        \bigg(\frac{1}{n}\sum_{i=1}^n |x_i^n|^p\bigg)^{\frac{1}{p-p'}}
        +
        \bigg(\frac{1}{n}\sum_{i=1}^n |a_i^n|^p\bigg)^{\frac{1}{p-p'}}
        +
        \bigg(
            \frac1n\sum_{i=1}^n
            \Big|
                B_i(t^n+h^n)-B_i(t^n)
            \Big|^p
        \bigg)^{\frac{1}{p-p'}}.
    \end{equation*}
    The first two terms are uniformly bounded in $n$ because $m_{\bx^n}^n$ and $m_{\ba^n}^n$ converge in $\bd_p$. The third term is uniformly bounded in $L^r$ for every $r>1$ by Jensen's inequality. We conclude that the sequence \eqref{eq.dpiconstant.pf.ui} is uniformly integrable.
\end{proof}

We now use Lemma \ref{lem.dpiconstant} together with an approximation argument to prove Proposition \ref{prop.dpi}.

\begin{proof}[Proof of Proposition \ref{prop.dpi}]
    First, suppose that $\alpha$ is piecewise constant, i.e., for some partition
    \begin{align*}
        t_0 < t_1 < \cdots < t_k = t_0 + h,
    \end{align*}
    we have
    \begin{equation*}
        \alpha(t) = \sum_{i=0}^{k-1} \alpha(t_i) \mathbf{1}_{[t_i,t_{i+1})}(t),
        \qquad t_0 \leq t \leq t_0+h.
    \end{equation*}
    In particular, for each $i$,
    \begin{equation*}
        X(t_{i+1})
        = X(t_i) + (t_{i+1}-t_i)\alpha(t_i) + B(t_{i+1}) - B(t_i),
    \end{equation*}
    and thus
    \begin{equation*}
        \cL\big(X(t_{i+1})\big)
        =
        \cL\big(X(t_i) + (t_{i+1}-t_i)\alpha(t_i)\big) * \gamma_{t_{i+1}-t_i}.
    \end{equation*}
    Since $\alpha(t_i)$ is $\cF(t_i)$-measurable, belongs to $L^\infty(\Omega)$, and satisfies $\|\alpha(t_i)\|_2 \leq \cstar$, we may apply Lemma~\ref{lem.dpiconstant} on each interval $[t_i,t_{i+1}]$ to obtain
    \begin{align*}
        \cV_p^+\big(t_i,\cL(X(t_i))\big)
        \leq
        \cV_p^+\big(t_{i+1},\cL(X(t_{i+1}))\big).
    \end{align*}
    The desired inequality then follows by induction.

    Now consider the general case. For each $k \in \N$, write
    \begin{equation*}
        t_i^k \coloneqq t_0 + \frac{ih}{k}, \qquad i=0,\dots,k,
    \end{equation*}
    and define a piecewise constant approximation $\alpha^k$ by
    \begin{align*}
        \alpha^k(t)
        \coloneqq
        \sum_{i=0}^{k-1} \alpha^{k,i} \mathbf{1}_{[t_i^k,t_{i+1}^k)}(t),
        \qquad
        \alpha^{k,i}
        \coloneqq
        \begin{cases}
            0 & i = 0,
            \\
            \displaystyle \frac{k}{h} \int_{t_{i-1}^k}^{t_i^k} \alpha(s)\,ds & i \in \{1,\dots,k-1\}.
        \end{cases}
    \end{align*}
    Then $\alpha^{k,i}$ is $\cF(t_i^k)$-measurable for each $i$, and
    \begin{equation*}
        \|\alpha^{k,i}\|_\infty \leq C,
        \qquad
        \|\alpha^{k,i}\|_2 \leq \cstar
    \end{equation*}
    by Jensen's inequality.

    Let $X^k$ denote the corresponding state process, i.e.,
    \begin{align*}
        X^k(t)
        \coloneqq
        X_0 + \int_{t_0}^t \alpha^k(s)\,ds + B(t)-B(t_0),
        \qquad t_0 \leq t \leq t_0+h.
    \end{align*}
    By the piecewise constant case, we have
    \begin{align*}
        \cV_p^+\big(t_0,\cL(X_0)\big)
        \leq
        \cV_p^+\big(t_0+h,\cL\big(X^k(t_0+h)\big)\big)
    \end{align*}
    for every $k$.

    Moreover,
    \begin{equation*}
        X^k(t_0+h)-X(t_0+h)
        =
        -\int_{t_{k-1}^k}^{t_0+h} \alpha(s)\,ds,
    \end{equation*}
    and therefore, since $\|\alpha(s)\|_\infty \leq C$ for almost every $s$,
    \begin{equation*}
        \|X^k(t_0+h)-X(t_0+h)\|_\infty
        \leq
        \frac{Ch}{k}
        \to 0.
    \end{equation*}
    In particular,
    \begin{equation*}
        \cL\big(X^k(t_0+h)\big) \to \cL\big(X(t_0+h)\big)
        \qquad\text{in } \cP_p(\R).
    \end{equation*}
    By the upper semicontinuity of $\cV_p^+$, we conclude
    \begin{align*}
        \cV_p^+\big(t_0,\cL(X_0)\big)
        \leq
        \limsup_{k\to\infty}
        \cV_p^+\big(t_0+h,\cL\big(X^k(t_0+h)\big)\big)
        \leq
        \cV_p^+\big(t_0+h,\cL\big(X(t_0+h)\big)\big).
    \end{align*}
\end{proof}

\subsection{Completion of the proof}

We remind the reader that at this stage, we have produced for each $p > 4$ a function $\cV_p^+ : [0,1) \times \cP_p(\R) \to \R$ which satisfies the following properties. First, directly from the definition in \eqref{def.halfrelaxed}, we have that
\begin{align} \label{fact1.vpplus}
    \limsup_{n \to \infty} V^n \leq \cV_p^+(0, \delta_0).
\end{align}
Second, by Lemma \ref{lem.halfrelaxed}, we have $\cV_p^+(t,m) \leq \cU_p(t,m)$ and
\begin{align} \label{fact2.vpplus}
    \cV_p^+ \text{  is upper semicontinuous (with respect to $\bd_p$)},
\end{align}
on $[0,1) \times \cP_p(\R)$.
Combining this with Lemma \ref{lem.follmervalue} and Lemma \ref{lem.updet.upperbound}, we see that
\begin{align} \label{fact3.vpplus}
    \cV_p^+(t,m) \leq [m]_{\infty} + \kappa(t),
\end{align}
where $\kappa(t) \coloneqq \sup_{t \leq s \leq 1} \cU_p^{\mathrm{F}}(s) \to 0 \text{  as  } t \to 1.$

Finally, we know from Corollary \ref{cor.dpi} that $\cV_p^+$ satisfies a dynamic programming inequality: for each $t_0 \in [0,1)$, $h \in (0,1-t_0)$, $m_0 \in \cP_p(\R)$, and $(m,\alpha) \in \cA(t_0,m_0)$ such that $\|\alpha\|_{L^{\infty}([t_0,1] \times \R)} < \infty$, we have
\begin{align} \label{fact4.vpplus}
        \cV_p^+(t_0,m_0) \leq \cV_p^+\big( t_0 + h, m_{t_0 + h} \big)
\end{align}
Our goal is now to combine \eqref{fact1.vpplus}, \eqref{fact2.vpplus}, \eqref{fact3.vpplus} and \eqref{fact4.vpplus} to complete the proof of \eqref{pf:upperbound-final}.

As a final preparation, we will need an approximation result because the dynamic programming inequality \eqref{fact4.vpplus} has been established only for bounded controls $\alpha$. To state this precisely, we introduce a truncated version of the optimization problem which defines $V^{\infty}$. In particular, for $0 < T \leq 1$, we consider the problem
\begin{equation} \label{optimization.truncated}
\inf_{(m,\alpha) \in \cA(0,\delta_0)} [m_T]_{\infty}.
\end{equation}
We note that the optimization problem in \eqref{optimization.truncated} is just a reparameterization of the problem which defines $\cV^{\infty}(1-T,\delta_0)$, i.e.,
\begin{equation}
    \inf_{(m,\alpha) \in \cA(0,\delta_0)} [m_T]_{\infty} = \cV^{\infty}(1-T, \delta_0).
\end{equation}
The following lemma, whose proof is postponed to the next section, explains that we can find good approximations of optimizers of \eqref{optimization.truncated} with bounded $\alpha$.

\begin{lemma} \label{lem.Wpapproximation}
    Fix $T \in (0,1]$ and $p \geq 1$. Then, there exists an optimizer $(m,\alpha)$ for \eqref{optimization.truncated} and a sequence $(m^{(k)},\alpha^{(k)}) \in \cA(0,\delta_0)$ such that $\|\alpha^{(k)}\|_{L^{\infty}([0,T]\times \R)} < \infty$, and
    \begin{align*}
      \bd_p\big( m_T^{(k)}, m_T \big)  \longrightarrow 0.
    \end{align*}
\end{lemma}

We are now ready to complete the proof of the upper bound \eqref{pf:upperbound-final} of Theorem \ref{th:intro}.
    Fix any $p > 4$, and then fix $\eps > 0$. Let $(m,\alpha)$ and $(m^{(k)}, \alpha^{(k)})$ be as in the statement of Lemma \ref{lem.Wpapproximation}, with $T = 1 - \eps$. Because each $\alpha^{(k)}$ is bounded, we obtain from \eqref{fact4.vpplus} that
    \begin{align*}
        \cV^+_p(0,\delta_0) \leq \cV^+_p(1- \eps, m_{1-\eps}^{(k)}).
    \end{align*}
    By \eqref{fact2.vpplus},
    \begin{align*}
        \cV^+_p(0,\delta_0) \leq \limsup_{k \to \infty} \cV^+_p(1- \eps, m_{1-\eps}^{(k)}) \leq \cV^+_p(1- \eps, m_{1-\eps}).
    \end{align*}
    Then applying \eqref{fact3.vpplus}, we get
    \begin{align*}
        \cV^+_p(0,\delta_0) &\leq [m_{1-\eps}]_{\infty} + \kappa(\eps) = \cV^{\infty}(\eps, \delta_0) + \kappa(\eps)  \leq \cV^{\infty}(0,\delta_0) + \kappa(\eps).
    \end{align*}
    with the equality following from the optimality of $(m,\alpha)$ for \eqref{optimization.truncated}, and the inequality $\cV^{\infty}(\eps, \delta_0) \le \cV^{\infty}(0, \delta_0)$ coming from Lemma \ref{lem.nondecreasing}. Sending $\eps \to 0$ completes the proof. This completes the proof of \eqref{pf:upperbound-final}, pending the proof of Lemma \ref{lem.Wpapproximation}, given at the end of the next section.

\section{Duality and Properties of Optimizers}
\label{sec:duality-and-optimizers}

\label{sec.duality}

In this section, we analyze the structure of the limiting problem $V^\infty$ and ultimately prove Lemma \ref{lem.Wpapproximation}, which asserts that there exists an optimizer for $ \cV^{\infty}(0,\delta_0)$, as defined in \ref{eq:Vinf-originaldef}, for which the corresponding measure $m_t$ can be approximated in $p$-Wasserstein distance for arbitrary $p \ge 1$. This is not as easy as one might expect; because the controls are constrained in $L^2$ norm, a naive truncation argument does not work for $p > 2$.
To achieve this we will develop some duality results which let us deduce symmetry properties of optimizers, most notably Proposition \ref{lem.optproperties} which shows that there exists an optimizer for $V^\infty$ satisfying certain symmetry properties which imply higher-order integrability.

We remark that the results in this section hold more generally if we replace the time interval $[0,1]$ by $[0,T]$ for an arbitrary finite horizon $T>0$. This will be needed in the proof of the upper bound in Section~\ref{sec.upperbounds}. To simplify notation, we will work only with $T=1$.

Fix a complete probability space $(\Omega, \cF, \mathbb{P})$ supporting a Brownian motion $B$ starting at zero and let $\mathbb{F}$ be the (completed) filtration generated by $B$. The set of controls $\mathbb{A}$ is defined as the set of processes $(\alpha(t))_{t\in[0,1]}$ that are $\mathbb{F}$-progressively measurable and satisfy
\begin{equation*}
    \|\alpha(t)\|_2 \leq c_0 \qquad\text{for a.e. } t\in[0,1].
\end{equation*}
Here, $c_0 \coloneqq \sqrt{2/\pi}$, though the results in this section hold for any other positive constant. By Lemma~\ref{lem.equivstrong}, we have
\begin{equation*}
    V^\infty = \inf_{\alpha\in \mathbb{A}} \|X^\alpha(1)\|_{\infty},\qquad X^\alpha(t) \coloneqq \int_0^t \alpha(s)\,ds + B(t),\quad t\in[0,1].
\end{equation*}
The constraint set $\mathbb{A}$ is convex, and the objective is a convex functional of $\alpha$.
This can be viewed as a convex optimization problem in a suitable Banach space.
As a consequence of Theorem~\ref{thm:duality-regularized} below, derived from the Fenchel--Rockafellar duality theorem, we will see that an optimal $\alpha$ exists, and the following strong duality formula holds:
\begin{align*}
    V^\infty =
    \sup \bigg\{
        \E[YB(1)] - c_0 \int_0^1 \big\|\E[Y\,|\,\F(t)]\big\|_2\,dt
        :
        \begin{array}{c}
            Y:\Omega\to \R \mathrm{\ s.t.\ } \|Y\|_1 \leq 1,\\
            \,
            \E|Y|\sqrt{\log(1+|Y|)}<\infty
        \end{array}
        \bigg\}.
\end{align*}
It will become apparent that $V^\infty$ is the Fenchel dual of the optimization problem on the right-hand side, which we call the ``primal'' problem.
A natural strategy would be to try to use this strong duality formula to prove structural properties of the optimal $\alpha$.
However, it turns out that there does not exist $Y$ attaining the supremum on the right-hand side, essentially because it would have to attain the supremum of $\E[Y X^\alpha(1)]$ over the unit ball in $L^1$, and $X^\alpha(1)$ is nonatomic. This makes it difficult to extract any useful structural properties on the optimal $\alpha$ for $V^\infty$. (In a sense, an optimal $Y$ should exist in a sufficiently ``large'' space like the dual of $L^\infty$, but we did not find this viewpoint to be sufficiently tractable.) As a workaround, we introduce a family of regularized control problems that approximate $V^\infty$ and for which the corresponding primal problems admit a unique optimizer. From strong duality we will then deduce structural properties of primal ($Y$) and dual ($\alpha$) optimizers, and the relevant properties of the optimal $\alpha$ can be passed to the limit and deduced for $V^\infty$.

\subsection{Regularized problem} \label{sec.duality.regularized}
Let $p\in[2,\infty]$ and $\delta\in[0,\infty)$, we consider the regularized problem
\begin{equation}\label{eq:regularized-control-problem}
    V_{p,\delta}^{\infty} \coloneqq \inf_{\alpha\in\mathbb{A}} \Big\{ \|X^\alpha(1)\|_{p} + \frac{\delta}{2}\E\int_0^1\alpha(t)^2\,dt \Big\}.
\end{equation}
Note in particular that the values $p=\infty$ and $\delta=0$ are allowed, so that $V^{\infty}_{\infty,0} =V^\infty$. For large $p<\infty$ and small $\delta>0$, the problem $V_{p,\delta}^{\infty}$ regularizes $V^\infty$ by approximating the $L^\infty$ norm with the $L^p$ norm and adding a quadratic cost on the control. Note that $V_{p,\delta}^{\infty}$ is finite since it is bounded from above by $V^\infty + \delta c_0^2 / 2<\infty$. We will see that $V^\infty$ is finite in Section~\ref{sec.upperbounds}.

To state the general duality result, we define a function $\varphi_\delta : \R_+ \to \R$ by
\begin{align*}
    \varphi_\delta(y)
    \coloneqq
    \sup_{|a| \le c_0 }\Big(ay - \frac{\delta}{2}a^2\Big).
\end{align*}
If $\delta>0$, this supremum is attained uniquely at $a = \min(c_0, x/\delta)$, and we have the explicit formula
\begin{equation}\label{eq:varphi-delta-explicit}
    \varphi_\delta(y) =
    \begin{cases}
        c_0  y - \frac{\delta c_0^2}{2} &\text{if }y > \delta c_0, \\
        \frac{1}{2\delta}y^2 &\text{if } y \leq \delta c_0.
    \end{cases}
\end{equation}
If $\delta=0$, then $\varphi_0(y) = c_0  y$, and the supremum is attained at $a=c_0$ (uniquely when $y>0$).

\begin{theorem}\label{thm:duality-regularized}
Let $p\in[2,\infty]$ and $\delta\in[0,\infty)$. Let $q=p/(p-1)$ denote the conjugate exponent, with $q=1$ when $p=\infty$. We have
\begin{align*}
    V_{p,\delta}^{\infty} =
    \sup
    \bigg\{
        \E[YB(1)] - \int_0^1\varphi_\delta\big(\big\|\E[Y\,|\,\F(t)]\big\|_{2}\big)\,dt
        :
        \begin{array}{c}
            Y:\Omega\to \R \mathrm{\ s.t.\ } \|Y\|_{q} \leq 1,\\
            \,
            \E|Y|\sqrt{\log(1+|Y|)}<\infty
        \end{array}
    \bigg\}
    \eqqcolon P_{q,\delta}.
\end{align*}
Moreover, there exists $\alpha\in\mathbb{A}$ such that the minimization problem $V_{p,\delta}^{\infty}$ is attained.
\end{theorem}

Note here that the constraint $\E|Y|\sqrt{\log(1+|Y|)}<\infty$ in the definition of $P_{q,\delta}$ is redundant when $q>1$ (i.e., $p<\infty$). It is needed for $q=1$ to ensure that $YB(1)$ is integrable and to prove the corresponding weak duality.

The proof of Theorem~\ref{thm:duality-regularized} is essentially an application of the Fenchel--Rockafellar duality, with $V_{p,\delta}^{\infty}$ arising as the dual problem of $P_{q,\delta}$. We thus refer to $P_{q,\delta}$ as the primal problem of $V_{p,\delta}^{\infty}$. We first introduce a suitable choice of function spaces. Let $L_\mathbb{F}^{\infty,2}$ denote the Banach space of progressively measurable processes $\alpha=(\alpha(t))_{t\in[0,1]}$ equipped with the norm
\begin{equation*}
	\|\alpha\|_{\infty,2}\coloneqq \esssup_{t\in[0,1]} \|\alpha(t)\|_{2}.
\end{equation*}
The set of controls $\mathbb{A}$ is then exactly the closed ball in $L_\mathbb{F}^{\infty,2}$ of radius $c_0$. We also define the space $L_\mathbb{F}^{1,2}$ of progressively measurable processes $\beta:[0,1]\times\Omega\to\R$, with norm given by
\begin{equation*}
	\|\beta\|_{1,2} \coloneqq \int_0^1 \|\beta(t)\|_{2}\,dt.
\end{equation*}
The space $L_\mathbb{F}^{\infty,2}$ can be identified with the topological dual of $L_\mathbb{F}^{1,2}$, the duality pairing being
\begin{equation*}
	L_\mathbb{F}^{\infty,2}\times L_\mathbb{F}^{1,2} \ni (\alpha,\beta) \mapsto \int_0^1 \E[\alpha(t) \beta(t)]\,dt.
\end{equation*}
See \cite[Section 2]{lu-2012} for a proof of this duality result.

Before beginning the proof of Theorem \ref{thm:duality-regularized}, we recall that given a Banach space $\mathcal{X}$ and a convex function $\phi : \mathcal{X} \to \R \cup \{\infty\}$, the Fenchel conjugate of $\phi$ is defined by
\begin{align*}
    \phi^* : \mathcal{X}^* \to \R, \quad \phi^*(x^*) = \sup_{x \in \mathcal{X}} \Big\{ \langle x, x^* \rangle - \phi(x)  \Big\},
\end{align*}
where $\mathcal{X}^*$ denotes the topological dual of $\mathcal{X}$ and $\langle x, x^* \rangle = x^*(x)$ denotes the duality pairing.\color{black}

\begin{proof}[Proof of Theorem~\ref{thm:duality-regularized}]
Let us prove the duality formula for the case $\delta>0$ and explain at the end how to modify the proof for $\delta=0$. We define the following functions:
\begin{alignat*}{6}
    f :\;& L^2(\Omega) \;&\longrightarrow\;& \R \cup \{\infty\}, \qquad
    & f(Z) \;&\coloneqq\;
    \begin{cases}
    -\E[Z B(1)] & \text{if } \|Z\|_{q} \le 1,\\
    \infty & \text{otherwise},
    \end{cases}\\
    g :\;& L_{\mathbb F}^{1,2} \;&\longrightarrow\;& \R, \qquad
    & g(\beta) \;&\coloneqq\;
    c_0 \,\|\beta\|_{1,2},\\
    h :\;& L_{\mathbb F}^{1,2} \;&\longrightarrow\;& \R, \qquad
    & h(\beta) \;&\coloneqq\;
    \frac{1}{2\delta}\E\int_0^1 \beta(t)^2\,dt,\\
    A:\;& L^2(\Omega) \;&\longrightarrow\;& L_\mathbb{F}^{1,2}, \qquad
    & (AZ)(t) \;&\coloneqq\; \E[Z\,|\,\mathcal{F}(t)],\quad t\geq 0.
\end{alignat*}
Note that we can ensure the progressive measurability of $AZ$ by choosing an appropriate modification of the martingale $\E[Z\,|\,\mathcal{F}(t)]$.

The functions $f$, $g$, and $h$ are convex, and $A$ is a bounded linear operator. It is straightforward to compute the Fenchel conjugates of $f$, $g$, and $h$, as well as the adjoint of $A$:
\begin{alignat*}{6}
    f^*:\;& L^2(\Omega) \;&\longrightarrow\;& \R\cup\{\infty\}, \qquad
    & f^*(Y) \;&=\; \|Y + B(1)\|_p,\\
    g^*:\;& L_{\mathbb F}^{\infty,2} \;&\longrightarrow\;& \R\cup\{\infty\}, \qquad
    & g^*(\alpha) \;&=\;
        \begin{cases}
        0 &\text{if } \alpha\in \mathbb{A},\\
        \infty &\text{otherwise},
        \end{cases}\\
    h^*:\;& L_{\mathbb F}^{\infty,2} \;&\longrightarrow\;& \R, \qquad
    & h^*(\alpha) \;&=\; \frac{\delta}{2}\E\int_0^1 \alpha(t)^2\,dt,\\
    A^*:\;& L_{\mathbb F}^{\infty,2} \;&\longrightarrow\;& L^2(\Omega), \qquad
    & A^*\alpha \;&=\; \int_0^1 \alpha(t)\,dt.
\end{alignat*}
Noting that $\|X^\alpha(1)\|_{p} = f^*(A^*\alpha)$ for $\alpha\in\mathbb{A}$, we can now write the problem $V_{p,\delta}^{\infty}$ as
\begin{equation}\label{eq:V-p-delta-fenchel-form}
    V_{p,\delta}^{\infty} = -\sup_{\alpha \in L^{\infty,2}_{\FF}}\Big(-f^*(A^*\alpha) - g^*(-\alpha) - h^*(-\alpha)\Big).
\end{equation}
Note also that $g^*+h^*=(g\square h)^*$, where $g \square h$ is the infimal convolution, which can be computed explicitly:
\begin{align*}
    g \square h(\beta)
    &\coloneqq \inf_{\alpha \in L^{1,2}_{\FF}}\big(g(\alpha) + h(\beta-\alpha) \big) \\
	&= \inf_{\alpha \in L^{1,2}_{\FF}} \int_0^1 \bigg(c_0\|\alpha(t)\|_{2} + \frac{1}{2\delta}\|\beta(t)-\alpha(t)\|_{2}^2\bigg)\,dt \\
	&= \int_0^1\inf_{\eta \in \R}  \bigg(c_0|\eta|\|\beta(t)\|_{2} + \frac{1}{2\delta}(\eta-1)^2\|\beta(t)\|_{2}^2\bigg)\,dt \\
	&= \int_0^1  \varphi_\delta\big( \|\beta(t)\|_{2}\big)\,dt.
\end{align*}
Indeed, the third line follows because it is optimal to choose $\beta(t)=\eta(t)\alpha(t)$ for some non-random $\eta(t) \in \R$, and the fourth follows from an explicit calculation. Since $g\square h: L_{\mathbb{F}}^{1,2}\to\R$ is continuous, we may apply the Fenchel--Rockafellar theorem (see, e.g., condition (4.3.2) and Theorem~4.4.3 in \cite{borwein-2005-variational}), yielding
\begin{align}\label{eq:duality-proof:prelim-duality}
    V_{p,\delta}^{\infty}
    &= -\inf_{Y \in L^2(\Omega)} \big(f(Y) + (g\square h)(AY)\big) \nonumber\\
	&= \sup \bigg\{\E[YB(1)] - \int_0^1\varphi_\delta(\|\E[Y\,|\,\F(t)]\|_{2}) \, dt  : \|Y\|_{q} \le 1, \, Y \in L^2(\Omega) \bigg\}.
\end{align}

The claimed duality formula only differs in the constraint; the right-hand side of \eqref{eq:duality-proof:prelim-duality} is smaller than $P_{q,\delta}$ because $q \in [1,2]$. To complete the proof, we will argue through weak duality that $V^\infty_{p,\delta} \ge P_{p,\delta}$.
Take an arbitrary $\alpha\in\mathbb{A}$, and let $Y\in L^q(\Omega)$ with $\|Y\|_{q}\leq 1$ and $\E|Y|\sqrt{\log(1+|Y|)}<\infty$ (the latter condition being redundant when $q>1$). We now show that the objective for $\alpha$ is at least as large as the objective for $Y$. We may assume that these quantities are finite and so $\|X^\alpha(1)\|_p < \infty$ and $t\mapsto \|\E[Y\,|\,\F(t)]\|_2$ is integrable on $[0,1]$.

We start by noting that
\begin{equation}\label{eq:duality-proof:weak-duality:1}
    \|X^\alpha(1)\|_{p}
        \ge \E[Y X^\alpha(1)]
        = \E[Y B(1)] + \E\bigg[\int_0^1 Y \alpha(t)\,dt\bigg],
\end{equation}
where we used that $YX^\alpha(1)$ and $YB(1)$ are integrable to split the expectation. For the second term, we would now like to swap the expectation and integral, and then condition on $\mathcal{F}(t)$ to replace $Y$ with $Y(t)\coloneqq \E[Y\,|\,\F(t)]$, and thus obtain
\begin{equation}\label{eq:claim:fubini-issue}
    \E\bigg[Y\int_0^1\alpha(t)\,dt\bigg] = \int_0^1 \E[Y(t) \alpha(t)]\,dt.
\end{equation}
However, the use of Fubini's theorem is not justified a priori, since $Y$ may not be in $L^2$, and so $(t,\omega)\mapsto Y(\omega)\alpha(t,\omega)$ may not be jointly integrable. We will resolve this detail at the end of the proof, but for now assume that \eqref{eq:claim:fubini-issue} holds, and we continue the argument. We have
\begin{align}
    \|X^\alpha(1)\|_{p} + \frac{\delta}{2} \int_0^1 \E\,\alpha(t)^2\,dt
    &\geq \E[Y B(1)] + \int_0^1 \E\Big[Y(t)\alpha(t) + \frac{\delta}{2}\alpha(t)^2\Big]\,dt \nonumber\\
    &\geq \E[Y B(1)] - \int_0^1 \Big(\|Y(t)\|_2 \|\alpha(t)\|_2 - \frac{\delta}{2}\|\alpha(t)\|_2^2\Big)\,dt
    \label{eq:duality-proof:weak-duality:2}\\
    &\geq \E[Y B(1)] - \int_0^1 \varphi_\delta\big(\|Y(t)\|_{2}\big)\,dt,
    \label{eq:duality-proof:weak-duality:3}
\end{align}
by the definition of $\varphi_\delta$ and since $\|\alpha(t)\|_2\leq c_0$. In particular, this implies $V_{p,\delta}^{\infty}\geq P_{q,\delta}$, and combined with \eqref{eq:duality-proof:prelim-duality}, we obtain $V_{p,\delta}^{\infty}=P_{q,\delta}$.

The existence of a dual optimizer, i.e., an optimal control for $V_{p,\delta}^{\infty}$, is ensured by the Fenchel--Rockafellar theorem. Alternatively, this may be proved directly from the compactness of $\mathcal D$ and lower semicontinuity with respect to the weak-$*$ topology on $L_{\mathbb F}^{\infty,2}$.

We now justify \eqref{eq:claim:fubini-issue}. Recall that $t\mapsto \|Y(t)\|_2$ is integrable by assumption. Fix $s<1$, and use Fubini's theorem and $Y(s)\in L^2$ to write
\begin{equation}\label{eq:claim:fubini-issue:prelim}
    \E\bigg[Y(s)\int_0^1\alpha(t)\,dt\bigg] = \int_0^1 \E[Y(s)\alpha(t)]\,dt = \int_0^s \E[Y(t)\alpha(t)]\,dt + \int_s^1 \E[Y(s)\alpha(t)]\,dt.
\end{equation}
We now take $s\to 1$. The first term on the right-hand side converges to the right-hand side of \eqref{eq:claim:fubini-issue} since $Y(t,\omega)\alpha(t,\omega)$ is jointly integrable. For the second term, we estimate
\begin{equation*}
    \int_s^1 \E[Y(s)\alpha(t)]\,dt \leq c_0  \int_s^1\|Y(s)\|_{2}\,dt \leq c_0  \int_s^1\|Y(t)\|_{2}\,dt\longrightarrow 0,
\end{equation*}
using that $Y(t)$ is a martingale and the integrability of $t\mapsto \|Y(t)\|_2$. It remains to see that the left-hand side of \eqref{eq:claim:fubini-issue:prelim} converges to that of \eqref{eq:claim:fubini-issue}. Since $\mathbb{F}$ is the Brownian filtration, we have $Y(s)\to Y$ a.s., and so it is enough to argue that $Y(s)\int_0^1\alpha(t)\,dt$ is uniformly integrable in $s$. If $p<\infty$, this follows from $Y\in L^q$ and $\int_0^1\alpha(t)\,dt = X^\alpha(1) - B(1) \in L^p$. Similarly, if $p=\infty$, we use the condition $\E|Y|\sqrt{\log(1+|Y|)}<\infty$ and the fact that $\int_0^1\alpha(t)\,dt$ is subgaussian as the sum of $-B(1)$ and the bounded random variable $X^\alpha(1)$.

The theorem is proved in exactly the same way for the case $\delta=0$, except that we take $h=h^*\equiv 0$ and $\varphi_0(y)=c_0 y$.
\end{proof}

\subsection{Optimality conditions for the regularized problems}

Let us now turn our attention to the existence and characterization of optimizers of $V_{p,\delta}^{\infty}$ and $P_{q,\delta}$ for $p<\infty$.

\begin{proposition}\label{prop:primal-optimizers-finite-p}
    Let $p\in[2,\infty)$ and $\delta\in[0,\infty)$, and write $q=p/(p-1)$ for the conjugate exponent. The optimization problems $V_{p,\delta}^{\infty}$ and $P_{q,\delta}$ admit unique optimizers $\alpha$ and $Y$. Denoting $X=X^\alpha$ and $Y(t)=\E[Y\,|\,\mathcal{F}(t)]$, we have
    \begin{equation}
        0 < \|X(1)\|_p <\infty,\qquad \int_0^1 \|Y(t)\|_2\,dt <\infty,\qquad \|Y(t)\|_2 > 0\quad\text{for $t>0$},
        \label{eq:primal-optimizers-finite-p:Yt-normfinite-positive}
    \end{equation}
    and the following duality relations hold almost surely:
    \begin{equation}
        Y = \frac{X(1) |X(1)|^{p-2}}{\|X(1)\|_p^{p-1}}, \qquad \alpha(t) =-\lambda(t)Y(t)\quad\text{for a.e.\ $t\in(0,1)$},
        \label{eq:primal-optimizers-finite-p:duality-identities}
    \end{equation}
    where $\lambda(t) \in (0,\infty)$ is defined by
    \begin{equation}
        \lambda(t) \coloneqq \min\bigg(\frac{c_0}{\|Y(t)\|_2},\, \frac{1}{\delta}\bigg) ,\qquad t\in(0,1).
        \label{eq:primal-optimizers-finite-p:def-lambda}
    \end{equation}
    If $\delta=0$, we interpret the minimum as $\lambda(t)=c_0/\|Y(t)\|_2$. Lastly, the function $t\mapsto \lambda(t)$ is continuous, and if $\delta>0$, it is bounded from above.

\end{proposition}

Before we prove Proposition~\ref{prop:primal-optimizers-finite-p}, we note that the optimality conditions \eqref{eq:primal-optimizers-finite-p:duality-identities} and \eqref{eq:primal-optimizers-finite-p:def-lambda} can equivalently be written in terms of a forward-backward stochastic differential equation. We obtain a corresponding characterization of the (unique) optimizers $\alpha$ and $Y$ as a direct corollary.

\begin{corollary}\label{cor.duality.fbsde}
    Let $p\in[2,\infty)$ and $\delta\in[0,\infty)$, and write $q=p/(p-1)$ for the conjugate exponent. The unique optimizers $\alpha$ and $Y$ of the problems $V_{p,\delta}^{\infty}$ and $P_{q,\delta}$ are given by
    \begin{align*}
        \alpha(t) = - \lambda(t) Y(t),\quad t\in(0,1),\qquad Y=Y(1),
    \end{align*}
    where $(X(t),Y(t),Z(t))$, $0\leq t \leq 1$, is the unique strong solution of the forward-backward stochastic differential equation
    \begin{equation*}
        \left\{
            \begin{alignedat}{2}
                dX(t) &= -\min\bigg(\frac{c_0}{\|Y(t)\|_2},\, \frac{1}{\delta}\bigg) Y(t)\,dt + dB(t),\qquad&& X(0) = 0,\\
                dY(t) &= Z(t)\, dB(t),\qquad&& Y(1) = \frac{X(1) |X(1)|^{p-2}}{\|X(1)\|_p^{p-1}},
            \end{alignedat}
        \right.
    \end{equation*}
    and $\lambda:(0,1)\to[0,\infty)$ is defined as in \eqref{eq:primal-optimizers-finite-p:def-lambda}.
\end{corollary}

\begin{proof}[Proof of Corollary~\ref{cor.duality.fbsde}]
    Let $\alpha$ and $Y$ be the unique optimizers of $V_{p,\delta}^{\infty}$ and $P_{q,\delta}$, respectively. Write $X(t)=X^\alpha(t)$. Since $Y(t)=\E[Y\,|\,\F(t)]$ is a martingale in the Brownian filtration $\mathbb{F}$, it can be uniquely represented in the form $dY(t) = Z(t)\,dB(t)$ for some progressively measurable process $Z$. The duality relations \eqref{eq:primal-optimizers-finite-p:duality-identities} and \eqref{eq:primal-optimizers-finite-p:def-lambda} immediately imply that $(X(t),Y(t),Z(t))$, $0\leq t \leq 1$, solves the forward-backward system above.

    The uniqueness of the forward-backward system follows from the uniqueness of the optimizers $\alpha$ and $Y$. This is because any solution $(X(t),Y(t),Z(t))$, $0\leq t \leq 1$, to the forward-backward system defines optimizers of $V_{p,\delta}^{\infty}$ and $P_{q,\delta}$ by taking $\alpha(\cdot)$ to be the drift of $X(\cdot)$ and setting $Y\coloneqq Y(1)$.
\end{proof}

\begin{proof}[Proof of Proposition~\ref{prop:primal-optimizers-finite-p}]
    We know that $V_{p,\delta}^{\infty}$ admits an optimal control $\alpha\in\mathbb{A}$ from Theorem~\ref{thm:duality-regularized}. Since $p<\infty$, we can also argue directly that $P_{q,\delta}$ admits an optimizer $Y\in L^q(\Omega)$ with $\|Y\|_q\le 1$. Indeed, the constraint set is weakly compact in $L^q$, and we now show that the objective is upper semicontinuous with respect to the weak topology on $L^q$. Take $Y_n\in L^q$ such that $Y_n\to Y$ weakly. Then, for any $t\in(0,1)$ and $Z\in L^p$, we have
    \begin{equation*}
        \E\big[\E[Y_n\,|\,\F(t)] Z\big] = \E\big[Y_n \E[Z\,|\,\F(t)]\big] \longrightarrow \E\big[Y \E[Z\,|\,\F(t)]\big] = \E\big[\E[Y\,|\,\F(t)] Z\big],
    \end{equation*}
    which implies $\E[Y_n\,|\,\F(t)]\to \E[Y\,|\,\F(t)]$ weakly for every $t$. The $L^2$ norm $\|\cdot\|_2$ is weakly lower semicontinuous and $\phi_\delta$ is continuous, and so Fatou's lemma yields
    \begin{equation*}
        \E[YB(1)] - \int_0^1 \varphi_{\delta}( \| \E[Y \,|\, \cF(t)] \|_2 )\, dt \geq \limsup_{n\to\infty} \E[Y_n B(1)] - \int_0^1 \varphi_{\delta}( \| \E[Y_n \,|\, \cF(t)] \|_2 )\, dt.
    \end{equation*}
    It follows that the objective is weakly upper semicontinuous, and so $P_{q,\delta}$ is attained.

    Let $\alpha$ and $Y$ be optimizers for $V_{p,\delta}^{\infty}$ and $P_{q,\delta}$, respectively, and let us show why \eqref{eq:primal-optimizers-finite-p:Yt-normfinite-positive} and \eqref{eq:primal-optimizers-finite-p:duality-identities} hold, with $X=X^\alpha$ and $Y(t)=\E[Y\,|\,\F(t)]$. The uniqueness of $\alpha$ will follow from the existence of an optimal $Y$ with $\|Y(t)\|_2>0$, since the second identity of \eqref{eq:primal-optimizers-finite-p:duality-identities} uniquely specifies $\alpha$. Similarly, the uniqueness of $Y$ will follow from the existence of an optimal control $\alpha$ and the first identity of \eqref{eq:primal-optimizers-finite-p:duality-identities}.

    The fact that $\|X(1)\|_p$ is finite and $\|Y(t)\|_2$ is integrable in $t$ comes immediately from the finiteness of $V_{p,\delta}^{\infty}$ and $P_{q,\delta}$. By the uniform $L^2$ bound on the drift $\alpha$ and a standard entropy estimate \cite[Proposition 1]{lehec2013representation}, the law of $X$ on path space is absolutely continuous with respect to Wiener measure, which implies that $\|X(t)\|_p > 0$ for any $t>0$, and in particular $\|X(1)\|_p>0$.

    It remains to see that $\|Y(t)\|_2>0$ for $t>0$, and that the duality identities \eqref{eq:primal-optimizers-finite-p:duality-identities} hold. To this end, inspect the inequalities \eqref{eq:duality-proof:weak-duality:1}, \eqref{eq:duality-proof:weak-duality:2}, and \eqref{eq:duality-proof:weak-duality:3} in the proof of Theorem~\ref{thm:duality-regularized}, which we used to establish weak duality. Since $\alpha$ and $Y$ are optimal, these inequalities must collapse to equality, and this occurs precisely when
    \begin{equation*}
        \|X(1)\|_{p} = \E[YX(1)],\qquad \E\Big[Y(t)\alpha(t) + \frac{\delta}{2} \alpha(t)^2\Big] = - \varphi_\delta( \|Y(t)\|_{2})\quad \text{for a.e.\ } t.
    \end{equation*}
    The first identity is equivalent to the first identity in \eqref{eq:primal-optimizers-finite-p:duality-identities}. In particular, this specifies the optimizer $Y$ for $P_{q,\delta}$ uniquely. To see what the second identity implies, consider first the case $\delta>0$. Recalling the definition of $\varphi_\delta$ and \eqref{eq:varphi-delta-explicit}, we have for a.e.\ $t$
    \begin{equation}\label{eq:primal-optimizers-finite-p:proof:delta-positive}
        \alpha(t) \|Y(t)\|_2 = -Y(t)\|\alpha(t)\|_2\quad\text{a.s.},\qquad \|\alpha(t)\|_2 = \min\bigg(c_0,\, \frac{\|Y(t)\|_2}{\delta}\bigg).
    \end{equation}
    Similarly, if $\delta=0$, we have for a.e.\ $t$
    \begin{equation}\label{eq:primal-optimizers-finite-p:proof:delta-zero}
        \alpha(t) \|Y(t)\|_2 = -c_0 Y(t)\quad\text{a.s.}.
    \end{equation}
    In either case, whenever $\|Y(t)\|_2>0$, we obtain $\alpha(t)=-\lambda(t)Y(t)$, where $\lambda(t)$ is given by \eqref{eq:primal-optimizers-finite-p:def-lambda}. Now note that because $Y(\cdot)$ is a martingale, $\|Y(t)\|_2$ is nondecreasing in $t$, and so it only remains to show that $t^*=\inf\{t>0: \|Y(t)\|_2>0\}$ is equal to zero.

    The rest of the proof is devoted to proving $t^*=0$. Since $\mathbb{F}$ is the filtration generated by $B$, we may assume that $(\Omega,\mathcal{F},\mathbb{F},\mathbb{P})$ is the canonical setup. That is, $B(\cdot)$ is the canonical process on $\Omega=C[0,1]$, $\mathbb{F}=(\mathcal{F}(t))_{t\in[0,1]}$ denotes the (completed) filtration generated by $B(\cdot)$, and $\mathbb{P}$ denotes Wiener measure.

    Assume for the sake of contradiction that $t^*>0$. Since $Y(\cdot)$ is a martingale in the Brownian filtration $\mathbb{F}$, the norm $\|Y(t)\|_2$ is continuous in $t$, and we have $\|Y(t^*)\|_2 = 0$. We will show using symmetry arguments that this implies $X(t^*)=0$ a.s., which contradicts the fact that $\|X(t)\|_p>0$ for all $t>0$.

    Define a map $\rho:C[0,1]\to C[0,1]$ that takes $\omega$ to the reflected path $\rho(\omega)$ given by
    \begin{equation*}
        \rho(\omega)(t) =
        \begin{cases}
            \omega(t) &\text{if } t\leq t^*, \\
            2 \omega(t^*) - \omega(t) & \text{if } t > t^*.
        \end{cases}
    \end{equation*}
    It is easy to see that the process $(B\circ\rho)(\omega,t)=\rho(\omega)(t)$ is a Brownian motion under $\mathbb{P}$, and that it generates the same filtration $\mathbb{F}$. We first show that $Y\circ\rho = -Y$ almost surely. Let $\widetilde{Y}=\frac{1}{2}(Y-Y\circ\rho)$. Since $\rho_{\#} \bP = \bP$, $(B\circ\rho) (1) = 2B(t^*) - B(1)$, and $Y(t^*)=0$, we have
    \begin{equation*}
        \E[(Y\circ\rho)B(1)] = \E[Y(B\circ \rho)(1)] = \E[Y(2B(t^*) - B(1))] = - \E[YB(1)],
    \end{equation*}
    which implies $\E[\widetilde{Y} B(1)] = \E[Y B(1)]$. One can check that $Y(t)\circ\rho = \E[Y\circ \rho\,|\,\mathcal{F}(t)]$, and thus
    \begin{equation*}
        \big\|\E[\widetilde{Y}\,|\,\mathcal{F}(t)]\big\|_2 \leq \frac{1}{2} \|Y(t)\|_2 + \frac{1}{2} \big\|\E[Y\circ\rho\,|\,\mathcal{F}(t)]\big\|_2 = \frac{1}{2} \|Y(t)\|_2 + \frac{1}{2} \|Y(t)\circ\rho\|_2 = \|Y(t)\|_2.
    \end{equation*}
    Similarly, $\|\widetilde Y\|_q \leq \| Y\|_q = 1$. It follows that $\widetilde{Y}$ is admissible for $P_{q,\delta}$ and improves over $Y$ on the objective. Thus, $\widetilde{Y}$ is optimal. By the first identity of \eqref{eq:primal-optimizers-finite-p:duality-identities}, the optimizer $Y$ of $P_{q,\delta}$ is unique, which implies $Y=\widetilde{Y}$ and in turn $Y\circ\rho = -Y$. By the oddness of $x\mapsto x|x|^{p-2}$, this identity also gives $X(1)\circ\rho = -X(1)$. In addition, $Y(t)\circ\rho = \E[Y\circ \rho\,|\,\mathcal{F}(t)] = -Y(t)$, and for $t>t^*$,
    \begin{equation*}
        \alpha(t)\circ\rho = -\lambda(t) Y(t)\circ\rho = -\alpha(t).
    \end{equation*}
    Noting that $X(t^*)\circ\rho=X(t^*)$, we get
    \begin{equation*}
        -X(1) = X(1)\circ\rho = X(t^*) - \int_{t^*}^1 \alpha(t)\,dt - (B(1)-B(t^*)) = 2X(t^*) - X(1),
    \end{equation*}
    yielding $X(t^*) = 0$ a.s. We obtain the desired contradiction, and thus $t^*=0$.
\end{proof}

\subsection{Properties of regularized optimizers}
We next study the properties of optimizers for $p<\infty$ and $\delta>0$ in more detail.

\begin{lemma}\label{lemma:optimal-control-structure:finite-p}
    Let $p\in[2,\infty)$ and $\delta>0$. Let $\alpha\in\mathbb{A}$ be the unique optimizer for $V_{p,\delta}^{\infty}$. Then, there exists $\hat{\alpha} : [0,1]\times\R \to \R$ such that
    \begin{enumerate}
        \item $\alpha(t)=\hat{\alpha}(t,X^\alpha(t))$ for all $t\in[0,1]$ almost surely,
        \item $\hat{\alpha}$ points inwards, i.e., $x\hat{\alpha}(t,x) \le 0$ for all $(t,x)\in[0,1]\times\R$,
        \item $\hat{\alpha}$ is odd, i.e., $\alpha(t,x)=-\alpha(t,-x)$ for all $(t,x)\in[0,1]\times\R$.
    \end{enumerate}
\end{lemma}

The rest of this section is devoted to the proof of this lemma. As such, we fix for the remainder of the section $p \in [2,\infty)$ and $\delta > 0$, and we let $\alpha\in\mathbb{A}$ and $Y \in L^q(\Omega)$ be the unique optimizers for $V_{p,\delta}^{\infty}$ and $P_{q,\delta}$, respectively. Define $Y(t)\coloneqq\E[Y\,|\,\F(t)]$ and $\lambda(t)$ as in Proposition~\ref{prop:primal-optimizers-finite-p}. The key step is to prove that $\alpha$ is the unique optimal control for an auxiliary stochastic control problem that we can analyze using a dynamic programming equation. We assume $\delta>0$ so that $t\mapsto \lambda(t)$ is bounded on $[0,1)$, which is needed for the well-posedness of the auxiliary control problem and the corresponding Hamilton--Jacobi--Bellman equation.

To this end, recall from Proposition~\ref{prop:primal-optimizers-finite-p} that $\alpha(t)=-\lambda(t) Y(t)$ for a.e.\ $t \in (0,1)$, and so we can also assert that
\begin{equation*}
    \alpha(t) = \argmin_{a \in \R}\Big(aY(t) + \frac{1}{2\lambda(t)}a^2\Big).
\end{equation*}
Fixing $p \in [2,\infty)$ and $\delta > 0$, we use the notation
\begin{align} \label{def.gammalambda}
    \gamma \coloneqq 1/(p\|X^\alpha(1)\|_{p}^{p-1}).
\end{align}
Our aim is to show that $\alpha$ is also  optimal for the auxiliary control problem
\begin{equation} \label{eq:auxiliary-control-problem:def}
\widetilde{V} \coloneqq \inf_\beta \E\bigg[ \gamma \big|X^\beta(1)\big|^p + \int_0^1 \frac{1}{2\lambda(t)}\beta(t)^2\,dt\bigg],
\end{equation}
where the infimum is taken over all progressively measurable processes $\beta$ such that
\begin{equation}\label{eq:auxiliary-control-problem:integrability}
    \E\bigg[ \int_0^1 \frac{1}{\lambda(t)} \beta(t)^2\,dt  \bigg] < \infty.
\end{equation}
We stress here that $\lambda$ and the value $\gamma\in(0,\infty)$ are kept fixed. Note also that because $\delta>0$ and $\lambda(t)\leq 1/\delta$ for all $t$, this integrability condition is stronger than square integrability.  The proof that $\beta=\alpha$ is optimal for this control problem could follow from the sufficient condition for optimality coming from Pontryagin's principle, though the textbook references typically only prove this for terminal cost functions of quadratic growth. Here is a direct proof using (weak) duality.
\begin{lemma}\label{lemma:auxiliary-control-problem:optimality}
Fix $p \in [2,\infty)$ and $\delta > 0$, and let $\alpha$ be the unique optimizer for the problem $V_{p,\delta}^{\infty}$. Then, $\alpha$ is also the unique optimizer for the problem $\widetilde{V}$, with $\lambda$ defined as in \eqref{eq:primal-optimizers-finite-p:def-lambda}, and $\gamma$ given by \eqref{def.gammalambda}.
\end{lemma}
\begin{proof}
Uniqueness follows from strict convexity of the cost as a functional of $\beta$.
Note that the convex conjugate of $x \mapsto \gamma |x|^p$ is $x \mapsto |x|^q/(q(\gamma p)^{q-1})$, and we have the Fenchel inequality
\[
\gamma |x|^p + \frac{1}{q(\gamma p)^{q-1}}|z|^q \ge xz,
\]
with equality precisely when $z=\gamma p x^{p-1}$.
Hence, for $\beta$ with \eqref{eq:auxiliary-control-problem:integrability} and any random variable $Z \in L^2(\Omega)$, we have
\begin{align*}
\E\big[\gamma \big|X^\beta(1)\big|^p\big] &= \E\bigg[ \gamma\bigg|B(1) + \int_0^1 \beta(t)\,dt\bigg|^p\bigg] \\
	&\ge \E\bigg[ Z B(1) + \int_0^1 Z\beta(t)\,dt - \frac{1}{q(\gamma p)^{q-1}}|Z|^q\bigg],
\end{align*}
and also
\begin{align*}
\E   \int_0^1 \frac{1}{2\lambda(t)}\beta(t)^2\,dt  &\ge  \E\int_0^1 \Big(-\E[Z\,|\,\F(t)]\beta(t) - \frac{\lambda(t)}{2}\E[Z\,|\,\F(t)]^2\Big)\,dt .
\end{align*}
Combining these inequalities, we deduce the weak duality
\begin{align*}
\widetilde{V} \ge \sup_{Z \in L^2(\Omega)}\E\bigg[Z B(1)  - \frac{1}{q(\gamma p)^{q-1}}Z^q - \int_0^1\frac{\lambda(t)}{2}\E[Z\,|\,\F(t)]^2 \, dt\bigg].
\end{align*}
There is equality precisely when
\begin{align*}
\beta(t) = -\lambda(t)\E[Z\,|\,\F(t)], \qquad Z = \gamma p\big|X^\beta(1)\big|^{p-1}.
\end{align*}
To complete the proof, we show that these identities hold with the choice $\beta=\alpha$ and $Z=Y$. The first is clear, recalling that $\alpha(t)=-\lambda(t) Y(t)=-\lambda(t)\E[Y|\F(t)]$. The definition $\gamma = 1/(p\|X^\alpha(1)\|_p^{p-1})$ and the identity $Y=\sgn(X^\alpha(1))|X^\alpha(1)|^{p-1}/\|X^\alpha(1)\|_p^{p-1}$ show that the second holds as well.
\end{proof}

We now analyze the auxiliary control problem \eqref{eq:auxiliary-control-problem:def} using the corresponding HJB equation
\begin{align} \label{eq:auxiliary-control-problem:hjb}
\partial_tu(t,x) - \frac{\lambda(t)}{2}|\partial_x u(t,x)|^2 +\frac12 \partial_{xx}u(t,x) = 0, \quad u(1,x) = \gamma |x|^p.
\end{align}
The following lemma verifies that the value function of \eqref{eq:auxiliary-control-problem:def} is indeed a classical solution of \eqref{eq:auxiliary-control-problem:hjb}, and that we have an appropriate verification result. We obtain Lemma~\ref{lemma:optimal-control-structure:finite-p} as an immediate corollary.

\begin{lemma} \label{eq:auxiliary-control-problem:verification}
    Fix $p \in [2,\infty)$ and $\delta > 0$, and let $\lambda$ and $\gamma$ be as defined in \eqref{eq:primal-optimizers-finite-p:def-lambda} and \eqref{def.gammalambda}. Define a function $u : [0,1] \times \R \to \R$ via
    \begin{align*}
    u(t_0,x_0) = \inf_{\beta} \E\bigg[ \int_{t_0}^1 \frac{1}{2 \lambda(t)} \beta(t)^2 \, dt + \gamma |X(1)|^p \bigg],
\end{align*}
where the infimum is taken over all progressively measurable processes $\beta$ satisfying
\begin{align*}
    \E\bigg[ \int_{t_0}^1 \frac{1}{\lambda(t)} \beta(t)^2 \, dt  \bigg] < \infty,
\end{align*}
and
\begin{align*}
    X(t) = x_0 + \int_{t_0}^t \beta(s) \, ds + (B(t)  - B(t_0)).
\end{align*}
Then, the following assertions hold:
\begin{enumerate}
    \item $u$ is convex and even in $x$, and there is a constant $C$ such that
    \begin{align*}
        |u(t,x)| \leq C(1 + |x|^p)
    \end{align*}
    for all $(t,x) \in [0,1] \times \R$.
    \item $u$ is a classical solution to \eqref{eq:auxiliary-control-problem:hjb}.
    \item the unique optimizer for the problem \eqref{eq:auxiliary-control-problem:def}  (and hence for the problem $V_{p,\delta}^{\infty}$, thanks to Lemma \ref{lemma:auxiliary-control-problem:optimality}) is give in feedback form by
    \begin{align*}
        \beta(t,x) = - \lambda(t) \partial_x u(t,x),
    \end{align*}
    and in particular is odd and points inwards.
\end{enumerate}
\end{lemma}

Before we prove this verification result, let us discuss how it implies Lemma~\ref{lemma:optimal-control-structure:finite-p}. By the definition of $u$, we have $u(0,0) = \widetilde{V}$, and $\beta(t) = -\lambda(t) \partial_x u(t,X^\beta(t))$ is optimal for the problem defining $u(0,0)$. Since $\alpha$ is the unique optimizer for $\widetilde{V}$ by Lemma~\ref{lemma:auxiliary-control-problem:optimality}, we must have $\alpha(t) = -\lambda(t) \partial_x u(t,X^\alpha(t))$ for a.e.\ $t \in (0,1)$, and the properties of $\hat{\alpha}(t,x) = -\lambda(t)\partial_x u(t,x)$ follow from the properties of $u$ stated in the lemma.

\begin{proof}[Proof of Lemma~\ref{eq:auxiliary-control-problem:verification}]
    This result would be standard, except for the fact that the terminal cost $x \mapsto \gamma |x|^p$ has some growth at infinity, and the fact that $\lambda(t) \to 0$ as $t \to 0$. Thus, we only sketch the proof and explain how to circumvent these issues.

    We begin by fixing $(\lambda^{\eps})_{\eps > 0}$ and $(g^{\eps})_{\eps > 0}$ such that $\lambda^{\eps}$ is smooth, $\lambda^{\eps}(t) \geq \lambda(t) \vee \eps$ for each $t \in [0,1]$, and $\lambda^{\eps} \to \lambda$ uniformly as $\eps \to 0$, $g^{\eps}$ is $C^2$ (with second derivatives bounded, locally uniformly in $\eps$), convex and Lipschitz, and satisfies
    \begin{align*}
        0 \leq g^{\eps}(x) \leq 1 + \mu |x|^p, \quad g^{\eps}(x) = \mu|x|^p \text{ on } [-\eps^{-1},\eps^{-1}].
    \end{align*}
    Now define $u^{\eps}$ exactly like $u$, but with $\lambda^{\eps}$ replacing $\lambda$, and $g^{\eps}$ replacing $\mu |x|^p$.
   Then, it is standard to show that $u^{\eps}$ is smooth, and is the unique viscosity solution of linear growth to the equation
    \begin{align*}
        - \partial_t u^{\eps} - u^{\eps}_{xx} + \frac{\lambda^{\eps}(t)}{2} |u^{\eps}_x|^2 = 0, \quad u^{\eps}(1,x) = g^{\eps}(x).
    \end{align*}
    Moreover, the convexity of the cost implies that $u^{\eps}(t,\cdot)$ is convex for each fixed $t$, see, e.g., Lemma 10.6 in \cite{flemingsoner}. In addition, testing the control $\alpha = 0$ from each initial condition $(t,x)$ shows that
    \begin{align} \label{eq:auxiliary-control-problem:verification:proof-uniformupperbound}
      0 \leq  u^{\eps}(t,x) \leq \E\big[g^{\eps}\big(x + B(1) - B(t)\big)\big] \leq 1 + \mu \E[|x + B(1) - B(t)|^p] \leq C(1 + |x|^p).
    \end{align}
   Using the well-known fact that the Lipschitz constant of a convex function on a ball is controlled by its oscillation over a larger ball, the convexity of $u^{\eps}(t,\cdot)$ and \eqref{eq:auxiliary-control-problem:verification:proof-uniformupperbound} together imply the estimate
    \begin{align*}
        |u_x^{\eps}(t,x)| \leq 2\sup_{|y| \leq |x| + 1} |u^{\eps}(t,y)| \leq C(1 + |x|^p),
    \end{align*}
    with $C$ independent of $\eps$. Now, we fix a smooth cut-off function $\kappa : \R \to \R$ satisfying $\kappa(x) = 1$ for $|x| \leq R$ and $\kappa(x) = 0$ for $|x| \geq 2R$. For $R > 0$, we set $\kappa^R(x) = \kappa(x/R)$, and then we consider the function
    \begin{align*}
        u^{\eps,R}(t,x) = \kappa^R(x) u^{\eps}(t,x).
    \end{align*}
    Then by explicit computation,
    \begin{align*}
        \partial_t u^{\eps,R} + \partial_{xx} u^{\eps,R} &= \kappa^R \big(\partial_t u^{\eps} + \partial_{xx} u^{\eps} \big) + 2 \partial_x u^{\eps} \partial_x \kappa^R + u^{\eps} \partial_{xx} \kappa^R
        \\
        &= \kappa^R \frac{\lambda^{\eps}(t)}{2} | \partial_x u^{\eps}|^2 + 2 \partial_x u^{\eps} \partial_x \kappa^R + u^{\eps} \partial_{xx} \kappa^R \eqqcolon f^{\eps,R}.
    \end{align*}
    The estimates on $u^{\eps}$ and $u^{\eps,R}$ above and the fact that $\kappa^R = 0$ for $|x| > 2R$ imply that for each $R > 0$,
    \begin{align*}
        \| f^{\eps, R}\|_{\infty} \leq C_R,
    \end{align*}
    with $C_R$ depending on $R$ but not on $\eps$. Thus, $u^{\eps,R}$ solves the heat equation on $[0,1] \times \R$ with a right-hand side which is bounded in $L^{\infty}$ by $C_R$ and with terminal condition $g^{\eps,R}(x) = \kappa^R(x) g^{\eps}(x)$ which is smooth, uniformly in $\eps$ for each fixed $R$. We deduce from classical parabolic estimates for the heat equation (see, e.g., Chapter 4, Section 1 of \cite{lady})
    \begin{align*}
        \big\| u^{\eps,R} \big\|_{W^{2,p}_{t,x}} \leq C \Big( \|g^{\eps,R}\|_{W^{2,p}_x} + \| f^{\eps,R}\|_{L^p_{t,x}} \Big) \leq C_{R,p},
    \end{align*}
    where $W^{2,p}_x$ and $W^{2,p}_{t,x}$ denote the usual elliptic and parabolic Sobolev spaces of functions with two derivatives in $L^p$, and $C_{R,p}$ can depend on $R$ and $p$ but not on $\eps$. Since $u^{\eps,R} = u^{\eps}$ for $|x| \leq R$, this implies that for each $R > 0$, $1 < p < \infty$,
    \begin{align} \label{firstw2pest}
        \big\| u^{\eps} \big\|_{W^{2,p}_{t,x}([0,1] \times [-R,R] )} \leq C_{R,p}.
    \end{align}
    Differentiating the equation for $u^{\eps,R}$, we find that $v^{\eps,R} = \partial_x u^{\eps,R}$ satisfies
    \begin{align*}
       &\partial_t v^{\eps,R} + \partial_{xx} v^{\eps,R} = \partial_x f^{\eps,R}
       \\
       &\quad = \kappa^R \lambda^{\eps} \partial_x u^{\eps} + 2 \partial_{xx} u^{\eps} \partial_x \kappa^R + 3 \partial_x u^{\eps} \partial_{xx} \kappa^R + u^{\eps} \partial_{xxx} \kappa^R.
    \end{align*}
    We know that $u_x^{\eps}$ is locally bounded and $u_{xx}^{\eps}$ is locally in $L^p_{t,x}$ from \eqref{firstw2pest}, we easily deduce that $\| \partial_x f^{\eps, R}\|_{L^p_{t,x}} \leq C_{p,R}$, and so $W^{2,p}$ estimates for the heat equation again give
    \begin{align*}
        \big\| u_x^{\eps,R} \big\|_{W^{2,p}_{t,x}} = \big\| v^{\eps,R} \big\|_{W^{2,p}_{t,x}} \leq C \Big( \|\partial_x f^{\eps,R}\|_{L^p_{t,x}} + \| \partial_x g^{\eps,R}\|_{W_x^{2,p}} \Big) \leq C_{R,p}.
    \end{align*}
    But now by parabolic Sobolev embedding, we deduce that for each $R > 0$ and $\alpha \in (0,1)$, $u_x^{\eps} \in C_{t,x,\text{loc}}^{(1 + \alpha)/2, 1 + \alpha} ([0,1] \times [-R,R])$, uniformly in $\eps$, i.e., for each $\alpha \in (0,1)$, $R > 0$,
    \begin{align*}
       u^{\eps}, \quad  \partial_x u^{\eps}, \quad \partial_{xx} u^{\eps} \in C_{t,x}^{\alpha/2, \alpha} ([0,1] \times [-R,R]),
    \end{align*}
    with bounds which are uniform in $\eps$. The same must be true for $\partial_t u^{\eps}$, thanks to the equation satisfied by $u^{\eps}$. Sending $\eps \to 0$, one can check (from the control formulation) that $u^{\eps} \to u$ pointwise, and thus by compactness $u, \partial_t u, \partial_x u$, and $\partial_{xx} u$ all inherit the local H\"older bounds from $u^{\eps}$, $\partial_t u^{\eps}$, $\partial_x u^{\eps}$, and $\partial_{xx} u^{\eps}$, and the PDE
    \begin{align*}
        - \partial_t u - u_{xx} + \frac{\lambda}{2} |u_x|^2 = 0, \quad u(1,x) = \mu |x|^p
    \end{align*}
    is satisfied in a classical sense. This proves that (2) holds. Moreover, sending $\eps \to 0$ in \eqref{eq:auxiliary-control-problem:verification:proof-uniformupperbound}, we complete the proof of (1).

    It remains to prove the verification result (3). We first consider the state equation
    \begin{align*}
        dX(t) = - \frac{\lambda(t)}{2} \partial_x u(t,X(t)) \, dt + dB(t), \quad X(0) = 0.
        \end{align*}
    Since $\beta(t,x) = - \lambda(t) \partial_xu(t,x)$ is locally Lipschitz in $x$ and satisfies
    \begin{align*}
      (x- y) \big(  \beta(t,x) - \beta(t,y) \big) = - \lambda(t) (x-y) \big(\partial_x u(t,x) - \partial_x u(t,y) \big) \leq 0
    \end{align*}
    by convexity of $u(t,\cdot)$, this equation has a unique strong solution which satisfies
    \begin{align*}
        \E\Big[ \sup_{0 \leq t \leq 1} |X(t)|^r \Big]  < \infty
    \end{align*}
    for each $r < \infty$. Now, It\^o's formula shows that
    \begin{align*}
        du(t,X(t)) &= \Big( \partial_t u + \partial_{xx} u - \lambda |u_x|^2 \Big)(t,X(t)) \, dt + \partial_x u(t,X(t)) \, dB(t)
        \\
        &= - \frac{\lambda(t)}{2} |u_x(t,X(t))|^2 \, dt + \partial_x u(t,X(t)) \, dB(t)
        \\
        &= - \frac{1}{2 \lambda(t)} |\beta(t,X(t))|^2 \, dt + \partial_x u(t,X(t)) \, dB(t).
    \end{align*}
    Since $X$ has moments of all orders and $\partial_x u(t,X(t)) \leq C(1 + |X(t)|^p)$, we deduce that
    \begin{align*}
        \E\bigg[ \int_0^1 |\partial_x u(t,X(t))|^2 \, dt \bigg] < \infty,
    \end{align*}
    and so $t\mapsto\int_0^t \partial_x u(s,X(s))\, dB(s)$ is a true martingale. It follows that
    \begin{align*}
        u(0,0) = \E\bigg[ \gamma |X(1)|^p + \int_0^1 \frac{1}{2\lambda(t)} |\beta(t,X(t))|^2\,dt \bigg],
    \end{align*}
    and so $\beta$ is optimal.

    Uniqueness, meanwhile, follows from the strict convexity of the problem. Finally, we note that $u_x(t,\cdot)$ is odd because $u(t,\cdot)$ is even, and $\alpha(t,\cdot)$ points inwards because $u(t,\cdot)$ is even and convex, hence attains its minimum at $0$. Thus $\partial_x u(t,\cdot)$ is an increasing function with $\partial_x u(t,0) = 0$.
\end{proof}

\subsection{Properties of optimizers}

Let us now return to the original problem $V^\infty$, written in the formulation of Section \ref{sec:prelim} which we repeat here:
\begin{equation}\label{eq:v-infty:pde-formulation}
    V^\infty = \inf_{(m,\alpha) \in \cA(0,\delta_0)} [m_1]_{\infty},
\end{equation}
where $\cA(0,\delta_0)$ is the set of all pairs $(m, \alpha)$ consisting of a (continuous) curve $[0,1] \ni t \mapsto m_t \in \cP_2(\R)$ and a measurable function $\alpha : [0,1] \times \R \to \R$, satisfying the equation
\begin{equation}\label{eq:v-infty:pde-formulation:fpe}
    \partial_t m = \frac{1}{2}\partial_{xx} m - \partial_x ( m \alpha ), \quad (t,x) \in (0,1) \times \R, \quad m_0 = \delta_0
\end{equation}
in the sense of distributions, as well as the constraint
\begin{equation}\label{eq:v-infty:pde-formulation:constraint}
    \int_{\R} \alpha(t,x)^2\, m_t(dx) \leq \frac{2}{\pi}, \quad \text{for a.e.\ } 0 \leq t \leq 1.
\end{equation}

We next show that the problem \eqref{eq:v-infty:pde-formulation} admits an optimizer, and that this optimizer satisfies some nice properties. We will obtain these properties by passing to the limit from the optimizers of the approximating problems $V_{p,\delta}^{\infty}$, with $p\uparrow \infty$ and $\delta\downarrow 0$. We will state the result in slightly more generality, allowing for any time horizon $T\in(0,1]$ in place of $1$. This will be needed for carrying out the dynamic programming argument in the proof of the upper bound.

\begin{proposition} \label{lem.optproperties}
    For any $0 < T \leq 1$, the problem
    \begin{equation}\label{eq:v-infty:pde-formulation:general-T}
        \inf_{(m,\alpha) \in \cA(0,\delta_0)} [m_T]_{\infty},
    \end{equation}
    admits an optimizer $(m,\alpha)$ with the following properties:
    \begin{enumerate}
        \item $m_t$ is even, in the sense that $(x \mapsto -x)_{\#} m_t = m_t$, for each $0 \leq t \leq T$, \label{lem.optproperties.even}
        \item $\alpha(t,\cdot)$ is odd, for each fixed $t$, \label{lem.optproperties.odddrift}
        \item $\alpha$ ``points inwards", in the sense that $\alpha(t,x)x \leq 0$ for each $(t,x) \in [0,T] \times \R$. \label{lem.optproperties.inward}
    \end{enumerate}
\end{proposition}

\begin{proof}
    Let us assume $T=1$ for notational simplicity. Fix $\delta_k>0$ converging monotonically toward $0$ and $p_k\in[2,\infty)$ increasing to $\infty$. Let $\alpha^k$ be the optimizer for $V_{p_k,\delta_k}^\infty$ on the time interval $[0,1]$, with $X^k$ its associated state process. We may assume that these are defined on a common probability space, with a common Brownian motion $B$. We know from Lemma~\ref{eq:auxiliary-control-problem:verification} that $\alpha^k(t)=\hat{\alpha}^k(t,X^k(t))$ is ``odd'' and ``pointing inward'' in the sense that
    \begin{equation*}
        \hat{\alpha}^k(t,-x)=-\hat{\alpha}^k(t,x), \quad \hat{\alpha}^k(t,x)x \le 0,\qquad\text{for all } (t,x) \in [0,1] \times \R.
    \end{equation*}
    We next show  how to pass this property to the limit. It could be done at the PDE level, but we will stick here with a probabilistic viewpoint of Lemma \ref{lem.equivstrong}, similar to the compactness arguments used in Section~\ref{sec:lower-bound} to prove the lower bound.

    Roughly speaking, we first show that any subsequential limit of $\alpha^k$ must land on an optimizer of our true control problem \eqref{eq:v-infty:pde-formulation:general-T}, and that this optimizer satisfies the required properties. Let $\Lambda^k$ denote the random probability measure $dt\,\delta_{\alpha^k(t)}(da)$ on $[0,1] \times \R$. The constraint $\alpha^k \in \mathbb{A}$ automatically implies
    \begin{equation}
    \sup_k \int_{[0,1]\times \R} a^2\,\Lambda^k(dt,da) = \sup_k\E\int_0^1(\alpha^k(t))^2\,dt \le \frac{2}{\pi} < \infty. \label{approxcontrol-momentbound}
    \end{equation}
    It follows that $\mathrm{Law}(\Lambda^k)$ is a tight sequence in $\P(\P_1([0,1] \times \R))$, and also that $\mathrm{Law}(X^k)$ is a tight sequence in $\P(C[0,1])$.

    Suppose now that $(X,\Lambda)$ is a random element of $C[0,1] \times \P_1([0,1] \times \R)$ whose law is a limit point of that of $(X^k,\Lambda^k)$ along some subsequence. By extending our probability space if necessary, we can assume that $(X,\Lambda)$ is also defined on $(\Omega,\mathcal{F},\mathbb{P})$. Now, define $\alpha(t)=\int_\R a\, \Lambda_t(da)$, where we have disintegrated $\Lambda(dt,da)=dt\,\Lambda_t(da)$, noting that the first marginal of $\Lambda$ is a.s.\ uniform. (A progressively measurable version of $\alpha$ can be more carefully constructed as the limit of $h^{-1}\int_{[(t-h)_+,t] \times \R}a\,\Lambda(ds,da)$ as $h\downarrow 0$  through the rationals, which exists a.s.) It is straightforward to check that the process
    \begin{equation*}
        X(t) - \int_0^t \alpha(s)\,ds = X(t) - \int_{[0,t] \times \R}a\,\Lambda(ds,da)
    \end{equation*}
    must be a Brownian motion with respect to the filtration $\FF^{X,\Lambda}$ generated by $(X,\Lambda)$, because the same is true of $(X^k,\Lambda^k)$ for each $k$.

    Let us define $m_t \coloneqq \cL(X(t))$ for $t \in [0,1]$.
    As in Lemma~\ref{lem.equivstrong}, we can find a jointly measurable function $\hat{\alpha}:[0,1]\times \R\to\R$ such that
    \begin{equation}\label{eq.lem.approximation.def.alphahat}
        \hat{\alpha}(t,X(t)) = \E[\alpha(t)\,|\,X(t)],\qquad\text{a.s.\ for a.e.\ } t.
    \end{equation}
    We now show that the pair $(m,\hat{\alpha})$ is admissible for the problem \eqref{eq:v-infty:pde-formulation:general-T}, and that it is optimal.

    The first part of admissibility, namely that $(m,\hat{\alpha})$ is a weak solution to the Fokker--Planck equation \eqref{eq:v-infty:pde-formulation:fpe} is shown in exactly the same way as in Lemma~\ref{lem.equivstrong}. To see that the constraint \eqref{eq:v-infty:pde-formulation:constraint} is satisfied, let $\phi\in C[0,1]$ be non-negative. Then, by \eqref{eq.lem.approximation.def.alphahat} and Jensen's inequality
    \begin{align*}
        \int_0^1 \phi(t)\int_\R \hat{\alpha}(t,x)^2 \, m_t(dx)\,dt
        &\leq \int_0^1 \phi(t) \E\big[\alpha(t)^2\big]\,dt \\
        &\leq \int_0^1 \phi(t) \E\bigg[\int_\R a^2\,\Lambda_t(da)\bigg]\,dt \\
        &= \E \int_{[0,1]\times\R} \phi(t) a^2 \,\Lambda(dt, da) \\
        &\leq \liminf_k  \E\int_{[0,1] \times \R} \phi(t)\,a^2\,\Lambda^k(dt,da) \\
        &\leq \frac{2}{\pi}\int_0^1\phi(t)\,dt,
    \end{align*}
    where the last step used \eqref{approxcontrol-momentbound}.
    As $\phi$ was arbitrary, it follows that \eqref{eq:v-infty:pde-formulation:constraint} holds for $(m,\hat{\alpha})$.

    Let us now prove optimality. By Lemma~\ref{lem.equivstrong}, the value of the problem \eqref{eq:v-infty:pde-formulation:general-T} matches the value of the corresponding ``strong'' formulation defined in terms of $(\Omega,\mathcal{F},\mathbb{P})$ and $B$. In particular, $(m,\hat{\alpha})$ is optimal if for all $\beta\in\mathbb{A}$, we have $\|X^\beta(1)\|_{\infty} \geq [m_1]_\infty$. So let $\beta\in\mathbb{A}$ be arbitrary, and assume $\|X^\beta(1)\|_{\infty}<\infty$. We have
    \begin{align*}
        \|X^\beta(1)\|_{\infty} + \frac{\delta_k}{\pi}
        &\ge \|X^\beta(1)\|_{\infty} + \frac{\delta_k}{2}\E\int_0^1\beta(t)^2\,dt \\
        &\ge \|X^\beta(1)\|_{p_k}  + \frac{\delta_k}{2}\E\int_0^1\beta(t)^2\,dt \\
        &\ge \|X^k(1)\|_{p_k}  + \frac{\delta_k}{2}\E\int_0^1(\alpha^k(t))^2\,dt
    \end{align*}
    where the last step used the optimality of $\alpha^k$ for $V_{p_k,\delta_k}^\infty$. Since $\E[\alpha^k(t)^2]\le 2/\pi$, passing to the limit yields
    \begin{equation*}
        \|X^\beta(1)\|_{\infty} \ge \liminf_k \|X^k(1)\|_{p_k},
    \end{equation*}
    with the $\liminf$ taken along the same subsequence constructed above. For any $0 \le R < [m_1]_\infty=\|X(1)\|_{\infty}$, we have
    \begin{equation*}
    \big(\E[|X^k(1)|^{p_k}]\big)^{1/p_k} \ge  \big(\E[|X^k(1)|^{p_k}1_{\{|X^k(1)| > R\}}] \big)^{1/p_k} \ge R\, \PP\big[|X^k(1)| > R\big]^{1/p_k}.
    \end{equation*}
    Since $X^k(1)$ converges in law to $X(1)$, we have
    \begin{equation*}
        \liminf_k \PP\big[|X^k(1)| > R\big] \ge \PP\big[|X(1)| > R\big] > 0,
    \end{equation*}
    the last inequality coming from the definition of $\|X(1)\|_{\infty}$. Hence, sending $k\to\infty$ yields
    \begin{equation*}
        \|X^\beta(1)\|_{\infty} \ge \liminf_k \|X^k(1)\|_{p_k} \ge R.
    \end{equation*}
    As $\beta \in \mathbb{A}$ and $R < [m_1]_\infty$ were arbitrary, we deduce that
    \begin{equation*}
        V^\infty = \inf_{\beta\in\mathbb{A}}\|X^\beta(1)\|_{\infty} \ge [m_1]_\infty.
    \end{equation*}
    In particular, the pair $(m,\hat{\alpha})$ is optimal for \eqref{eq:v-infty:pde-formulation:general-T}.

    It remains to prove that $(m,\hat{\alpha})$ satisfies the desired properties \eqref{lem.optproperties.even}, \eqref{lem.optproperties.odddrift}, and \eqref{lem.optproperties.inward}. We may modify $\hat{\alpha}(t,x)$ on a null set of the measure $dt\,m_t(dx)$ without affecting the admissibility of $(m,\hat{\alpha})$, and so it suffices to show that
    \begin{equation*}
        \hat{\alpha}(t,-X(t)) = -\hat{\alpha}(t,X(t)), \quad \hat{\alpha}(t,X(t))X(t) \le 0,\qquad\text{a.s.\ for a.e.\ } t,
    \end{equation*}
    as well as $X(t) \stackrel{d}{=} -X(t)$ for each $t$.

    For the inward pointing property, recall that $\alpha^k(t)X^k(t) \le 0$ a.s.\ for a.e.\ $t$. Let us argue that similarly $\hat{\alpha}(t,X(t)) X(t) \le 0$ a.s.\ for a.e.\ $t$.
    Fix a Lipschitz function $h : [0,1] \times \R \to [0,\infty)$ of compact support. We have
    \begin{align*}
        \E\int_0^1 h(t, X^k(t)) \alpha^k(t) X^k(t)\,dt = \E\int_{[0,1] \times \R} h(t,X^k(t)) a X^k(t)\,\Lambda^k(dt,da),
    \end{align*}
    as well as the same identity with $(X,\alpha,\Lambda)$ in place of $(X^k,\alpha^k, \Lambda^k)$. Since $\R\ni x\mapsto h(t,x)x$ is bounded, the function
    \begin{equation*}
            C[0,1] \times \P_1([0,1] \times \R) \ni (x,q) \mapsto \int_{[0,1] \times \R} h(t,x(t)) a x(t)\,q(dt,da)
    \end{equation*}
    is jointly continuous. Using the weak convergence of $(X^k,\Lambda^k)$ to $(X,\Lambda)$ in $C[0,1] \times \P_1([0,1] \times \R)$ along with $\E \int_{[0,1]\times \R}a^2\,\Lambda(dt,da) \le 2/\pi$, it follows that
    \begin{equation*}
        \lim_{k\to\infty}\E\int_{[0,1] \times \R} h(t,X^k(t)) a X^k(t)\,\Lambda^k(dt,da) = \E\int_{[0,1] \times \R} h(t,X(t)) a X(t)\,\Lambda(dt,da).
    \end{equation*}
    As $h \ge 0$, the left-hand side is not positive, and hence nor is the right side. As $h \ge 0$ was arbitrary, we deduce that $\alpha(t)X(t) \le 0$ a.s.\ for a.e.\ $t$, and in turn
    \begin{align*}
        \hat{\alpha}(t,X(t)) X(t) = \E\big[ \alpha(t) \,|\, X(t) \big] X(t) = \E\big[ \alpha(t) X(t) \,|\, X(t) \big] \leq 0.
    \end{align*}

    To see that the drift is odd, recall that the approximating problem has an odd drift, i.e., $\hat{\alpha}^k(t,-x)=-\hat{\alpha}^k(t,x)$ for all $t$ and $x$. As $X(0)=0$, this implies that the corresponding state process $X^k$ is symmetric, in the sense that $X^k(t) \stackrel{d}{=} -X^k(t)$ for all $t$. Combining these two properties, for any bounded continuous $h$ we find
    \begin{equation}
    \E\int_0^1 h(t,-X^k(t))\hat{\alpha}^k(t,X^k(t))\,dt = -\E\int_0^1 h(t,X^k(t))\hat{\alpha}^k(t,X^k(t))\,dt. \label{pf:driftodd1}
    \end{equation}
    We take limits, justified as in the previous step:
    \begin{align*}
        \E\int_0^1 h(t,X^k(t))\hat{\alpha}^k(t,X^k(t))\,dt
        &= \E\int_{[0,1] \times \R} h(t,X^k(t))a\,\Lambda^k(dt,da) \\
        &\to \E\int_{[0,1] \times \R} h(t,X(t))a\,\Lambda(dt,da) \\
        &= \E\int_0^1  h(t,X(t))\alpha(t)\,dt \\
        &= \E\int_0^1 h(t,X(t))\hat{\alpha}(t,X(t))\,dt.
    \end{align*}
    Thus, passing to the limit on both sides of \eqref{pf:driftodd1} yields
    \begin{equation*}
       \E\int_0^1 h(t,-X(t))\hat{\alpha}(t,X(t))\,dt = -\E\int_0^1 h(t,X(t))\hat{\alpha}(t,X(t))\,dt.
    \end{equation*}
    The symmetry of the state process $X^k(t)$ obviously passes to the limit, and so $X(t) \stackrel{d}{=} - X(t)$. Apply this to the left-hand side to get
    \begin{equation*}
        \E\int_0^1 h(t,X(t))\hat{\alpha}(t,-X(t))\,dt = -\E\int_0^1 h(t,X(t))\hat{\alpha}(t,X(t))\,dt.
    \end{equation*}
    As $h$ was arbitrary, we finally deduce that
    \begin{equation*}
        \hat{\alpha}(t,-X(t)) = -\hat{\alpha}(t,X(t)), \quad \hat{\alpha}(t,X(t))X(t) \le 0,\qquad\text{a.s.\ for a.e.\ } t. \qedhere
    \end{equation*}
\end{proof}

\subsection{Proof of Lemma \ref{lem.Wpapproximation}}
    To start, we apply Proposition~\ref{lem.optproperties} to produce an optimizer $(m,\alpha)\in\cA(0,\delta_0)$ such that $\alpha$ is odd and inward-pointing, in the sense that
    \begin{equation*}
        \alpha(t,x) = -\alpha(t,-x),\qquad x\,\alpha(t,x) \leq 0,\qquad\text{for all } (t,x)\in[0,T]\times \R.
    \end{equation*}
    In addition, $m_t$ is even for every $t$. We proceed by first mollifying $(m,\alpha)$ to obtain regularized controls $(m^\delta,\alpha^\delta)\in\cA(0,\delta_0)$. We pass to the corresponding strong solutions of the state equation, and truncate the resulting open-loop controls to obtain the desired approximation.

    Let us first fix $\delta > 0$, and let $\gamma_\delta$ denote the Gaussian density with mean zero and variance $\delta$, and define
    \begin{equation}\label{eq.malpha.mollified}
        m_t^{\delta}(x) \coloneqq m_t * \gamma_{\delta}(x), \qquad \alpha^{\delta}(t,x)  \coloneqq \frac{ \big( m_t \alpha(t,\cdot) \big) * \gamma_{\delta}(x) }{m_t * \gamma_{\delta}(x)}
    \end{equation}
    for $(t,x)\in[0,T]\times\R$. Then, $m_t^\delta$ is a probability density on $\R$, and we abuse notation by using the same symbol $m_t^\delta$ for the corresponding measure. By a standard calculation, $(m^\delta, \alpha^\delta)$ satisfies the same Fokker--Planck equation as $(m,\alpha)$. Moreover, the $L^2$ constraint is preserved. Indeed, for $(t,x)\in[0,T]\times\R$
    \begin{equation}
        \label{eq.alpha.mollified.formula}
        \alpha^\delta(t,x) = \frac{\int_\R \alpha(t,y) \gamma_\delta(x-y)\,m_t(dy)}{\int_\R \gamma_\delta(x-y)\,m_t(dy)}.
    \end{equation}
    Thus, by Jensen's inequality,
    \begin{align*}
        \int_\R |\alpha^\delta(t,x)|^2 m_t^\delta(x)\,dx
        &\leq \int_\R \bigg(\frac{\int_\R |\alpha(t,y)|^2 \gamma_\delta(x-y)\,m_t(dy)}{\int_\R \gamma_\delta(x-y)\,m_t(dy)}\bigg)\,m_t^\delta(x)\,dx \\
        &= \int_\R |\alpha(t,y)|^2\,m_t(dy) \leq \cstar^2,
    \end{align*}
    where we used $m_t^\delta(x)=m_t*\gamma_\delta(x)$ to cancel the denominator and applied Fubini's theorem. Hence, $(m^\delta, \alpha^\delta)\in \cA(0,\delta_0)$.

    Next, because the second moment of $m_t$ is uniformly bounded in $t$, it is straightforward to check that the function $m_t^\delta(x) = m_t * \gamma_\delta(x)$ is locally bounded away from zero, uniformly in $t$. On the other hand,
    \begin{equation*}
        \big|\big( m_t \alpha(t,\cdot) \big) * \gamma_{\delta}(x)\big| \leq \|\gamma_\delta\|_{L^\infty(\R)} \int_\R |\alpha(t,y)| \,m_t(dy) \leq \cstar \|\gamma_\delta\|_{L^\infty(\R)}.
    \end{equation*}
    Thus, $\alpha^\delta(t,\cdot)$ is locally bounded, uniformly in $t$. Moreover, $\alpha^\delta(t,\cdot)$ is smooth, inherited from $\gamma_\delta$. Let us also show that $\alpha^\delta$ is odd and inward-pointing. The oddness comes directly from the oddness of $\alpha$, together with the evenness of $m_t$ and $\gamma_\delta$. For the inward-pointing property, it suffices then to show $\alpha^\delta(t,x) \leq 0$ for $x>0$. So let $x>0$, and observe that for every $y>0$,
    \begin{equation*}
        \alpha(t,y) \gamma_\delta(x-y) \leq -\alpha(t,-y) \gamma_\delta(x+y).
    \end{equation*}
    Combining this with the evenness of $m_t$, we get
    \begin{equation*}
        \int_0^\infty \alpha(t,y)\gamma_\delta(x-y)\,m_t(dy) \leq - \int_{-\infty}^0 \alpha(t,y)\gamma_\delta(x-y)\,m_t(dy).
    \end{equation*}
    In view of \eqref{eq.alpha.mollified.formula}, this yields $\alpha^\delta(t,x)\leq 0$.

    Let us now use the properties of $\alpha^\delta$ to obtain a unique strong solution to the state equation
    \begin{equation*}
        dX^{\delta}(t) = \alpha^{\delta}\big(t,X^\delta(t)\big)\, dt + dB(t), \qquad X^{\delta}(0) = X^{\delta}_0,
    \end{equation*}
    where $(B(t))_{t\in[0,T]}$ is a Brownian motion on some fixed probability space and $X_0^\delta\sim m_0^\delta=\gamma_\delta$. The local smoothness guarantees the existence and pathwise uniqueness of a solution up to some possible explosion time. The inward-pointing property rules out the possibility of blow-up. Indeed, define the stopping times
    \begin{equation*}
        \tau_M^\delta \coloneqq \inf\{t \geq 0 : |X^{\delta}(t)| \geq M\}\wedge T,\qquad M\in\N,
    \end{equation*}
    and notice that by It\^{o}'s formula and $m_0^\delta = \gamma_\delta$, for every $t\in[0,T]$
    \begin{equation}\label{eq.xdelta.secondmoment}
        \E\big(X^\delta(t\wedge \tau_M^\delta)\big)^2 =  \delta + \E \int_0^{t\wedge \tau_M^\delta} \big(2 X^\delta(s) \alpha^\delta\big(s, X^\delta(s)\big) + 1\big)\,ds \leq \delta + t.
    \end{equation}
    It follows that $M^2\PP[\tau_M^\delta \leq t] \leq \delta + t$ for every $M$ and $t$, and thus blow-up cannot occur.

    We now define, for each $M\in\N$, an open-loop control $\alpha^{\delta,M}$ by
    \begin{equation*}
        \alpha^{\delta,M}(t) \coloneqq 1_{\{t<\tau_M^\delta\}} \alpha^\delta\big(t, X^\delta(t)\big),\qquad 0\leq t \leq T,
    \end{equation*}
    and the corresponding state process
    \begin{equation*}
        X^{\delta, M}(t) \coloneqq X_0^\delta + \int_0^t \alpha^{\delta,M}(s)\,ds + B(t),\qquad 0\leq t \leq T.
    \end{equation*}
    Notice that $\alpha^{\delta,M}$ is uniformly bounded since $\alpha^\delta(t,\cdot)$ is locally bounded, uniformly in $t$. Moreover, $X^{\delta,M}$ and $X^\delta$ coincide before time $\tau_M^\delta$, and so the inward-pointing property is preserved in the sense that
    \begin{equation*}
        \alpha^{\delta, M}(t) X^{\delta, M}(t) \leq 0,\qquad \text{a.s.\ for a.e.\ $t$.}
    \end{equation*}
    As in \eqref{eq.xdelta.secondmoment}, one can show that this guarantees that for $t\in[0,T]$
    \begin{equation*}
        \E\big|X^{\delta,M}(t)\big|^{2k} \leq \E\big|B(\delta + t)\big|^{2k} = (\delta + t)^{k}\, \E|N(0,1)|^{2k},\qquad k\in\N,
    \end{equation*}
    by using It\^{o}'s formula and induction over $k$. In particular, $X^{\delta,M}(T)$ is bounded in $L^p$, uniformly in $M$. Since also $\PP[\tau_M^\delta < T] \to 0$ as $M\to\infty$, this implies
    \begin{equation}\label{xdeltam}
        \big\|X^{\delta,M}(T) - X^\delta(T)\|_p \longrightarrow 0 \qquad\text{as $M\to\infty$.}
    \end{equation}

    Next, we define a process $\widetilde{X}^{\delta,M}$ exactly like $X^{\delta,M}$, but started at zero:
    \begin{equation*}
        \widetilde{X}^{\delta,M}(t) \coloneqq \int_0^t \alpha^{\delta,M}(s)\,ds + B(t),\qquad 0\leq t \leq T.
    \end{equation*}
    In particular,
    \begin{align} \label{lpdelta}
        \big\| \wt{X}^{\delta, M}(T) - X^{\delta,M}(T)\big\|_p = \|X_0^{\delta}\|_p = \sqrt{\delta}\,\|N(0,1)\|_p.
    \end{align}
    As in the proof of Lemma~\ref{lem.equivstrong}, we can find a jointly measurable $\widetilde{\alpha}^{\delta,M}:[0,T]\times\R\to\R$ such that
    \begin{equation*}
        \wt{\alpha}^{\delta,M}\big(t,\wt{X}^{\delta,M}(t)\big) = \E\big[\alpha^{\delta,M}(t)\,\big|\,\wt{X}^{\delta,M}(t)\big],\qquad\text{a.s.\ for a.e.\ $t$},
    \end{equation*}
    and we set
    \begin{equation*}
        \wt{m}_t^{\delta,M} \coloneqq \cL\big(\wt{X}^{\delta,M}(t)\big),\qquad 0\leq t \leq T.
    \end{equation*}
    Then, $(\wt{m}^{\delta,M},\wt{\alpha}^{\delta,M})$ lies in $\cA(0,\delta_0)$, and we now argue that this gives the desired approximation.

    Fix $\eps>0$. Using $m_T^\delta = m_T * \gamma_\delta$ and \eqref{lpdelta}, we can first choose $\delta>0$ and then $M=M(\delta)$ to ensure that
    \begin{equation*}
        \sqrt{\delta}\,\|N(0,1)\|_p < \frac\eps3,\quad \bd_p\big(m_T^{\delta}, m_T\big) < \frac{\eps}{3},\quad \bd_p\Big(\cL\big(X^{\delta, M}(T)\big), m_T^{\delta}\Big) < \frac{\eps}{3}.
    \end{equation*}
    By \eqref{lpdelta} and the triangle inequality, we conclude
    \begin{equation*}
        \bd_p\big(\wt{m}_T^{\delta, M}, m_T\big) \leq \eps.
    \end{equation*}
    This completes the proof. \hfill \qedsymbol

\section{Bounds and asymptotics for the mean-field value}
\label{sec:bounds_and_asymptotics}

While at present we do not know how to solve the limiting optimization $V^{\infty}$ directly, it is possible to give some fairly explicit upper and lower bounds, which we present in this section. In fact, we will also give upper and lower bounds for the value $V_T^{\infty}$ of the problem with time horizon $T$ as defined in \eqref{def.vTinf}, which will allow us to obtain some asymptotics as $T \to \infty$ or $T \to 0$. As discussed already in the introduction, the proof of our main convergence result easily adapts to show that
\begin{align*}
    V_T^{\infty} \approx V^{n,m}, \quad \text{ for $n \gg 1$ and $\frac{m}{n} \approx T$},
\end{align*}
where $V^{n,m}$ is as in \eqref{def.vnm}. Thus, the asymptotics of $V_T^{\infty}$ give insight into the problem $V^{n,m}$ when either $m \gg n$ or $m \ll n$.

\subsection{First upper bound: A stationary solution}

We now prove the following upper bound which is uniform in $T$.

\begin{proposition} \label{prop.firstupperbound}
    For any $T > 0$, we have  $V_T^{\infty}  \le (\pi/2)^{3/2}$.
\end{proposition}

\begin{proof}
    For $T > 0$, we introduce the value function $\cV_T^{\infty} : [0,T] \times \cP_2(\R) \to \R$, defined exactly like $\cV^{\infty}$ but over the time interval $[0,T]$, i.e.,
    \begin{align*}
        \cV_T^{\infty}(t_0,m_0) \coloneqq \inf_{(m,\alpha)} [m_T]_{\infty},
    \end{align*}
    with the infimum taken over all pairs $(m,\alpha)$ such that
    \begin{gather*}
        \partial_t m = \frac{1}{2}\partial_{xx} m - \partial_x(m\alpha) \text{ in } (t_0,T) \times \R, \quad m_{t_0} = m_0, \\
        \quad \int_{\R} \alpha(t,x)^2 \, m_t(dx) \leq c_0^2 \quad \text{ for a.e. $t \in [t_0,T]$},
    \end{gather*}
   and we use $c_0 = \sqrt{2/\pi}$ in what follows for clarity. Consider the measure $m^{\text{st}} \in \cP(\R)$ and the function $\alpha^{\text{st}} : \R \to \R$ defined by
   \begin{align*}
       m^{\text{st}}(dx) = \frac{2c_0}{\pi} \cos^2(c_0 x) 1_{\{|x| \leq \pi/(2c_0)\}} \, dx, \quad \alpha^{\text{st}}(x) = \begin{cases}
            - c_0 \tan(c_0 x) & |x| \leq \pi/(2c_0),
            \\
            0 & \text{otherwise.}
       \end{cases}
   \end{align*}
   The value of $\alpha^{\text{st}}$ for $|x| \geq \pi/(2c_0)$ plays no role, and could be set to any arbitrary value. Now, explicit computation shows that $(m^{\text{st}}, \alpha^{\text{st}})$ is a stationary solution of the Fokker--Planck equation, i.e.,
   \begin{align*}
       \frac{1}{2} \partial_{xx} m^{\text{st}} - \partial_x \big( m^{\text{st}} \alpha^{\text{st}} \big) = 0
   \end{align*}
   in a distributional sense. Moreover, we have
   \begin{align*}
       \int_{\R} (\alpha^{\text{st}}(x))^2 \, m^{\text{st}}(dx) &= \frac{2c_0^3}{\pi} \int_{-\pi/(2c_0)}^{\pi/(2c_0)} \tan^2(c_0 x) \cos^2(c_0 x) \,dx
       \\
       &= \frac{2c_0^3}{\pi} \int_{-\pi/(2c_0)}^{\pi/(2c_0)} \sin^2(c_0x)\,dx
       \\
       &= \frac{2c_0^2}{\pi} \int_{-\pi/2}^{\pi/2} \sin^2(y) \,dy = c_0^2.
   \end{align*}
   As a consequence, for any $T > 0$, any $t_0 \in [0,T)$, the pair
   \begin{align*}
       m_t = m^{\text{st}}, \quad \alpha(t,x) = \alpha^{\text{st}}(x)
   \end{align*}
   is admissible for the problem defining $\cV_T^{\infty}(t_0,m^{\text{st}})$, and thus
   \begin{align*}
       \cV^{\infty}_T(t_0, m^{\text{st}}) \leq \big[m^{\text{st}}\big]_{\infty} = \frac{\pi}{2c_0} = (\pi/2)^{3/2}.
   \end{align*}
   Finally, we note that Lemmas \ref{lem.equivstrong} and \ref{lem.vinf.convex} adapt in the obvious way to general $T \neq 1$, and it follows that
   \[
       V_T^{\infty} = \cV_T^{\infty}(0,\delta_0) \leq \cV_T^{\infty}(0,m^{\text{st}}) \leq (\pi/2)^{3/2}. \qedhere
   \]
\end{proof}

\begin{remark} \label{re:conditionedBM}
It is a classic result of Knight \cite{knight1969brownian} that the solution of the SDE
\[
dX(t) = \alpha^{\text{st}}(X(t)) \, dt + dB(t), \quad X(0)=0
\]
has the same law as Brownian motion conditioned to never exit the interval $[-a,a]$ where $a=\pi/(2c_0)=(\pi/2)^{3/2}$. More precisely, for each $t > 0$ the law of $(X_s)_{0 \le s \le t}$ is equal to the weak limit as $T\to\infty$ of the conditional law of $(B_s)_{0 \le s \le t}$ given $\{\sup_{0 \le s \le T}|B_s| < a\}$.
\end{remark}

\subsection{Second upper bound: The F\"ollmer drift construction}

The  F\"ollmer drift construction from Lemma \ref{lem.follmer} leads to another upper bound, which is tighter for small $T$ than that of Proposition \ref{prop.firstupperbound}.
It is given in terms of an eigenvalue problem for the Ornstein--Uhlenbeck operator restricted to a bounded domain.
In the following, let $L_T$ denote the Ornstein--Uhlenbeck infinitesimal generator, defined for smooth $f$ by
\[
L_Tf(x) \coloneqq  f''(x) - T^{-1}xf'(x).
\]
For each $r > 0$, let $H_{r,T} \coloneqq H^1((-r,r);\gamma_T|_{[-r,r]})$ denote the closure, with respect to the norm $\|f\|_{H^1(\gamma_T)} \coloneqq \|f\|_{L^2(\gamma_T)} + \|f'\|_{L^2(\gamma_T)}$, of the set of smooth functions with compact support contained in $(-r,r)$. For smooth compactly supported $f$ we have the integration by parts identity
\[
\int_{\R}f(-L_Tf)\,d\gamma_T = \int_{\R} (f')^2\,d\gamma_T,
\]
and it is standard to check that the domain of the closure of $L_T$ is precisely $H_{r,T}$. Let $\lambda_T(r)$ denote the smallest number $\lambda \ge 0$ such that there exists $f\in H_{r,T}$ satisfying $-L_T f = \lambda f$. Equivalently, this leading eigenvalue $\lambda_T(r)$ admits the Rayleigh--Ritz representation
\begin{equation}
\lambda_T(r) = \inf\big\{ \|f'\|_{L^2(\gamma_T)}^2 : f \in H_{r,T}, \ \|f\|_{L^2(\gamma_T)}=1 \big\}, \label{def:EvalProblem}
\end{equation}
and the infimum is uniquely attained by a solution of $-L_T f=\lambda_T(r) f$.
Clearly $\lambda_T(\cdot)$ is nonincreasing, because $H_{r,T}$ embeds into $H_{r',T}$ for any $r < r'$.
The recent paper \cite{colesanti2024brunn} contains a detailed study of this Dirichlet eigenvalue problem of the Ornstein--Uhlenbeck operator restricted to a convex domain, more generally in $\R^n$. Interestingly, their Theorem 1.2 shows that the leading eigenvalue  is convex with respect to Minkowski addition of the underlying domain, in the spirit of the Brunn--Minkowski inequality; this implies that our function $\lambda_T(\cdot)$ is convex. The equality case shown in \cite[Theorem 1.2]{qin2025strong} implies that $\lambda_T(\cdot)$ is in fact \emph{strictly} convex, and thus it is strictly decreasing. Hence, the inverse function $\lambda_T^{-1}$ is well-defined (though we do not really need this and could instead work with the right-continuous inverse).

\begin{proposition} \label{prop.secondupperbound}
With $\lambda_T(\cdot)$ defined as above, we have $V_T^{\infty} \le \lambda_T^{-1}(1/ (2\pi) )$.
\end{proposition}
\begin{proof}
Equivalently, we must show that $V_T^{\infty}\le r$ for any $r > 0$ such that $\lambda_T(r) \le 1/ (2\pi)$. Let $f \in H_{r,T}$ satisfy $\|f\|_{L^2(\gamma_T)}=1$ and $\|f'\|_{L^2(\gamma_T)}^2 \le 1/(2\pi)$; indeed, the existence of an optimal $f$ in \eqref{def:EvalProblem} follows from standard Sturm--Liouville theory. Define $\mu \ll \gamma_T$ by $d\mu=f^2\,d\gamma_T$. Consider the F\"ollmer process constructed as in Lemma \ref{lem.follmer}. Precisely, set $v(t,x) = \log P_{T-t} f^2(x)$, where $P_s$ is again the heat semigroup. As explained in Lemma \ref{lem.follmer}, the SDE
\begin{align*}
dX(t) = \partial_x v(t,X(t))\,dt + dB(t), \quad X(0) = 0,
\end{align*}
admits a pathwise unique strong solution, and it satisfies $X(T) \sim \mu$. Moreover, the drift $\partial_xv(t,X(t))$ is a martingale, and the submartingale $|\partial_xv(t,X(t))|^2$ satisfies
\[
\E|\partial_xv(t,X(t))|^2 \le \E|\partial_xv(T,X(T))|^2 = \int_\R |(\log f^2)'|^2\,d\mu  = 4\int_\R (f')^2\,d\gamma_T \le \frac{2}{\pi}.
\]
This provides an admissible control and thus shows that $V_T^{\infty}\le r$.
\end{proof}

Notice that there is a simple scaling relation between $T$ and $r$: For $s > 0$, since the pushforward of $\gamma_T$ by $x\mapsto \sqrt{s} x$ is $\gamma_{sT}$, by setting $g(x)=f(\sqrt{s}x)$ we have
\begin{align*}
\lambda_T(r) &= \inf\big\{ \|f'\|_{L^2(\gamma_T)}^2 : f \in H_{r,T}, \ \|f\|_{L^2(\gamma_T)}=1 \big\} \\
    &= \inf\big\{ s\|g'\|_{L^2(\gamma_{sT})}^2 : g \in H_{r\sqrt{s},Ts}, \ \|g\|_{L^2(\gamma_{sT})}=1 \big\}  \\
    &= s\lambda_{sT}(r\sqrt{s}).
\end{align*}
In particular, $\lambda_{T}(r)=T^{-1}\lambda_1(rT^{-1/2})$, and so $\lambda_T^{-1}(1/(2\pi)) = \sqrt{T}\lambda_1^{-1}(T/(2\pi))$. We will use this scaling relationship to show that this strategy gives an asymptotically tight upper bound in the $T \to 0$ limit.

\begin{proposition} \label{prop.smalltimeasymptotics}
    With $\lambda_T(\cdot)$ defined as above, we have $\lambda_T^{-1}(1/ (2 \pi)) \sim \sqrt{2T \log(1/T)}$ as $T \to 0$.
\end{proposition}
\begin{proof}
    By the above scaling relation, it suffices to show that $\lambda_1^{-1}(T/ (2 \pi)) \sim \sqrt{2 \log(1/T)}$ as $T \to 0$.
    This will follow if we can show that
    \begin{equation}\label{eq:asymp:lambda}
        \lambda_1(r) = e^{-(1+o(1))r^2/2}
    \end{equation}
    as $r \to \infty$.
    Fix $r$ large, and define for $|x| \leq r$
    \begin{equation}
        f_r(x) \coloneqq e^{r^2/2} - e^{x^2/2}\,.
    \end{equation}
    This function lies in $H_{r,1}$, and satisfies
    \begin{align*}
        \|f\|_{L^2(\gamma)} & = e^{(1+o(1))r^2/2} \\
        \|f'\|_{L^2(\gamma)}^2 & =  e^{(1+o(1))r^2/2}\,.
    \end{align*}
    Renormalizing, we obtain an upper bound on $\lambda_1(r)$, yielding one direction of \eqref{eq:asymp:lambda}.

    To obtain the other direction, we employ Barta's method.
    Note that with $L=L_1$
    \begin{equation}\label{eq:asymp:barta}
        (- L f_r)(x) = e^{x^2/2} \geq e^{-r^2/2} f_r(x) \quad \quad \forall |x| \leq r\,.
    \end{equation}
    Let $g$ be optimal in \eqref{def:EvalProblem}.
    Sturm--Liouville theory guarantees that $g$ is positive on $(-r, r)$, which, combined with~\eqref{eq:asymp:barta} yields
    \begin{equation*}
        0 \leq \int g ((-L f_r) - e^{-r^2/2} f_r) d \gamma = (\lambda_1(r) - e^{-r^2/2}) \int g f_r d \gamma = (\lambda_1(r) - e^{-r^2/2}) c
    \end{equation*}
    for some $c > 0$.
    We obtain the lower bound $\lambda_1(r) \geq e^{-r^2/2}$, proving the claim.
    \end{proof}

\subsection{A lower bound: Relaxing the constraint}

Simple information-theoretic considerations lead to the following lower bound, which is tight for small $T$:

\begin{proposition} \label{prop.firstlowerbound}
With $\Phi$ denoting the cumulative distribution function of the standard Gaussian, we have
\[
V_T^{\infty} \ge \sqrt{T}\,\Phi^{-1}\Big(\frac{1+e^{-T/\pi}}{2}\Big).
\]
\end{proposition}
\begin{proof}
Fix an admissible control $\alpha$, and let $\mathbb{P}$ denote the law on path space $C[0,T]$ of the corresponding process $X$. Let $\mathbb{W}$ denote Wiener measure. Let $\mu$ denote the law of $X_T$. By the data processing inequality and Girsanov's theorem, we have the relative entropy bound \cite[Proposition 1]{lehec2013representation}
\[
    H(\mu\,|\,\gamma_T) \le H(\mathbb{P}\,|\,\mathbb{W}) \le \frac12 \, \E\int_0^T \alpha(t)^2 \, dt \le \frac{T}{\pi}.
    \]
On the other hand, the infimum of $H(\mu\,|\,\gamma_T)$ over all measures $\mu$ supported on $[-r,r]$ is exactly $-\log \gamma_T([-r,r])$, for each $r>0$; indeed, the unique minimizer is just $\gamma_T$ conditioned on $[-r,r]$. If $r > V_T^{\infty}$, then there exists a control achieving this support $[-r,r]$ for $\mu$, and thus
\[
\frac{T}{\pi} \ge -\log \gamma_T([-r,r]).
\]
Noting that $\gamma_T[-r,r]=2\Phi(r/\sqrt{T})-1$, we deduce
\[
r \ge \sqrt{T}\,\Phi^{-1}\Big(\frac{1+e^{-T/\pi}}{2}\Big).
\]
As this holds for every $r > V_T^{\infty}$, the proof is complete.
\end{proof}

Observe from Lemma \ref{lem.nondecreasing} that $T \mapsto V_T^{\infty}$ is nondecreasing. Thus, for all $T > 0$, Proposition \ref{prop.firstlowerbound} implies
\begin{equation}\label{eq.phiinverse.monotonicity}
    V_T^{\infty} \ge \sup_{0 < t < T}\sqrt{t}\,\Phi^{-1}\Big(\frac{1+e^{-t/\pi}}{2} \Big).
\end{equation}

\begin{remark} \label{rmk.phiinverse.taylor}
At leading order, as $T\to 0$, note that the lower bound of Proposition \ref{prop.firstlowerbound} is
\[
\sqrt{T}\,\Phi^{-1}\Big(\frac{1+e^{-T/\pi}}{2}\Big) \sim \sqrt{2T\log(1/T)}.
\]
\end{remark}

The lower bound of Proposition \ref{prop.firstlowerbound} is precisely what arises by relaxing the constraint on the control, from $\sup_{t \in [0,T]}\E[\alpha(t)^2] \le 2/\pi$ to $\int_0^T \E[\alpha(t)^2]\,dt \le 2T/\pi$. Indeed, the latter constraint, by a standard Girsanov calculation, is equivalent to $H(\mathbb{P}\,|\,\mathbb{W}) \le T/\pi$. Hence, recalling the definition of $V^\infty_T$ from \eqref{def.vTinf}, we have
\begin{align*}
    V_T^{\infty} &\ge \inf \bigg\{ \Big\| \int_0^T \alpha(t)\, dt + B(T) \Big\|_\infty : \int_0^T \E[\alpha(t)^2]\,dt \leq \frac{2T}{\pi}\bigg\} \\
        &= \inf\big\{ [\mathbb{P}_T]_\infty : \mathbb{P} \in \cP(C[0,T]), \  H(\mathbb{P}\,|\,\mathbb{W}) \le T/\pi\big\},
\end{align*}
where we recall that $[m]_\infty$ for $m \in \cP(\R)$ was defined in \eqref{mfc}, and here we write $\mathbb{P}_T$ for the time-$T$ marginal of the path measure $\mathbb{P}$. The right-hand side further simplifies to
\begin{equation}
\inf\big\{ [\mu]_\infty : \mu \in \cP(\R), \  H(\mu\,|\,\gamma_T) \le T/\pi\big\}, \label{def:entropyminimization}
\end{equation}
by the data processing inequality; equality holds because any $\mu \ll\gamma_T$ induces a path measure $\mathbb{P}(d\omega) \coloneqq \frac{d\mu}{d\gamma_T}(\omega(T))\mathbb{W}(d\omega)$ satisfying $\mathbb{P}_T=\mu$ and $H(\mu\,|\,\gamma_T)=H(\mathbb{P}\,|\,\mathbb{W})$. (In fact, $\mathbb{P}$ is precisely the F\"ollmer process associated with $\mu$, \cite{lehec2013representation}.) Finally, the proof of Proposition \ref{prop.firstlowerbound} shows that \eqref{def:entropyminimization} is precisely $\sqrt{T}\,\Phi^{-1}\big((1 + e^{-T/\pi}) / 2 \big)$.
This interpretation of Proposition \ref{prop.firstlowerbound} explains why it is sharp for small $T$; the relaxation of the constraint from a supremum over $[0,T]$ to an integral has less of an effect for a small horizon $T$.

\section{An explicitly solvable \texorpdfstring{$L^2$}{L2} variant of the limiting problem}
\label{sec:l2-problem}

In this section, we explicitly solve an $L^2$ version of the limiting problem $V_T^\infty$ where the $L^\infty$ norm in the objective is replaced by the $L^2$ norm. Specifically, fix $T>0$ and consider
\begin{equation*}
    V_T^2 \coloneqq \inf_{\alpha} \|X(T)\|_2,\qquad X(t) = \int_0^t \alpha(s)\,ds + B(t),\quad 0 \le t \le T,
\end{equation*}
where the infimum is taken over all progressively measurable processes $\alpha$ satisfying
\begin{equation*}
    \|\alpha(t)\|_2 \le c_0 \coloneqq \sqrt{\frac{2}{\pi}},\qquad \text{for all } 0 \le t \le T.
\end{equation*}
This control problem naturally arises as the scaling limit of the $\ell^2$-version of the discrete problem $V^{n,m}$ with $m/n\to T$; see Remark~\ref{rmk.l2.scalinglimit}.

We showed in Corollary~\ref{cor.duality.fbsde} that $V_T^2$ admits a unique optimal control $\alpha$, characterized by the solution of a forward-backward stochastic differential equation.\footnote{The statement of Corollary~\ref{cor.duality.fbsde} is given for $T=1$, but the proof extends verbatim to arbitrary $T>0$.} We now solve this system explicitly.

\begin{proposition}\label{prop.l2control.lowerbound}
    For $T>0$, the forward-backward stochastic differential equation
    \begin{equation}\label{eq.l2control.fbsde}
        \left\{
            \begin{alignedat}{2}
                dX(t) &= -c_0 \frac{Y(t)}{\|Y(t)\|_{2}}\,dt + dB(t),\qquad&& X(0) = 0,\\
                dY(t) &= Z(t)\, dB(t),\qquad&& Y(T) = \frac{X(T)}{\|X(T)\|_2},
                \qquad \smash{\raisebox{\normalbaselineskip}{$0 \le t \le T$,}}
            \end{alignedat}
        \right.
    \end{equation}
    has a unique solution $(X,Y,Z)$, given by
    \begin{equation}\label{eq.l2control.solution.explicit}
        X(t) = \int_0^t \frac{w(t)}{w(s)}\,dB(s),\quad Y(t) = \int_0^t Z(s)\,dB(s),\quad Z(t) = \frac{2c_0 w(T)}{w(t)(1-w(T))},
    \end{equation}
    for $0 \le t \le T$, where
    \begin{equation*}
        w(t) \coloneqq -W_0\big(-e^{-2c_0^2 t - 1}\big) \in (0,1].
    \end{equation*}
    Here, $W_0$ denotes the principal branch of the Lambert $W$-function. Moreover, for $0 < t \leq T$,
    \begin{equation}\label{eq.l2control.solution.relations}
        \|X(t)\|_2 = \frac{1-w(t)}{2c_0},\qquad \frac{Y(t)}{\|Y(t)\|_2} = \frac{X(t)}{\|X(t)\|_2}.
    \end{equation}
    Consequently, the process
    \begin{equation}\label{eq.l2control.solution.control}
        \alpha(t) \coloneqq -c_0 \frac{X(t)}{\|X(t)\|_2},\qquad 0 < t \le T.
    \end{equation}
    is the unique optimal control for $V_T^2$.
\end{proposition}

\begin{remark} \label{rmk.l2.longtime}
    Because the $L^2$ norm bounds the $L^\infty$ norm from below, we obtain the alternative lower bound
    \begin{equation*}
        V_T^\infty \geq V_T^2 = \frac{1 - w(T)}{2c_0}.
    \end{equation*}
    However, one can see that this bound is strictly worse than that of \eqref{eq.phiinverse.monotonicity} for all times $T>0$.
\end{remark}

\begin{remark} \label{rmk.l2.scalinglimit}
    As mentioned above, if $n\to\infty$ and $m/n\to T$, the control problem $V_T^2$ is the scaling limit of the $\ell^2$-version of $V^{n,m}$, in which the $\ell^\infty$-objective is replaced by the Euclidean norm. The lower bound is shown exactly as in Section~\ref{sec:lower-bound}. The proof of the upper bound is significantly simpler than in the $\ell^\infty$ setting, especially because the optimal control $(\alpha(t))_{0 \leq t \leq T}$ of $V_T^2$ is bounded in $L^p$ for arbitrary $p>0$.\footnote{In fact, from \eqref{eq.l2control.solution.explicit} and \eqref{eq.l2control.solution.control}, we know that $\alpha(t)\sim \cN(0,c_0^2)$ for every $t>0$.} Following the strategy used for Proposition~\ref{prop.vnupperbound}, one first takes a system $\bX$ of $n$ independent copies of the continuous state process, driven by the optimal control $\alpha$, and then invokes Lemma~\ref{lem.coupling.discrete} in a suitable way to construct Gaussian increments $\bm{Z}$ and an admissible choice $\eps$. One can then show that
    \begin{equation*}
        \frac{1}{\sqrt{n}} \E|\bY(m) - \bX(T)|_2 = o_n(1),
    \end{equation*}
    by following similar estimates as in the proof of Proposition~\ref{prop.vnupperbound}.
\end{remark}

Before proving the proposition, let us recall the Lambert $W$-function and some basic properties. The Lambert $W$-function, also called the product logarithm, is the multivalued inverse of the map $w \mapsto we^w$ on $\mathbb{C}$. On the real line, the function $f(w)=we^w$ is strictly decreasing on $(-\infty,-1)$, strictly increasing on $(-1,\infty)$, and attains its minimum value $f(-1)=-e^{-1}$. Accordingly, it has two real branches: the principal branch $W_0:[-e^{-1},\infty)\to[-1,\infty)$ and the lower branch $W_{-1}:[-e^{-1},0)\to(-\infty,-1]$. Differentiating the identity $W_0(x)e^{W_0(x)}=x$, we obtain
\begin{equation}\label{eq:lambert-ode}
    W_0'(x)=\frac{W_0(x)}{x(1+W_0(x))},\qquad x\notin\Big\{0,-\frac1e\Big\}.
\end{equation}

\begin{proof}[Proof of Proposition~\ref{prop.l2control.lowerbound}]
    By Corollary~\ref{cor.duality.fbsde}, the forward-backward system \eqref{eq.l2control.fbsde} has a unique solution $(X,Y,Z)$, and the corresponding drift $\alpha(t)=-c_0 Y(t)/\|Y(t)\|_2$ is the unique optimal control for $V_T^2$. It therefore remains only to verify that the triple $(X,Y,Z)$ defined in \eqref{eq.l2control.solution.explicit} satisfies \eqref{eq.l2control.fbsde} and \eqref{eq.l2control.solution.relations}.

    First, $X(t)$ is Gaussian with mean zero and variance
    \begin{equation*}
        \E\,X(t)^2 = \int_0^t \frac{w(t)^2}{w(s)^2}\,ds.
    \end{equation*}
    We compute this integral as follows. Differentiating $w(t) = -W_0(x(t))$ with $x(t)=-e^{-2c_0^2 t - 1}$ and using \eqref{eq:lambert-ode}, we obtain
    \begin{equation*}
        w'(t) = -x'(t) W_0'(x(t)) = 2c_0^2\,x(t)W_0'(x(t)) = -2c_0^2\,\frac{w(t)}{1-w(t)}.
    \end{equation*}
    Hence, using the change of variables $u=w(s)$ and the fact that $w(0)=1$, we get
    \begin{align*}
        \int_0^t \frac{1}{w(s)^2}\,ds
        = -\frac{1}{2c_0^2}\int_{w(0)}^{w(t)} \frac{1-u}{u^3}\,du
        = \frac{1}{2c_0^2}\bigg[\frac{1-2u}{2u^2}\bigg]_{u=w(0)}^{u=w(t)}
        = \frac{(1-w(t))^2}{4c_0^2 w(t)^2}.
    \end{align*}
    Rearranging and taking square roots yields
    \begin{equation*}
        \|X(t)\|_2
        = \sqrt{\int_0^t \frac{w(t)^2}{w(s)^2}\,ds}
        = \frac{1-w(t)}{2c_0}
        = -c_0 \frac{w(t)}{w'(t)},
    \end{equation*}
    which proves the first identity in \eqref{eq.l2control.solution.relations}. Applying It\^o's product rule to $X(t)$, we then find
    \begin{align}\label{eq.l2control.xdynamics}
        dX(t)
        = \frac{w'(t)}{w(t)}X(t)\,dt + dB(t)
        = -c_0 \frac{X(t)}{\|X(t)\|_2}\,dt + dB(t).
    \end{align}
    Moreover, by the definitions of $Y$ and $Z$,
    \begin{equation*}
        Y(t)
        = 2 c_0 \frac{w(T)}{1-w(T)} \int_0^t \frac{1}{w(s)}\,dB(s)
        = 2 c_0 \frac{w(T)}{(1-w(T))w(t)} X(t).
    \end{equation*}
    The prefactor is positive and deterministic, so normalizing both sides in $L^2$ gives
    \begin{equation*}
        \frac{Y(t)}{\|Y(t)\|_2} = \frac{X(t)}{\|X(t)\|_2}.
    \end{equation*}
    This proves the second identity in \eqref{eq.l2control.solution.relations}, and together with \eqref{eq.l2control.xdynamics} shows that the triple $(X,Y,Z)$ satisfies the dynamics in \eqref{eq.l2control.fbsde}. It remains only to check the terminal condition for $Y$, but this follows from
    \begin{equation*}
        Y(T) = \frac{2 c_0}{1-w(T)} X(T) = \frac{X(T)}{\|X(T)\|_2}. \qedhere
    \end{equation*}
\end{proof}

\appendix

\section{Miscellaneous technical results and missing proofs}
\label{sec:appendix-misc}

\begin{lemma}\label{lemma:compactness-K}
    For any finite constant $c\geq 0$, the set
    \begin{equation*}
		\mathcal{K}_c \coloneqq \bigg\{\mu\in\P(\C): \mu(H_0^1)=1,\ \esssup_{t\in[0,1]} \int_{H_0^1} (\dot{\phi}(t))^2\,\mu(d\phi) \leq c\bigg\}.
	\end{equation*}
    is compact in $\P(\C)$.
\end{lemma}

\begin{proof}
    For any non-negative measurable function $h:[0,1]\to[0,\infty)$
	\begin{equation*}
		\esssup_{t\in[0,1]} h(t) = \sup_{\substack{g\in L^1[0,1]:\\ \|g\|_{L^1}\leq 1}} \int_0^1 g(t) h(t)\,dt,
	\end{equation*}
	and by the separability of $L^1[0,1]$, we can restrict the supremum on the right-hand side to a family of non-negative functions $\{g^{(n)}:n\in\N\}\subset L^1[0,1]$ with $\|g^{(n)}\|_{L^1}\leq 1$, independently of $h$. In particular, we can rewrite $\mathcal{K}_c$ as
	\begin{equation*}
		\mathcal{K}_c = \bigcap_{n\in\N} \mathcal{K}^{(n)},
	\end{equation*}
	where
	\begin{equation*}
		\mathcal{K}^{(n)} = \bigg\{\mu\in\P(\C): \mu(H_0^1)=1,\ \int_0^1 g^{(n)}(t) \bigg(\int_{H_0^1} (\dot{\phi}(t))^2\,\mu(d\phi)\bigg) \, dt \leq c\bigg\}.
	\end{equation*}
    Note here that a version of $[0,1]\times\C\ni(t,\phi)\mapsto \dot{\phi}(t)\mathbf{1}_{H_0^1}(\phi)$ can be defined in a Borel measurable way, e.g., as the limit when it exists of difference quotients.
	We can assume without loss of generality that $g^{(1)}(t)\equiv 1$ so that
	\begin{equation*}
		\mathcal{K}^{(1)} = \bigg\{\mu\in\P(\C) : \int_{\C} G^{(1)}(\phi)\,\mu(d\phi) \leq c\bigg\},
	\end{equation*}
	where
	\begin{equation*}
		G^{(1)}(\phi) =
		\begin{cases}
			\int_0^1 (\dot{\phi}(t))^2\,dt, &\text{if } \phi\in H_0^1,\\
			\infty, &\text{if }\phi\notin H_0^1.
		\end{cases}
	\end{equation*}
    It is a classical consequence of the Arzel\`a--Ascoli theorem that the sublevel sets of the function $G^{(1)}:\C\to[0,\infty]$ are compact, and it follows that  sublevel sets of the function $\mu\mapsto\int G^{(1)}\,d\mu$ are compact in $\P(\C)$ (see \cite[Theorem 2.10]{budhiraja-dupuis}). It is straightforward to check that $\mathcal{K}^{(1)} \cap \mathcal{K}^{(n)}$ is closed and thus compact for each $n$. It follows that $\mathcal{K}_c$ is compact as a countable intersection of the compact sets $\mathcal{K}^{(1)}\cap \mathcal{K}^{(n)}$, $n\in\N$.
\end{proof}

\begin{proof}[Proof of Lemma~\ref{lemma:norm-concentration}]
	Denote $\sigma^2\coloneqq\E\,U_1^2$. The statement is trivially true when $\sigma^2=0$. So suppose that $\sigma^2\in(0,1]$. Noting that $\E|\bm{U}|_2^2=\sigma^2 n$, we estimate
	\begin{align*}
		\E\big[\big(|\bm{U}|_2-\sqrt{n}\big)_+^2\big]
		\leq \E\big[\big(|\bm{U}|_2-\sigma \sqrt{n}\big)^2\big]
		&\leq \frac{1}{\sigma^2 n} \E\big[\big(|\bm{U}|_2-\sigma \sqrt{n}\big)^2\big(|\bm{U}|_2+\sigma \sqrt{n}\big)^2\big] \\
		&= \frac{1}{\sigma^2 n} \Var(|\bm{U}|_2^2) = \frac{1}{\sigma^2} \Var(U_1^2).
	\end{align*}
	The last equality holds by independence. Note that this blows up when $\sigma^2$ is small. For $\sigma^2$ close to zero, we instead estimate
	\begin{align*}
		\E\big[\big(|\bm{U}|_2-\sqrt{n}\big)_+^2\big]
		&\leq \E\big[\big(|\bm{U}|_2^2 - n\big)_+\big] \\
		&= \int_0^\infty \PP\big[|\bm{U}|_2^2 > n + t\big]\,dt \\
		&= \int_{(1-\sigma^2)n}^\infty \PP\big[|\bm{U}|_2^2 -\sigma^2 n > s\big]\,ds \\
		&\leq \int_{(1-\sigma^2)n}^\infty \frac{1}{s^2} \Var(|\bm{U}|_2^2)\,ds \\
		&= \frac{1}{(1-\sigma^2)n} \Var(|\bm{U}|_2^2) = \frac{1}{1-\sigma^2} \Var(U_1^2)
	\end{align*}
	We have thus shown that $\E\big[\big(|\bm{U}|_2-\sqrt{n}\big)_+^2\big] \leq (\frac{1}{\sigma^2}\wedge \frac{1}{1-\sigma^2})\Var(U_1^2)$. The factor in front of $\Var(U_1^2)$ is at most 2 (when $\sigma^2=1/2$), and this gives the desired bound.
\end{proof}

\section{The Brownian motion in the coupling construction}
\label{sec:appendix:bm-construction}

This section is devoted to proving that the auxiliary process $\bm{W}$ constructed in \eqref{eq:upper-bound:bm-construction} in Lemma~\ref{lem.coupling.discrete} is indeed a Brownian motion and that $\F^{\bm{W}}(k/n) = \F^{\bm{B}}(k/n)$ for all $k=k_0,\dots,n$. The main idea is to construct $\bm{W}$ from $\bm{B}$ by iteratively applying the following construction on each subinterval $[k/n,(k+1)/n]$.

Given a probability space $(\Omega,\F,\PP)$, let $(\bm{B}(t))_{t\ge 0}$ be a Brownian motion in $\R^n$, and let $(\F^{\bm B}(t))_{t\ge 0}$ denote its canonical filtration. Fix a unit vector $\bm v\in\R^n$, and define the orthogonal projection
\begin{equation*}
    \bm P \coloneqq I_n-\bm v\bm v^\top,
\end{equation*}
where $I_n$ is the identity matrix. Then the processes $t\mapsto \langle \bm v,\bm B(t)\rangle$ and $t\mapsto \bm P\bm B(t)$ are independent Brownian motions in $\R$ and $\bm v^\perp\coloneqq\{\bm y\in\R^n:\langle \bm v,\bm y\rangle=0\}$, respectively. We write their canonical filtrations as
\begin{equation*}
    \F^{\bm B,\bm v}(t)\coloneqq \sigma\big(\langle \bm v,\bm B(u)\rangle:u\le t\big),
    \qquad
    \F^{\bm B,\perp}(t)\coloneqq \sigma\big(\bm P\bm B(u):u\le t\big),
    \qquad t\ge 0.
\end{equation*}
In particular, the filtrations $(\F^{\bm B,\bm v}(t))_{t\ge 0}$ and $(\F^{\bm B,\perp}(t))_{t\ge 0}$ are independent, and
\begin{equation}\label{eq:bm-construction:filtration-decomposition}
    \F^{\bm B}(t)=\F^{\bm B,\bm v}(t)\vee \F^{\bm B,\perp}(t),
    \qquad t\ge 0.
\end{equation}

Fix $T>0$, let $\eps$ be an $\F^{\bm B,\bm v}(T)$-measurable random variable with values in $\{\pm1\}$, and define an $\R^n$-valued process $(\bm W(t))_{t\in[0,T]}$ by
\begin{equation*}
    \bm W(t)\coloneqq \langle \bm v,\bm B(t)\rangle \bm v+\eps\,\bm P\bm B(t),
    \qquad t\in[0,T].
\end{equation*}
We define the filtrations $(\F^{\bm W,\bm v}(t))_{t\in[0,T]}$ and $(\F^{\bm W,\perp}(t))_{t\in[0,T]}$ for $\bm W$ analogously.

\begin{lemma}\label{lemma:bm-construction-step}
    The following statements hold:
    \begin{enumerate}
        \item $\F^{\bm W,\bm v}(t)=\F^{\bm B,\bm v}(t)$ for all $t\in[0,T]$.
        \label{lemma:bm-construction-step:1}

        \item $\F^{\bm W,\perp}(t)\vee \sigma(\eps)=\F^{\bm B,\perp}(t)\vee \sigma(\eps)$ for all $t\in[0,T]$.
        \label{lemma:bm-construction-step:2}

        \item $\F^{\bm W}(T)=\F^{\bm B}(T)$.
        \label{lemma:bm-construction-step:3}

        \item $(\bm W(t))_{t\in[0,T]}$ is a Brownian motion in $\R^n$.
        \label{lemma:bm-construction-step:4}
    \end{enumerate}
\end{lemma}

\begin{proof}
    The first two statements are immediate from the two identities
    \begin{equation}\label{eq:bm-construction:b-w-identities}
        \langle \bm v,\bm W(t)\rangle=\langle \bm v,\bm B(t)\rangle,\qquad
        \bm P\bm W(t)=\eps\,\bm P\bm B(t),\qquad 0\leq t \leq T.
    \end{equation}
    Statement \eqref{lemma:bm-construction-step:3} follows from the $\F^{\bm B,\bm v}(T)$-measurability of $\eps$ and the relation \eqref{eq:bm-construction:filtration-decomposition} for $t=T$.

    It remains to prove \eqref{lemma:bm-construction-step:4}. By the first identity in \eqref{eq:bm-construction:b-w-identities}, the process $t\mapsto \langle \bm v,\bm W(t)\rangle$ is a one-dimensional Brownian motion. Moreover, by the second identity, the $\F^{\bm B,\bm v}(T)$-measurability of $\eps$, and the independence of $\F^{\bm B,\bm v}(T)$ and $\F^{\bm B,\perp}(T)$, the conditional law of $t\mapsto \bm P\bm W(t)$ given $\F^{\bm B,\bm v}(T)$ coincides with that of $t\mapsto \bm P\bm B(t)$. Hence $t\mapsto \bm P\bm W(t)$ is a Brownian motion in $\bm v^\perp$, independent of $\F^{\bm W,\bm v}(T)=\F^{\bm B,\bm v}(T)$. Therefore, $\bm W$ is a Brownian motion in $\R^n$.
\end{proof}

We now return to the process $\bm W$ from \eqref{eq:upper-bound:bm-construction}. Recall that for each $k=k_0,\dots,n-1$, we are given a random unit vector $\bm v(k)$ and a $\{\pm1\}$-valued random variable $\eps(k+1)$ such that
\begin{gather*}
    \bm v(k)\text{ is }\F^{\bm B}(\tfrac{k}{n})\text{-measurable},\\
    \eps(k+1)\text{ is measurable w.r.t.\ }
    \F^{\bm B}(\tfrac{k}{n})\vee
    \sigma\Big(\big\langle \bm v(k),\bm B(t)-\bm B(\tfrac{k}{n})\big\rangle:\tfrac{k}{n} < t\le \tfrac{k+1}{n}\Big).
\end{gather*}
Write $t_0=k_0/n$. Starting from $\bm W(t_0)=\bm B(t_0)$, we define $(\bm W(t))_{t_0\le t\le 1}$ inductively by
\begin{align*}
    \bm W(t)
    &\coloneqq \bm W(\tfrac{k}{n})+\big\langle \bm v(k),\bm B(t)-\bm B(\tfrac{k}{n})\big\rangle \bm v(k) \\
    &\qquad\quad + \eps(k+1)\bm P(k)\big(\bm B(t)-\bm B(\tfrac{k}{n})\big),
    \qquad \tfrac{k}{n}<t\le \tfrac{k+1}{n},
\end{align*}
where
\begin{equation*}
    \bm P(k)\coloneqq I_n-\bm v(k)\bm v(k)^\top.
\end{equation*}
Let $(\F^{\bm W}(t))_{t_0\le t\le 1}$ be the filtration generated by $\bm W$.

\begin{lemma}\label{lemma:bm-construction-iterated}
    The process $(\bm W(t)-\bm W(t_0))_{t_0\le t\le 1}$ is a Brownian motion in $\R^n$, and
    \begin{equation}\label{eq:bm-construction-iterated:filtrations}
        \F^{\bm W}(\tfrac{k}{n})=\F^{\bm B}(\tfrac{k}{n}),\qquad k=k_0,\dots,n.
    \end{equation}
\end{lemma}

\begin{proof}
    We argue by induction on $k=k_0,\dots,n$ that \eqref{eq:bm-construction-iterated:filtrations} holds
    and that the process $t\mapsto \bm W(t)-\bm W(t_0)$ is a Brownian motion on the interval $[t_0,k/n]$. The claim is clear for $k=k_0$.

    Now assume that the claim holds for some $k\in\{k_0,\dots,n-1\}$. Then
    \begin{equation*}
        \big(\bm B(t+\tfrac{k}{n})-\bm B(\tfrac{k}{n})\big)_{0\le t\le 1/n}
    \end{equation*}
    is a Brownian motion on $[0,1/n]$, independent of $\F^{\bm B}(k/n)$. Applying Lemma~\ref{lemma:bm-construction-step} conditionally on $\F^{\bm B}(k/n)$ with $T=1/n$, $\bm v=\bm v(k)$, and $\eps=\eps(k+1)$, we obtain that
    \begin{equation*}
        \big(\bm W(t+\tfrac{k}{n})-\bm W(\tfrac{k}{n})\big)_{0\le t\le 1/n}
    \end{equation*}
    is a Brownian motion on $[0,1/n]$, independent of $\F^{\bm B}(k/n)$, and that
    \begin{equation*}
        \sigma\Big(\bm W(t+\tfrac{k}{n})-\bm W(\tfrac{k}{n}):0\le t\le \tfrac{1}{n}\Big)
        =
        \sigma\Big(\bm B(t+\tfrac{k}{n})-\bm B(\tfrac{k}{n}):0\le t\le \tfrac{1}{n}\Big).
    \end{equation*}
    Since $\F^{\bm B}(k/n)=\F^{\bm W}(k/n)$ by the induction hypothesis, it follows that
    \begin{equation*}
        \F^{\bm W}(\tfrac{k+1}{n})=\F^{\bm B}(\tfrac{k+1}{n}),
    \end{equation*}
    and $t\mapsto \bm W(t)-\bm W(t_0)$ is a Brownian motion on $[t_0, (k+1)/n]$. This completes the induction.
\end{proof}

\clearpage
\printbibliography

\end{document}

\typeout{get arXiv to do 4 passes: Label(s) may have changed. Rerun}